\documentclass[a4paper, 11pt, reqno]{amsart}
\usepackage{mathrsfs}
\usepackage{amsfonts}
\usepackage[centertags]{amsmath}
\usepackage{amssymb}
\usepackage{amsthm}
\usepackage{float}
\usepackage{graphicx}
\usepackage{caption}
\usepackage{hyperref}
\usepackage{enumerate}
\usepackage[textwidth=16cm, hmarginratio=1:1]{geometry}
\usepackage{appendix}
\usepackage[all]{xy}
\usepackage{extarrows}
\usepackage{resizegather}

\newtheorem{theorem}{Theorem}[section]

\newtheorem{lemma}[theorem]{Lemma}

\newtheorem{definition}[theorem]{Definition}
\newtheorem{remark}[theorem]{Remark}
\newtheorem{example}[theorem]{Example}

\newtheorem{conjecture}[theorem]{Conjecture}

\numberwithin{equation}{section}

\DeclareMathOperator*{\rep}{Rep}
\DeclareMathOperator*{\diag}{diag}
\DeclareMathOperator*{\res}{Res}

\begin{document}

\title{On the minimal affinizations of type $F_4$}
\author{Bing Duan, Jian-Rong Li, Yan-Feng Luo}
\address{Bing Duan: School of Mathematics and Statistics, Lanzhou University, Lanzhou 730000, P. R. China.}
\email{duan890818@163.com}

\address{Jian-Rong Li: School of Mathematics and Statistics, Lanzhou University, Lanzhou 730000, P. R. China.}
\email{lijr@lzu.edu.cn}

\address{Yan-Feng Luo: School of Mathematics and Statistics, Lanzhou University, Lanzhou 730000, P. R. China.}
\email{luoyf@lzu.edu.cn}
\date{}

\maketitle

\begin{abstract}
In this paper, we apply the theory of cluster algebras to study minimal affinizations over the quantum affine algebra of type $F_4$. We show that the $q$-characters of a large family of minimal affinizations of type $F_4$ satisfy some equations. Moreover, every minimal affinization these equations corresponds to some cluster variable in some cluster algebra $\mathscr{A}$. For the other minimal affinizations of type $F_4$ which are not these equations, we give some conjectural equations which contains these minimal affinizations. Furthermore, we introduce the concept of dominant monomial graphs to study the equations satisfied by $q$-characters of modules of quantum affine algebras.

\hspace{0.15cm}

\noindent
{\bf Key words}: quantum affine algebra of type $F_4$; minimal affinizations; cluster algebras; $q$-characters; Frenkel-Mukhin algorithm

\hspace{0.15cm}

\noindent
{\bf 2010 Mathematics Subject Classification}: 17B37; 13B60
\end{abstract}

\section{Introduction}
The theory of cluster algebras are introduced by Fomin and Zelevinsky in \cite{FZ02}. It has many applications to mathematics and physics including quiver representations, Teichm\"{u}ller theory, tropical geometry, integrable systems, and Poisson geometry.

Let $\mathfrak{g}$ be a simple Lie algebra and $U_q \widehat{\mathfrak{g}}$ the corresponding quantum affine algebra.
In \cite{C95}, V. Chari and A. Pressley introduced minimal affinizations of representations of quantum groups. The family of minimal affinizations is an important family of simple modules which contains the Kirillov-Reshetikhin modules. Minimal affinizations are studied intensively in recent years, see for examples, \cite{CG11}, \cite{CMY13}, \cite{Her07}, \cite{Li15}, \cite{LM13}, \cite{M10}, \cite{MP07}, \cite{MP11}, \cite{MY12a}, \cite{MY12b}, \cite{MY14}, \cite{Nao13}, \cite{Nao14}, \cite{QL14}, \cite{SS14}, \cite{ZDLL15}. However, there is not much work of the minimal affinizations over the quantum affine algebra of type $F_{4}$ in the literature. The aim of this paper is to apply the theory of cluster algebras to study minimal affinizations over the quantum affine algebra of type $F_4$.

M-systems and dual M-systems of types $A_{n}$, $B_{n}$, $G_{2}$ are introduced in \cite{ZDLL15}, \cite{QL14} to study the minimal affinizations of types $A_{n}$, $B_{n}$, $G_{2}$. The equations in these systems are satisfied by the $q$-characters of minimal affinizations of types $A_{n}$, $B_{n}$, $G_{2}$. It is shown that every equation in these systems corresponds to a mutation equation in some cluster algebra.

The minimal affinizations of type $F_4$ are much more complicated than the minimal affinizations of types $A_{n}$, $B_{n}$, $G_{2}$. In types $A_{n}$, $B_{n}$, $G_{2}$, all minimal affinizations are special or anti-special, \cite{Her07}, \cite{LM13}. A $U_q \widehat{\mathfrak{g}}$-module $V$ is called special (resp. anti-special) if there is only one dominant (resp. anti-dominant) monomial in the $q$-character of $V$. In the case of type $F_4$, there are minimal affinizations which are neither special nor anti-special, \cite{Her07}.

In \cite{ZDLL15}, \cite{QL14}, the M-systems and dual M-systems of types $A_{n}$, $B_{n}$, $G_{2}$ contain all minimal affinizations and only contain minimal affinizations. The situation is different in the case of type $F_4$. It is quite possible that a closed system which contains all minimal affinizations and only contains minimal affinizations of type $F_4$ does not exist. However, we are able to find two closed systems which contain a large family of minimal affinizations of type $F_4$, Theorem \ref{M-system of type F4}, Theorem \ref{The dual M-system}. We show that the equations these systems are satisfied by the $q$-characters of the minimal affinizations in the systems. We prove that the modules in one system are special, Theorem \ref{special}, and the modules in the other system are anti-special, Theorem \ref{anti-special}. Moreover, we show that every equation in Theorem \ref{M-system of type F4} (resp. Theorem \ref{The dual M-system}) corresponds to a mutation equation in some cluster algebra $\mathscr{A}$ (resp. $\mathscr{\widetilde{A}}$). The cluster algebra $\mathscr{A}$ is the same as the cluster algebra for the quantum affine algebra of type $F_4$ introduced in \cite{HL13}. Moreover, every minimal affinization in Theorem \ref{M-system of type F4} (resp. Theorem \ref{The dual M-system}) corresponds to a cluster variable in the cluster algebra $\mathscr{A}$ (resp. $\mathscr{\widetilde{A}}$), Theorem \ref{connection with cluster algebra 1} (resp. Theorem \ref{minimal affinizations correspond to cluster variablesII}).

The system of equations in Theorem \ref{The dual M-system} is dual to the system in Theorem \ref{M-system of type F4}. This system contains the modules which are dual to the modules in the system in Theorem \ref{M-system of type F4}.

For the minimal affinizations which are not in Theorem \ref{M-system of type F4} and Theorem \ref{The dual M-system}, we give some conjectural equations which contain these modules, Conjecture \ref{the conjecture 1}.

We introduce the concept of dominant monomial graphs to study the equations satisfied by $q$-characters of modules of quantum affine algebras. We draw dominant monomial graphs for the modules in the equivalence classes of the left hand side of some equations in Conjecture \ref{the conjecture 1}. In these graphs, we find that every graph can be divided into two parts. The vertices in the first (resp. second) part of the graph are dominant monomials in the first (resp. second) summand of the right hand side of the correpsonding equation.

The paper is organized as follows. In Section \ref{Background}, we give some background information about cluster algebras and representation theory of quantum affine algebras. In Section \ref{main results}, we describe a closed system containing a large family of minimal affinizations of type $F_4$. In Section \ref{Relation between the M-systems and cluster algebras}, we study relations between the system in Theorem \ref{M-system of type F4} and cluster algebras. In Section \ref{dual M system}, we study the dual system of Theorem \ref{M-system of type F4}. In Section \ref{proof of special}, Section \ref{proof main1}, and Section \ref{proof irreducible}, we prove Theorem \ref{special}, Theorem \ref{M-system of type F4}, and Theorem \ref{irreducible} given in Section \ref{main results}, respectively. In Section \ref{conjectural equations for F_4}, we give a conjecture about the equations satisfied by the $q$-characters of the other minimal affinizations of type $F_4$ and introduce the concept of dominant monomial graphs to study the conjecture.

\section{Background}\label{Background}
\subsection{Cluster algebras}
We first recall the definition of cluster algebras introduced by Fomin and Zelevinsky in \cite{FZ02}. Let $\mathbb{Q}$ be the rational field and $\mathcal{F}=\mathbb{Q}(x_{1}, x_{2}, \cdots, x_{n})$ the field of rational functions in $n$ indeterminates over $\mathbb{Q}$.
A seed in $\mathcal{F}$ is a pair $\Sigma=({\bf y}, Q)$, where ${\bf y} = (y_{1}, y_{2}, \cdots, y_{n})$ is a free generating set of $\mathcal{F}$, and $Q$ is a quiver with vertices labeled by $\{1, 2, \cdots, n\}$. Assume that $Q$ has neither loops nor $2$-cycles. For $k\in \{1, 2, \cdots, n\}$, one defines a new seed $\mu_k({\bf y}, Q) = ({\bf y}', Q')$ by the mutation of $({\bf y}, Q)$ at $k$. Here ${\bf y}' = (y_1', \ldots, y_n')$, $y_{i}'=y_{i}$, for $i\neq k$, and
\begin{equation}
y_{k}'=\frac{\prod_{i\rightarrow k} y_{i}+\prod_{k\rightarrow j} y_{j}}{y_{k}}, \label{exchange relation}
\end{equation}
where the first (resp. second) product in the right hand side is over all arrows of $Q$ with target (resp. source) $k$, and $Q'$ is obtained from $Q$ by the follow rule:
\begin{enumerate}
\item[(i)] Reverse the orientations of all arrow incident with $k$;

\item[(ii)] Add a new arrow $i\rightarrow j$ for every existing pair of arrow $i\rightarrow k$ and $k\rightarrow j$;

\item[(iii)] Erasing every pair of opposite arrows possible created by (ii).
\end{enumerate}
The mutation class $\mathcal{C}(\Sigma)$ is the set of all seeds obtained from $\Sigma$ by a finite sequence of mutation $\mu_{k}$. If $\Sigma'=((y_{1}', y_{2}', \cdots, y_{n}'), Q')$ is a seed in $\mathcal{C}(\Sigma)$, then the subset $\{y_{1}', y_{2}', \cdots, y_{n}'\}$ is called a $cluster$, and its elements are called \textit{cluster variables}. The \textit{cluster algebra} $\mathscr{A}_{\Sigma}$ as the subring of $\mathcal{F}$ generated by all cluster variables. \textit{Cluster monomials} are monomials in the cluster variables supported on a single cluster.

In this paper, the initial seed in the cluster algebra we use is of the form $\Sigma=({\bf y}, Q)$, where ${\bf y}$ is an infinite set and $Q$ is an infinite quiver.

\begin{definition}[{Definition 3.1, \cite{GG14}}] \label{definition of cluster algebras of infinite rank}
Let $Q$ be a quiver without loops or $2$-cycles and with a countably infinite number of vertices labelled by all integers $i \in \mathbb{Z}$. Furthermore, for each vertex $i$ of $Q$ let the
number of arrows incident with $i$ be finite. Let ${\bf y} = \{y_i \mid i \in \mathbb{Z}\}$. An infinite initial seed is the pair $({\bf y}, Q)$. By finite sequences of mutation
at vertices of $Q$ and simultaneous mutation of the set ${\bf y}$ using the exchange relation (\ref{exchange relation}), one obtains a family of infinite seeds. The sets of variables in these seeds are called
the infinite clusters and their elements are called the cluster variables. The cluster algebra of
infinite rank of type $Q$ is the subalgebra of $\mathbb{Q}({\bf y})$ generated by the cluster variables.
\end{definition}

\subsection{The quantum affine algebra of type $F_4$}
In this paper, we take $\mathfrak{g}$ to be the complex simple Lie algebra of type $F_4$ and $\mathfrak{h}$ a Cartan subalgebra of $\mathfrak{g}$. Let $I=\{1, 2, 3, 4\}$.
We choose simple roots $\alpha_1, \alpha_2, \alpha_3, \alpha_4$ and scalar product $(\cdot, \cdot)$ such that
\begin{align*}
&( \alpha_1, \alpha_1 ) =2, \ ( \alpha_1, \alpha_2 )=-1,    \ ( \alpha_2, \alpha_2 )=2,  \ ( \alpha_2, \alpha_3 )=-2, \\
&( \alpha_3, \alpha_3 ) =4,  \ ( \alpha_3, \alpha_4 )=-2,    \ ( \alpha_4, \alpha_4 )=4.
\end{align*}
Therefore $\alpha_1, \alpha_2$ are the short simple roots and $\alpha_3, \alpha_4$ are the long simple roots.

Let $\{\alpha_1^{\vee}, \alpha_2^{\vee}, \alpha_3^{\vee}, \alpha_4^{\vee} \}$ and $\{\omega_1, \omega_2, \omega_3, \omega_4\}$ be the sets of simple coroots and fundamental weights respectively. Let $C=(C_{ij})_{i, j\in I}$ denote the Cartan matrix, where $C_{ij}=\frac{2 ( \alpha_i, \alpha_j ) }{( \alpha_i, \alpha_i )}$. Let $d_1=1, d_2=1, d_3=2, d_4=2$, $D=\diag (d_1, d_2, d_3, d_4)$ and $B=DC=(b_{ij})_{i,j \in I}$.
Then
\begin{align*}
C=
\left(
\begin{array}{cccc}
2 & -1 & 0 & 0\\
-1 & 2 & -2 & 0\\
0 & -1 & 2 & -1 \\
0 & 0 & -1 & 2
\end{array}
\right), \quad
B=
\left(
\begin{array}{cccc}
2 & -1 & 0 & 0\\
-1 & 2 & -2 & 0\\
0 & -2 & 4 & -2 \\
0 & 0 & -2 & 4
\end{array}
\right).
\end{align*}
Let $q_i=q^{d_i}$, $i\in I$. Let $Q$ (resp. $Q^+$) and $P$ (resp. $P^+$) denote the $\mathbb{Z}$-span (resp. $\mathbb{Z}_{\geq 0}$-span) of the simple roots and fundamental weights respectively. Let $\leq$ be the partial order on $P$ in which $\lambda \leq \lambda'$ if and only if $\lambda' - \lambda \in Q^+$.

Quantum groups are introduced independently by Jimbo \cite{Jim85} and Drinfeld \cite{Dri87}. Quantum affine algebras form a family of infinite-dimensional quantum groups. Let $\widehat{\mathfrak{g}}$ denote the untwisted affine algebra corresponding to $\mathfrak{g}$. In this paper, we fix a $q\in \mathbb{C}^{\times}$, not a root of unity. The quantum affine algebra $U_q\widehat{\mathfrak{g}}$ in Drinfeld's new realization, see \cite{Dri88}, is generated by $x_{i, n}^{\pm}$ ($i\in I, n\in \mathbb{Z}$), $k_i^{\pm 1}$ $(i\in I)$, $h_{i, n}$ ($i\in I, n\in \mathbb{Z}\backslash \{0\}$) and central elements $c^{\pm 1/2}$, subject to certain relations.

The algebra $U_q\mathfrak{g}$ is isomorphic to a subalgebra of $U_q\widehat{\mathfrak{g}}$. Therefore $U_q\widehat{\mathfrak{g}}$-modules restrict to $U_q\mathfrak{g}$-modules.

\subsection{Finite-dimensional $U_q \widehat{\mathfrak{g}}$-modules and their $q$-characters} We recall some known results on finite-dimensional  $U_q\widehat{\mathfrak{g}}$-modules and their $q$-characters, \cite{CP94}, \cite{CP95a}, \cite{FR98}, \cite{MY12a}.

Let $\mathcal{P}$ be the free abelian multiplicative group of monomials in infinitely many formal variables $(Y_{i, a})_{i\in I, a\in \mathbb{C}^{\times}}$. Then $\mathbb{Z}\mathcal{P} = \mathbb{Z}[Y_{i, a}^{\pm 1}]_{i\in I, a\in \mathbb{C}^{\times}}$. For each $j\in I$, a monomial $m=\prod_{i\in I, a\in \mathbb{C}^{\times}} Y_{i, a}^{u_{i, a}}$, where $u_{i, a}$ are some integers, is said to be \textit{$j$-dominant} (resp. \textit{$j$-anti-dominant}) if and only if $u_{j, a} \geq 0$ (resp. $u_{j, a} \leq 0$) for all $a\in \mathbb{C}^{\times}$. A monomial is called \textit{dominant} (resp. \textit{anti-dominant}) if and only if it is $j$-dominant (resp. $j$-anti-dominant) for all $j\in I$.

Every finite-dimensional simple $U_{q}\mathfrak{\widehat{g}}$-module is parametrized by a dominant monomial in $\mathcal{P}^+$, \cite{CP94}, \cite{CP95a}. That is, for a dominant monomial $m=\prod_{i\in I, a \in \mathbb{C}^{\times}}Y_{i,a}^{u_{i,a}}$, there is a corresponding simple $U_{q}\mathfrak{\widehat{g}}$-module $L(m)$. Let $\rep(U_q\widehat{\mathfrak{g}})$ be the Grothendieck ring of finite-dimensional $U_q\widehat{\mathfrak{g}}$-modules and $[V]\in \rep(U_q\widehat{\mathfrak{g}})$ the class of a finite-dimensional $U_q\widehat{\mathfrak{g}}$-module $V$.

The $q$-character of a $U_q\widehat{\mathfrak{g}}$-module $V$ is given by
\begin{align*}
\chi_q(V) = \sum_{m\in \mathcal{P}} \dim(V_{m}) m \in \mathbb{Z}\mathcal{P},
\end{align*}
where $V_m$ is the $l$-weight space with $l$-weight $m$, see \cite{FR98}. We use $\mathscr{M}(V)$ to denote the set of all monomials in $\chi_q(V)$ for a  finite-dimensional $U_q\widehat{\mathfrak{g}}$-module $V$. Let $\mathcal{P}^+ \subset \mathcal{P}$ denote the set of all dominant monomials. For $m_+ \in \mathcal{P}^+$, we use $\chi_q(m_+)$ to denote $\chi_q(L(m_+))$. We also write $m \in \chi_q(m_+)$ if $m \in \mathscr{M}(\chi_q(m_+))$.

The following lemma is well-known.
\begin{lemma}\label{well-known}
Let $m_1, m_2$ be two monomials. Then $L(m_1m_2)$ is a sub-quotient of $L(m_1) \otimes L(m_2)$. In particular, $\mathscr{M}(L(m_1m_2)) \subseteq \mathscr{M}(L(m_1))\mathscr{M}(L(m_2))$.   $\Box$
\end{lemma}

A finite-dimensional $U_q\widehat{\mathfrak{g}}$-module $V$ is said to be \textit{special} if and only if $\mathscr{M}(V)$ contains exactly one dominant monomial. It is called \textit{anti-special} if and only if $\mathscr{M}(V)$ contains exactly one anti-dominant monomial. Clearly, if a module is special or anti-special, then it is simple.

Let $a\in \mathbb{C}^{\times}$ and
\begin{align}\label{eqution a}
\begin{split}
&A_{1, a} = Y_{1, aq}Y_{1, aq^{-1}} Y_{2, a}^{-1}, \\
&A_{2, a} = Y_{2, aq^{1}}Y_{2, aq^{-1}} Y_{1, a}^{-1} Y_{3, a}^{-1},\\
&A_{3, a} = Y_{3, aq^{2}}Y_{3, aq^{-2}} Y_{4, a}^{-1} Y_{2, aq^{1}}^{-1}Y_{2, aq^{-1}}^{-1},\\
&A_{4, a} = Y_{4, aq^{2}}Y_{4, aq^{-2}} Y_{3, a}^{-1}.
\end{split}
\end{align}
Let $\mathcal{Q}$ be the subgroup of $\mathcal{P}$ generated by $A_{i, a}, i\in I, a\in \mathbb{C}^{\times}$. Let $\mathcal{Q}^{\pm}$ be the monoids generated by $A_{i, a}^{\pm 1}, i\in I, a\in \mathbb{C}^{\times}$. There is a partial order $\leq$ on $\mathcal{P}$ in which
\begin{align}
m\leq m' \text{ if and only if } m'm^{-1}\in \mathcal{Q}^{+}. \label{partial order of monomials}
\end{align}
For all $m_+ \in \mathcal{P}^+$, $\mathscr{M}(L(m_+)) \subset m_+\mathcal{Q}^{-}$, see \cite{FM01}.

The concept of \textit{right negative} is introduced in Section 6 of \cite{FM01}.
\begin{definition}
A monomial m is called \textit{right negative} if for all $a \in \mathbb{C}^{\times}$, for $L= \max\{l\in \mathbb{Z}\mid u_{i,aq^{l}}(m)\neq 0 \text{ for some i $\in I$}\}$ we have $u_{j,aq^{L}}(m)\leq 0$ for $j\in I$.
\end{definition}
For $i\in I, a\in \mathbb{C}^{\times}$, $A_{i,a}^{-1}$ is right-negative. A product of right-negative monomials is right-negative. If $m$ is right-negative and $m'\leq m$, then $m'$ is right-negative, see \cite{FM01}, \cite{Her06}.

\subsection{$q$-characters of $U_q \widehat{\mathfrak{sl}}_2$-modules and the Frenkel-Mukhin algorithm}
We recall the results of the $q$-characters of $U_q \widehat{\mathfrak{sl}}_2$-modules which are well-understood, see \cite{CP91}, \cite{FR98}.

Let $W_{k}^{(a)}$ be the irreducible representation $U_q \widehat{\mathfrak{sl}}_2$ with
highest weight monomial
\begin{align*}
X_{k}^{(a)}=\prod_{i=0}^{k-1} Y_{aq^{k-2i-1}},
\end{align*}
where $Y_a=Y_{1, a}$. Then the $q$-character of $W_{k}^{(a)}$ is given by
\begin{align*}
\chi_q(W_{k}^{(a)})=X_{k}^{(a)} \sum_{i=0}^{k} \prod_{j=0}^{i-1} A_{aq^{k-2j}}^{-1},
\end{align*}
where $A_a=Y_{aq^{-1}}Y_{aq}$.

For $a\in \mathbb{C}^{\times}, k\in \mathbb{Z}_{\geq 1}$, the set $\Sigma_{k}^{(a)}=\{aq^{k-2i-1}\}_{i=0, \ldots, k-1}$ is called a \textit{string}. Two strings $\Sigma_{k}^{(a)}$ and $\Sigma_{k'}^{(a')}$ are said to be in \textit{general position} if the union $\Sigma_{k}^{(a)} \cup \Sigma_{k'}^{(a')}$ is not a string or $\Sigma_{k}^{(a)} \subset \Sigma_{k'}^{(a')}$ or $\Sigma_{k'}^{(a')} \subset \Sigma_{k}^{(a)}$.

Denote by $L(m_+)$ the irreducible $U_q \widehat{\mathfrak{sl}}_2$-module with
highest weight monomial $m_+$. Let $m_{+} \neq 1$ and $\in \mathbb{Z}[Y_a]_{a\in \mathbb{C}^{\times}}$ be a dominant monomial. Then $m_+$ can be uniquely (up to permutation) written in the form
\begin{align*}
m_+=\prod_{i=1}^{s} \left( \prod_{b\in \Sigma_{k_i}^{(a_i)}} Y_{b} \right),
\end{align*}
where $s$ is an integer, $\Sigma_{k_i}^{(a_i)}, i=1, \ldots, s$, are strings which are pairwise in general position and
\begin{align*}
L(m_+)=\bigotimes_{i=1}^s W_{k_i}^{(a_i)}, \qquad \chi_q(L(m_+))=\prod_{i=1}^s
 \chi_q(W_{k_i}^{(a_i)}).
\end{align*}

For $j\in I$, let
\begin{align*}
\beta_j : \mathbb{Z}[Y_{i, a}^{\pm 1}]_{i\in I; a\in \mathbb{C}^{\times}} \to \mathbb{Z}[Y_{a}^{\pm 1}]_{a\in \mathbb{C}^{\times}}
\end{align*}
be the ring homomorphism such that for all $a\in \mathbb{C}^{\times}$, $Y_{k, a} \mapsto 1$ for $k\neq j$ and $Y_{j, a} \mapsto Y_{a}$.

Let $V$ be a $U_q \widehat{\mathfrak{g}}$-module. Then $\beta_i(\chi_q(V))$, $i\in I$, is the $q$-character of $V$ considered as a $U_{q_i} \widehat{\mathfrak{sl}_2} $-module.

The Frenkel-Mukhin algorithm is introduced in Section 5 in \cite{FM01} to compute the $q$-characters of $U_q \widehat{\mathfrak{g}}$-modules. In Theorem 5.9 of \cite{FM01}, it is shown that the Frenkel-Mukhin algorithm works for modules which are special.

\subsection{Truncated $q$-characters}
In this paper, we need to use the concept truncated $q$-characters, see \cite{HL10}, \cite{MY12a}.
Given a set of monomials $\mathcal{R} \subset \mathcal{P}$, let $\mathbb{Z}\mathcal{R} \subset \mathbb{Z}\mathcal{P}$ denote the $\mathbb{Z}$-module of formal linear combinations of elements of $\mathcal{R}$ with integer coefficients. Define
\begin{align*}
\text{trunc}_{\mathcal{R}}: \mathcal{P} \to \mathcal{R}; \quad m\mapsto \begin{cases} m & \text{if } m\in \mathcal{R}, \\ 0 & \text{if } m \not\in \mathcal{R}, \end{cases}
\end{align*}
and extend $\text{trunc}_{\mathcal{R}}$ as a $\mathbb{Z}$-module map $\mathbb{Z}\mathcal{P} \to \mathbb{Z}\mathcal{R}$.

Given a subset $U\subset I \times \mathbb{C}^{\times}$, let $\mathcal{Q}_U$ be the subgroups of $\mathcal{Q}$ generated by $A_{i,a}$ with $(i,a)\in U$. Let $\mathcal{Q}_U^{\pm}$ be the monoid generated by $A_{i,a}^{\pm 1}$ with $(i,a)\in U$. The polynomial $\text{trunc}_{m_{+}\mathcal{Q}_U^{-}} \: \chi_q(m_{+})$ is called \textit{the $q$-character of $L(m_{+})$ truncated to $U$}.

The following theorem can be used to compute some truncated $q$-characters.

\begin{theorem}[{  \cite{MY12a} }, Theorem 2.1 ]  \label{truncated}
Let $U\subset I \times \mathbb{C}^{\times}$ and $m_{+} \in \mathcal{P}^+$. Suppose that $\mathcal{M} \subset \mathcal{P}$ is a finite set of distinct monomials such that
\begin{enumerate}[(i)]
\item $\mathscr{M} \subset m_+\mathcal{Q}_U^{-}$,
\item $\mathcal{P}^+ \cap \mathscr{M} = \{m_{+}\}$,
\item for all $m\in \mathscr{M}$ and all $(i,a)\in U$, if $mA_{i,a}^{-1} \not\in \mathscr{M}$, then $mA_{i,a}^{-1}A_{j,b} \not\in \mathscr{M}$ unless $(j,b)=(i,a)$,
\item for all $m \in \mathscr{M}$ and all $i \in I$, there exists a unique $i$-dominant monomial $M\in \mathscr{M}$ such that
\begin{align*}
\text{trunc}_{\beta_i(M\mathcal{Q}_U^{-})} \: \chi_q(\beta_i(M)) = \sum_{m'\in m\mathcal{Q}_{\{i\} \times \mathbb{C}^{\times}} \cap \mathscr{M} } \beta_i(m').
\end{align*}
\end{enumerate}
Then
\begin{align*}
\text{trunc}_{m_{+}\mathcal{Q}_U^{-}} \: \chi_q(m_+) = \sum_{m\in \mathscr{M}} m.
\end{align*}
\end{theorem}

Here $\chi_q(\beta_i(M))$ is the $q$-character of the irreducible $U_{q_i}(\hat{\mathfrak{sl}_2})$-module with highest weight monomial $\beta_i(M)$ and $\text{trunc}_{\beta_i(M\mathcal{Q}_{U}^{-})}$ is the polynomial obtained from $\chi_q(\beta_i(M))$ by keeping only the monomials of $\chi_q(\beta_i(M))$ in the set $\beta_i(M\mathcal{Q}_U^-)$.

\subsection{Minimal affinizations of $U_q \mathfrak{g}$-modules} \label{definition of minimal affinizations}
In what follows, we fix an $a\in \mathbb{C}^{\times}$ and denote $i_{s}=Y_{i, aq^{s}}$, $i\in I$, $s\in \mathbb{Z}$. Let $\lambda = k\omega_1 + l \omega_2 + m\omega_{3} + n\omega_{4}$, $k,l,m,n \in \mathbb{Z}_{\geq 0}$ and $V(\lambda)$ the simple $U_{q}\mathfrak{g}$-module with highest weight $\lambda$. Without loss of generality, we may assume that a simple $U_q \widehat{\mathfrak{g}} $-module $L(m_+)$ is a \textit{minimal affinization} of $V(\lambda)$ if and only if $m_+$ is one of the following monomials:
\begin{gather}
\begin{align*}
\widetilde{T}_{n, m, l, k}^{(s)}=&\left( \prod_{i=0}^{n-1} Y_{4, aq^{s+4i}} \right)  \left( \prod_{i=0}^{m-1} Y_{3, aq^{s+4n+4i+2}} \right)  \left( \prod_{i=0}^{l-1} Y_{2, aq^{s+4n+4m+2i+3}} \right)     \left( \prod_{i=0}^{k-1} Y_{1, aq^{s+4n+4m+2l+2i+4}} \right),
\end{align*}
\end{gather}
\begin{gather}
\begin{align*}
T_{n,m,l,k}^{(s)}=\left( \prod_{i=0}^{n-1} Y_{4, aq^{-s-4i}} \right)  \left( \prod_{i=0}^{m-1} Y_{3, aq^{-s-4n-4i-2}} \right)  \left( \prod_{i=0}^{l-1} Y_{2, aq^{-s-4n-4m-2i-3}} \right)     \left( \prod_{i=0}^{k-1} Y_{1, aq^{-s-4n-4m-2l-2i-4}} \right),
\end{align*}
\end{gather}
where $\prod_{0\leq j\leq -1}=1$, $s\in \mathbb{Z}$, see \cite{CP96a}. We denote $A^{-1}_{i, aq^{s}}$ by $A^{-1}_{i, s}$. We use $\mathcal T_{k, l, m, n}^{(s)}$ (resp. $\mathcal{\widetilde{T}}_{k, l, m, n}^{(s)}$) to denote the irreducible finite-dimensional $U_{q}\widehat{\mathfrak{g}}$-module with highest $l$-weight $T_{k, l, m, n}^{(s)}$ (resp. $\widetilde{T}_{k, l, m, n}^{(s)}$).

\section{A closed system which contains a large family of minimal affinizations of type $F_4$ } \label{main results}
In this section, we introduce a closed system of type $F_{4}$ that contains a large family of minimal affinizations:
\begin{align*}
&\mathcal{T}_{0,0,l,k}^{(-2)}, \  \mathcal{\widetilde{T}}_{n,0,l,0}^{(s)}, \ \mathcal{\widetilde{T}}_{0,m,l,0}^{(s)},
\ \mathcal{\widetilde{T}}_{n,m,0,0}^{(s)},
\ \mathcal{\widetilde{T}}_{n,m,l,0}^{(s)},
\ \mathcal{\widetilde{T}}_{n,0,0,k}^{(s)},
\ \mathcal{\widetilde{T}}_{0,m,0,k}^{(s)}(k \leq 2),\\
&\mathcal{\widetilde{T}}_{n,m,0,k}^{(s)}(k \leq 2),
\ \mathcal{\widetilde{T}}_{n,m,l,k}^{(s)}(k \leq 2),
\ \mathcal{\widetilde{T}}_{n,0,l,k}^{(s)}(k \leq 2),
\ \mathcal{\widetilde{T}}_{0,m,l,k}^{(s)}(k \leq 2).
\end{align*}
Here $k, l, m, n \in \mathbb{Z}_{\geq 1}$, $s\in \mathbb{Z}$.
\subsection{Special modules}
Let $k, l, m, n \in \mathbb{Z}_{\geq 1}$, $s\in \mathbb{Z}$. For $k\leq 0$, let $T^{(s)}_{0,0,0,k}=1$. We define
\begin{align*}
\widetilde{P}^{(s)}_{0,0,l,k}&=\widetilde{T}^{(s)}_{0, 0, l, k}, \quad k\in \mathbb{Z}_{\leq 2}, \\
\widetilde{P}^{(s)}_{0,m,0,k}&=\widetilde{T}^{(s)}_{0, m, 0, k}, \quad k\in \mathbb{Z}_{\leq 2}, \\
\widetilde{P}^{(s)}_{0,0,0,k}&=\widetilde{T}^{(s)}_{0,0,0,k}, \quad k\in \mathbb{Z}_{\geq 3}, \\
\widetilde{P}^{(s)}_{0,0,l,k}&=\widetilde{T}^{(s+2l+4)}_{0,0,0,k-2} \widetilde{T}^{(s+2l+8)}_{0,0,0,k-4}\widetilde{T}^{(s)}_{0,0,l,k}, \quad k\in \mathbb{Z}_{\geq 3}, \\
\widetilde{P}^{(s)}_{0,m,0,k}&=\widetilde{T}^{(s)}_{0,m,0,k} \widetilde{T}^{(s+4m+4)}_{0,0,0,k-2}, \quad k\in \mathbb{Z}_{\geq 3}.
\end{align*}
We use $\widetilde{\mathcal{P}}^{(s)}_{n,m,l,k}$ to denote the simple $U_q \widehat{\mathfrak{g}}$-module with the highest weight monomial $\widetilde{P}^{(s)}_{n,m,l,k}$.

\begin{theorem}[Theorem 3.9, \cite{Her07}]\label{Her 07}
For $l, m, n\in \mathbb{Z}_{\geq 1}$, $s\in \mathbb{Z}$, the modules $\mathcal{\widetilde{T}}_{n,m,0,0}^{(s)}$, $\mathcal{\widetilde{T}}_{n,0,l,0}^{(s)}$, $\mathcal{\widetilde{T}}_{0,m,l,0}^{(s)}$, $\mathcal{\widetilde{T}}_{n,m,l,0}^{(s)}$ are special.
\end{theorem}

\begin{remark}
In the paper \cite{Her07}, $\alpha_1, \alpha_2$ are simple long roots and $\alpha_3, \alpha_4$ are simple short roots. In this paper, $\alpha_1, \alpha_2$ are simple short roots and $\alpha_3, \alpha_4$ are simple long roots.
\end{remark}

\begin{theorem}\label{special}
The modules
\begin{align*}
&\mathcal{T}_{0,0,l,k}^{(-2)},
\  \mathcal{\widetilde{T}}_{n,0,l,0}^{(s)},
\ \mathcal{\widetilde{T}}_{0,m,l,0}^{(s)},
\ \mathcal{\widetilde{T}}_{n,m,0,0}^{(s)},
\ \mathcal{\widetilde{T}}_{n,m,l,0}^{(s)},
\ \mathcal{\widetilde{T}}_{n,0,0,k}^{(s)},
\mathcal{\widetilde{T}}_{n,m,0,k}^{(s)} \ (k \leq 2),\\
&\mathcal{\widetilde{T}}_{n,m,l,k}^{(s)} \ (k \leq 2),
\ \mathcal{\widetilde{T}}_{n,0,l,k}^{(s)} \ (k \leq 2),
\ \mathcal{\widetilde{T}}_{0,m,l,k}^{(s)} \ (k \leq 2),
\ \widetilde{P}^{(s)}_{0,0,l,k},
\ \widetilde{P}^{(s)}_{0,m,0,k},
\end{align*}
where $k, l, m, n\in \mathbb{Z}_{\geq 1}$, $s\in \mathbb{Z}$, are special.
\end{theorem}
We will prove Theorem \ref{special} in Section \ref{proof of special}.

Since the modules are special in Theorem \ref{Her 07} and Theorem \ref{special}, we can use the Frenkel-Mukhin algorithm to compute the $q$-characters of these modules.

\subsection{A closed system of type $F_4$}
\begin{theorem}\label{M-system of type F4}
For $s\in \mathbb{Z}$ and $k, l, m, n\geq 1$, we have
\begin{align}\label{eqn 1}
[\mathcal{T}_{0,0,l-1,k}^{(-2)}][\mathcal{T}_{0,0,l,k}^{(-2)}]=[\mathcal{T}_{0,0,l,k-1}^{(-2)}][\mathcal{T}_{0,0,l-1,k+1}^{(-2)}]
+[\mathcal{T}_{0,0,k+l,0}^{(-2)}][\mathcal{T}_{0,0,l-1,0}^{(-2)}],
\end{align}
\begin{align}\label{eqn 2}
[\mathcal{\widetilde{T}}_{n,m-1,0,0}^{(s+4)}][\mathcal{\widetilde{T}}_{n,m,0,0}^{(s)}]=[\mathcal{\widetilde{T}}_{n-1,m,0,0}^{(s+4)}][\mathcal{\widetilde{T}}_{n+1,m-1,0,0}^{(s)}]
+[\mathcal{\widetilde{T}}_{0,m-1,0,0}^{(s+4n+4)}][\mathcal{\widetilde{T}}_{0,n+m,0,0}^{(s)}],
\end{align}

\begin{align}\label{eqn 3a}
[\mathcal{\widetilde{T}}_{n,0,0,0}^{(s+4)}][\mathcal{\widetilde{T}}_{n,0,1,0}^{(s)}]&=[\mathcal{\widetilde{T}}_{n-1,0,1,0}^{(s+4)}][\mathcal{\widetilde{T}}_{n+1,0,0,0}^{(s)}]
+[\mathcal{\widetilde{T}}_{0,n,1,0}^{(s)}],
\end{align}
\begin{align}\label{eqn 3b}
[\mathcal{\widetilde{T}}_{n,0,l-2,0}^{(s+4)}][\mathcal{\widetilde{T}}_{n,0,l,0}^{(s)}]&=[\mathcal{\widetilde{T}}_{n-1,0,l,0}^{(s+4)}][\mathcal{\widetilde{T}}_{n+1,0,l-2,0}^{(s)}]
+[\mathcal{\widetilde{T}}_{0,0,l-2,0}^{(s+4n+4)}][\mathcal{\widetilde{T}}_{0,n,l,0}^{(s)}], \ l\geq 2,
\end{align}

\begin{align}\label{eqn 4a}
[\mathcal{\widetilde{T}}_{0,m,0,0}^{(s+4)}][\mathcal{\widetilde{T}}_{0,m,1,0}^{(s)}]&=[\mathcal{\widetilde{T}}_{0,m-1,1,0}^{(s+4)}][\mathcal{\widetilde{T}}_{0,m+1,0,0}^{(s)}]
+[\mathcal{\widetilde{T}}_{m,0,0,0}^{(s+4)}][\mathcal{\widetilde{T}}_{0,0,1+2m,0}^{(s)}],
\end{align}

\begin{align}\label{eqn 4b}
[\mathcal{\widetilde{T}}_{0,m,l-2,0}^{(s+4)}][\mathcal{\widetilde{T}}_{0,m,l,0}^{(s)}]&=[\mathcal{\widetilde{T}}_{0,m-1,l,0}^{(s+4)}][\mathcal{\widetilde{T}}_{0,m+1,l-2,0}^{(s)}]
+[\mathcal{\widetilde{T}}_{m,0,l-2,0}^{(s+4)}][\mathcal{\widetilde{T}}_{0,0,l+2m,0}^{(s)}],\ l\geq 2,
\end{align}

\begin{align}\label{eqn 5}
[\mathcal{\widetilde{T}}_{n,m-1,l,0}^{(s+4)}][\mathcal{\widetilde{T}}_{n,m,l,0}^{(s)}]=[\mathcal{\widetilde{T}}_{n-1,m,l,0}^{(s+4)}][\mathcal{\widetilde{T}}_{n+1,m-1,l,0}^{(s)}]
+[\mathcal{\widetilde{T}}_{0,m-1,l,0}^{(s+4n+4)}][\mathcal{\widetilde{T}}_{0,n+m,l,0}^{(s)}],
\end{align}

\begin{align}\label{eqn 6a}
[\mathcal{\widetilde{T}}_{n,0,0,0}^{(s+4)}][\mathcal{\widetilde{T}}_{n,0,0,k}^{(s)}]=[\mathcal{\widetilde{T}}_{n-1,0,0,k}^{(s+4)}][\mathcal{\widetilde{T}}_{n+1,0,0,0}^{(s)}]
+[\mathcal{\widetilde{T}}_{0,n,0,k}^{(s)}],\ k=1, 2,
\end{align}

\begin{align}\label{eqn 7a}
[\mathcal{\widetilde{T}}_{0,m,0,0}^{(s+4)}][\mathcal{\widetilde{T}}_{0,m,0,k}^{(s)}]&=[\mathcal{\widetilde{T}}_{0,m-1,0,k}^{(s+4)}][\mathcal{\widetilde{T}}_{0,m+1,0,0}^{(s)}]
+[\mathcal{\widetilde{T}}_{0,0,2m,k}^{(s)}][\mathcal{\widetilde{T}}_{m,0,0,0}^{(s+4)}],\ k=1, 2,
\end{align}

\begin{align}\label{eqn 8}
[\mathcal{\widetilde{T}}_{n,0,0,k-1}^{(s+4)}][\mathcal{\widetilde{T}}_{n,0,1,k}^{(s)}]=[\mathcal{\widetilde{T}}_{n-1,0,1,k}^{(s+4)}][\mathcal{\widetilde{T}}_{n+1,0,0,k-1}^{(s)}]
+[\mathcal{\widetilde{T}}_{0,n,1,k}^{(s)}],\ k=1,2,
\end{align}

\begin{align}\label{eqn 9}
[\mathcal{\widetilde{T}}_{n,0,l-2,k}^{(s+4)}][\mathcal{\widetilde{T}}_{n,0,l,k}^{(s)}]=[\mathcal{\widetilde{T}}_{n-1,0,l,k}^{(s+4)}][\mathcal{\widetilde{T}}_{n+1,0,l-2,k}^{(s)}]
+[\mathcal{\widetilde{T}}_{0,n,l,k}^{(s)}][\mathcal{\widetilde{T}}_{0,0,l-2,k}^{(s+4n+4)}],\ k=1,2,\ l \geq 2,
\end{align}

\begin{align}\label{eqn 10}
[\mathcal{\widetilde{T}}_{0,m,0,k-1}^{(s+4)}][\mathcal{\widetilde{T}}_{0,m,1,k}^{(s)}]=[\mathcal{\widetilde{T}}_{0,m-1,1,k}^{(s+4)}][\mathcal{\widetilde{T}}_{0,m+1,0,k-1}^{(s)}]
+[\mathcal{\widetilde{T}}_{0,0,1+2m,k}^{(s)}][\mathcal{\widetilde{T}}_{m,0,0,k-1}^{(s+4)}],\ k=1,2,
\end{align}

\begin{align}\label{eqn 11}
&[\mathcal{\widetilde{T}}_{0,m,l-2,k}^{(s+4)}][\mathcal{\widetilde{T}}_{0,m,l,k}^{(s)}]=[\mathcal{\widetilde{T}}_{0,m-1,l,k}^{(s+4)}][\mathcal{\widetilde{T}}_{0,m+1,l-2,k}^{(s)}]
+[\mathcal{\widetilde{T}}_{0,0,l+2m,k}^{(s)}][\mathcal{\widetilde{T}}_{m,0,l-2,k}^{(s+4)}],\ k=1,2,\ l \geq 2,
\end{align}

\begin{align}\label{eqn 12}
[\mathcal{\widetilde{T}}_{n,m-1,0,k}^{(s+4)}][\mathcal{\widetilde{T}}_{n,m,0,k}^{(s)}]=[\mathcal{\widetilde{T}}_{n-1,m,0,k}^{(s+4)}][\mathcal{\widetilde{T}}_{n+1,m-1,0,k}^{(s)}]
+[\mathcal{\widetilde{T}}_{0,n+m,0,k}^{(s)}][\mathcal{\widetilde{T}}_{0,m-1,0,k}^{(s+4)}],\ k=1,2,
\end{align}

\begin{align}\label{eqn 13}
[\mathcal{\widetilde{T}}_{n,m-1,l,k}^{(s+4)}][\mathcal{\widetilde{T}}_{n,m,l,k}^{(s)}]=[\mathcal{\widetilde{T}}_{n-1,m,l,k}^{(s+4)}][\mathcal{\widetilde{T}}_{n+1,m-1,l,k}^{(s)}]
+[\mathcal{\widetilde{T}}_{0,n+m,l,k}^{(s)}][\mathcal{\widetilde{T}}_{0,m-1,l,k}^{(s+4n+4)}],\ k=1,2,
\end{align}

\begin{align} \label{eqn 14}
\begin{split}
[\widetilde{\mathcal{P}}^{(s+2)}_{0,0,1,k-1}][\widetilde{\mathcal{P}}^{(s)}_{0,0,1,k}] & =[\widetilde{\mathcal{T}}^{(s+10)}_{0,0,0,k-4}][\widetilde{\mathcal{T}}^{(s+6)}_{0,0,0,k-2}]
[\widetilde{\mathcal{T}}^{(s+2)}_{0,0,0,k}] [\widetilde{\mathcal{P}}^{(s)}_{0,0,2,k-1}] \\
& \quad + [\widetilde{\mathcal{T}}^{(s+12)}_{0,0,0,k-5}] [\widetilde{\mathcal{T}}^{(s+8)}_{0,0,0,k-3}] [\widetilde{\mathcal{T}}^{(s+4)}_{0,0,0,k-1}] [\widetilde{\mathcal{T}}^{(s)}_{0,0,0,k+1}] [\widetilde{\mathcal{P}}^{(s+2)}_{0,1,0,k-2}],\ k \geq 3,\ l=1,
\end{split}
\end{align}

\begin{align}\label{eqn 15}
[\widetilde{\mathcal{P}}^{(s+2)}_{0,0,l,k-1}][\widetilde{\mathcal{P}}^{(s)}_{0,0,l,k}]=[\widetilde{\mathcal{P}}^{(s+2)}_{0,0,l-1,k}][\widetilde{\mathcal{P}}^{(s)}_{0,0,l+1,k-1}]
+[\widetilde{\mathcal{T}}^{(s+2l+10)}_{0,0,0,k-5}][\widetilde{\mathcal{T}}^{(s)}_{0,0,0,k+l}] [\widetilde{\mathcal{P}}^{(s+2)}_{0,\frac{l+1}{2},0,k-2}]
[\widetilde{\mathcal{P}}^{(s+4)}_{0,\frac{l-1}{2},0,k-1}],
\end{align}
where $k\geq3, \text{ l is odd}, \ l\geq3$,

\begin{align} \label{eqn 16}
[\widetilde{\mathcal{P}}^{(s+2)}_{0,0,l,k-1}][\widetilde{\mathcal{P}}^{(s)}_{0,0,l,k}]=[\widetilde{\mathcal{P}}^{(s+2)}_{0,0,l-1,k}][\widetilde{\mathcal{P}}^{(s)}_{0,0,l+1,k-1}]
+[\widetilde{\mathcal{T}}^{(s+2l+10)}_{0,0,0,k-5}][\widetilde{\mathcal{T}}^{(s)}_{0,0,0,k+l}] [\widetilde{\mathcal{P}}^{(s+4)}_{0,\frac{l}{2},0,k-2}]
[\widetilde{\mathcal{P}}^{(s+2)}_{0,\frac{l}{2},0,k-1}],
\end{align}
where $k\geq3, \text{ l is even}, \ l \geq 2$,

\begin{align} \label{eqn 17}
[\widetilde{\mathcal{P}}^{(s+4)}_{0,1,0,k-2}][\widetilde{\mathcal{P}}^{(s)}_{0,1,0,k}]=[\widetilde{\mathcal{T}}^{(s+8)}_{0,0,0,k-2}][\widetilde{\mathcal{T}}^{(s+4)}_{0,0,0,k}]
[\widetilde{\mathcal{P}}^{(s)}_{0,2,0,k-2}]+[\widetilde{\mathcal{P}}^{(s)}_{0,0,2,k}][\widetilde{\mathcal{T}}^{(s+4)}_{1,0,0,k-2}],\ k \geq 3,
\end{align}

\begin{align}\label{eqn 18}
[\widetilde{\mathcal{P}}^{(s+4)}_{0,m,0,k-2}][\widetilde{\mathcal{P}}^{(s)}_{0,m,0,k}]=[\widetilde{\mathcal{P}}^{(s+4)}_{0,m-1,0,k}][\widetilde{\mathcal{P}}^{(s)}_{0,m+1,0,k-2}]
+[\widetilde{\mathcal{P}}^{(s)}_{0,0,2m,k}][\widetilde{\mathcal{T}}^{(s+4)}_{m,0,0,k-2}],\ k\geq 3, \ m\geq 2,
\end{align}

\begin{align}\label{eqn 19}
[\widetilde{\mathcal{T}}^{(s+4)}_{n,0,0,k-2}][\widetilde{\mathcal{T}}^{(s)}_{n,0,0,k}]&=[\widetilde{\mathcal{T}}^{(s+4)}_{n-1,0,0,k}][\widetilde{\mathcal{T}}^{(s)}_{n+1,0,0,k-2}]+[\widetilde{\mathcal{P}}^{(s)}_{0,n,0,k}],
\ k\geq3.
\end{align}

\end{theorem}

Theorem \ref{M-system of type F4} will be proved in Section \ref{proof main1}. The system in Theorem \ref{M-system of type F4} is a closed system in the sense that all modules in the system can be computed recursively using Kirillov-Reshetikhin modules.

\begin{example}\label{example 1}
The following are some equations in the system in Theorem \ref{M-system of type F4}.
\begin{gather}
\begin{align*}
&[1_{-2}][1_{-4}2_{-1}]=[1_{-4}1_{-2}][2_{-1}]+ [2_{-3}2_{-1}], \\
&[1_{-4}1_{-2}][1_{-6}1_{-4}2_{-1}]=[1_{-4}2_{-1}][1_{-6}1_{-4}1_{-2}]+[2_{-5}2_{-3}2_{-1}], \\
&[1_{-6}1_{-4}1_{-2}][1_{-8}1_{-6}1_{-4}2_{-1}]=[1_{-6}1_{-4}2_{-1}][1_{-8}1_{-6}1_{-4}1_{-2}]+[2_{-7}2_{-5}2_{-3}2_{-1}],\\
&[3_{-4}][1_{0}2_{-3}3_{-8}]=[1_{0}2_{-3}][3_{-8}3_{-4}]+[1_{0}2_{-7}2_{-5}2_{-3}][4_{-6}],\\
&[3_{-8}3_{-4}][1_{0}2_{-3}3_{-12}3_{-8}=[1_{0}2_{-3}3_{-8}][3_{-12}3_{-8}3_{-4}]+[1_{0}2_{-11}2_{-9}2_{-7}2_{-5}2_{-3}][4_{-10}4_{-6}],\\
&[3_{-12}3_{-8}3_{-4}][1_{0}2_{-3}3_{-16}3_{-12}3_{-8}]=[1_{0}2_{-3}3_{-12}3_{-8}][3_{-16}3_{-12}3_{-8}3_{-4}]+[1_{0}2_{-15}2_{-13}2_{-11}2_{-9}2_{-7}2_{-5}2_{-3}][4_{-14}  4_{-10}4_{-6}],\\
&[4_{-6}][1_{0}2_{-3}4_{-10}]=[1_{0}2_{-3}][4_{-10}4_{-6}]+[1_{0}2_{-3}3_{-8}],\\
&[4_{-10}4_{-6}][1_{0}2_{-3}4_{-14}4_{-10}]=[1_{0}2_{-3}4_{-10}][4_{-14}4_{-10}4_{-6}]+[1_{0}2_{-3}3_{-12}3_{-8}],\\
&[4_{-14}4_{-10}4_{-6}][1_{0}2_{-3}4_{-18}4_{-14}4_{-10}]=[1_{0}2_{-3}4_{-14}4_{-10}][4_{-18}4_{-14}4_{-10}4_{-6}]+[1_{0}2_{-3}3_{-16}3_{-12}3_{-8}].\\
\end{align*}
\end{gather}
\end{example}

Moreover, we have the following theorem.
\begin{theorem}\label{irreducible}
For each equation in Theorem $\ref{M-system of type F4}$, all summands on the right hand side are simple.
\end{theorem}
Theorem \ref{irreducible} will be proved in Section \ref{proof irreducible}.
\subsection{A system corresponding to the system in Theorem \ref{M-system of type F4}}
Given $k, l, m, n \in \mathbb{Z}_{\geq 0}, s\in \mathbb{Z}$, let
\[
\mathfrak{m}_{k, l, m, n} =\res(\mathcal{T}_{k, l, m, n}^{(s)}) \ (\text{resp. } \widetilde{\mathfrak{m}}_{k, l, m, n} =\res(\widetilde{\mathcal{T}}_{k, l, m, n}^{(s)}))
\]
be the restriction of $\mathcal{T}_{k, l, m, n}^{(s)}$\ (resp. $\widetilde{\mathcal{T}}_{k, l, m, n}^{(s)}$) to $U_q \mathfrak{g}$. It is clear that $\res(\mathcal{T}_{k, l, m, n}^{(s)})$ and $\res(\widetilde{\mathcal{T}}_{k, l, m, n}^{(s)})$ do not depend on $s$. Let $\chi(M)$\ (resp. $\chi(\widetilde{M})$) be the character of a $U_q \mathfrak{g}$-module $M$\ (resp. $\widetilde{M}$). By replacing each $[\mathcal{T}_{n, m, l, k}^{(s)}]$ (resp. $[\mathcal{\widetilde{T}}_{n, m, l, k}^{(s)}]$) in the system in Theorem \ref{M-system of type F4} with  $\chi(\mathfrak{m}_{n, m, l, k})$ (resp. $\chi(\widetilde{\mathfrak{m}}_{n, m, l, k})$), we obtain a system of equations consisting of the characters of $U_{q}\mathfrak{g}$-modules. The following are two equations in the system.
\begin{equation*}
\begin{split}
&\chi(\mathfrak{m}_{0,0,l-1,k})\chi(\mathfrak{m}_{0,0,l,k})=\chi(\mathfrak{m}_{0,0,l,k-1})\chi(\mathfrak{m}_{0,0,l-1,k+1})+\chi(\mathfrak{m}_{0,0,k+l,0})\chi(\mathfrak{m}_{0,0,l-1,0}),\\
&\chi(\widetilde{\mathfrak{m}}_{n,m-1,0,0})\chi(\widetilde{\mathfrak{m}}_{n,m,0,0})=\chi(\widetilde{\mathfrak{m}}_{n-1,m,0,0})\chi(\widetilde{\mathfrak{m}}_{n+1,m-1,0,0})+\chi(\widetilde{\mathfrak{m}}_{0,m-1,0,0})\chi(\widetilde{\mathfrak{m}}_{0,n+m,0,0}).
\end{split}
\end{equation*}

\section{Relation between the system in Theorem \ref{M-system of type F4} and cluster algebras}\label{Relation between the M-systems and cluster algebras}

In this section, we will show that the equations in the system in Theorem \ref{M-system of type F4} correspond to mutations in some cluster algebra $\mathscr{A}$. Moreover, every minimal affinization in the system in Theorem \ref{M-system of type F4}  corresponds to a cluster variable in the cluster algebra $\mathscr{A}$.
\subsection{Definition of a cluster algebra $\mathscr{A}$} \label{definition of cluster algebra A}
Let $I=\{1, 2, 3, 4\}$ and
\begin{align*}
S&=\{-2u \mid u\in \mathbb{Z}_{\geq 0}\}, \\
S'&=\{-2u-1 \mid u\in \mathbb{Z}_{\geq 0}\}.
\end{align*}
Let
\begin{align*}
V=(\{1\}\times S)\cup (\{2\}\times S')\cup (\{3\}\times S)\cup (\{4\}\times S).
\end{align*}
We define $Q$ with vertex set $V$ as follows. The arrows of $Q$ from the vertex $(i,r)$ to the vertex $(j,s)$ if and only if $b_{ij}\neq 0$ and $s=r-b_{ij}+d_{i}-d_{j}$. The quiver $Q$ is the same as the quiver $G^-$ of type $F_4$ defined in
\cite{HL13}.

Let ${\bf t} = {\bf t}_1 \cup {\bf t}_2$, where
\begin{align*}
{\bf t}_1=\{t_{0, 0, l, 0}^{(-2)},\ t_{0, 0, 0, k}^{(-4)} \mid k, l \in \mathbb{Z}_{\geq 1}\},
\end{align*}
\begin{gather}
\begin{align*}
{\bf t}_2=\{\widetilde{t}_{n, 0, 0, 0}^{(-4n+4)},\ \widetilde{t}_{0, m, 0, 0}^{(-4m)}, \ \widetilde{t}_{n, 0, 0, 0}^{(-4n+2)},\ \widetilde{t}_{0, m, 0, 0}^{(-4m+2)},\ \widetilde{t}_{0, 0, l, 0}^{(-2l-2)},\ \widetilde{t}_{0, 0, 0, k}^{(2k-2)} \mid k, l, m, n \in \mathbb{Z}_{\geq 1}\}.
\end{align*}
\end{gather}
Let $\mathscr{A}$ be the cluster algebra defined by the initial seed $({\bf t}, Q)$. By Definition \ref{definition of cluster algebras of infinite rank}, $\mathscr{A}$ is the $\mathbb{Q}$-subalgebra of the field of rational functions $\mathbb{Q}({\bf t})$ generated by all the elements obtained from some elements of $\bf t$ via a finite sequence of seed mutations.

\subsection{Mutation sequences}
We use $``C_{1}"$, $``C_{2}"$, $``C_{3}"$, $``C_{4}"$, $``C_{5}"$, $``C_{6}"$ to denote the column of vertices $(1, 0)$, $(1, 2)$, $\ldots$, $(1, -2u)$, $\cdots$, the column of vertices $(2, -1)$, $(2, -3)$, $\ldots$, $(2, -2u-1)$, $\cdots$, the column of vertices $(3, 0)$, $(3, -4)$, $\ldots$, $(3, -4u)$, $\cdots$, the column of vertices $(3, -2)$, $(3, -6)$, $\ldots$, $(3, -4u-2)$, $\cdots$, the column of vertices $(4, 0)$, $(4, -4)$, $\ldots$, $(4, -4u)$, $\cdots$, the column of vertices $(4, -2)$, $(4, -6)$, $\ldots$, $(4, -4u-2)$, $\cdots$, respectively in $Q$.

By saying that we mutate at the column $C_i$, $i \in \{1, 2, 3, 4, 5, 6\}$, we mean that we mutate the vertices of $C_i$ as follows. First we mutate at the first vertex in this column, then the second vertex, and so on until the vertex at infinity. By saying that we mutate $(C_{i_0}, C_{i_1}, \ldots, C_{i_u} )$, where $i_j \in \{1, 2, 3, 4, 5, 6\}$, $j=0,1,2,\ldots, u$, we mean that we first mutate the column $C_{i_1}$, then the column $C_{i_2}$, and so on up to the column $C_{i_u}$.

For $k, l, m, n \in \mathbb{Z}_{\geq 1}$, we define some variables
\begin{align} \label{cluster variables}
\begin{split}
& t_{0, 0, l, k}^{(-2)}, \ \widetilde{t}_{n,m,0,0}^{(-4n-4m)}, \ \widetilde{t}_{n,m,0,0}^{(-4n-4m+2)}, \ \widetilde{t}_{n,0,l,0}^{(-4n-2l-2)}, \ \widetilde{t}_{0,m,l,0}^{(-4m-2l-2)}, \  \ \widetilde{t}_{n,m,l,0}^{(-4n-4m-2l-2)},  \\
&\widetilde{t}_{n,0,l,k}^{(-4n-2l-2k-2)} \ (k \leq 2),  \ \widetilde{t}_{0,m,l,k}^{(-4m-2l-2k-2)} \ (k \leq 2),  \ \widetilde{t}_{n,m,0,k}^{(-4n-4m-2k-2)} \ (k \leq 2), \\
&\widetilde{t}_{n,m,l,k}^{(-4n-4m-2l-2k-2)} \ (k \leq 2), \ \widetilde{p}^{(-2l-2k-2)}_{0,0,l,k}, \ \widetilde{p}^{(-4m-2k-2)}_{0,m,0,k},\  \widetilde{t}_{n,0,0,k}^{(-4n-2k-2)},
\end{split}
\end{align}
recursively as follows. The variables in ${\bf t}$ are already defined.

We mutate the first vertex of the first $C_{1}$ in $(C_1, C_1, \ldots, C_1)$ from the initial seed, and define $t_{0,0,0,1}^{(-2)}=t'{_{0,0,0,1}^{(-4)}}$, we obtain a quiver ($Q_{11}$). Therefore
\begin{eqnarray}
t_{0, 0, 0, 1}^{(-2)} = {t'}_{0, 0, 0, 1}^{(-4)} = \frac{t_{0, 0, 0, 2}^{(-4)}+t_{0, 0, 1, 0}^{(-2)}}{t_{0, 0, 0, 1}^{(-4)}}.
\end{eqnarray}
We mutate the second vertex of the first $C_{1}$ in $(C_1, C_1, \ldots, C_1)$ and define $t_{0, 0, 0, 2}^{(-2)} = t'{_{0, 0, 0, 2}^{(-4)}}$, the quiver ($Q_{11}$) becomes a quiver ($Q_{12}$). Therefore
\begin{eqnarray}
t_{0, 0, 0, 2}^{(-2)} = t'{_{0, 0, 0, 2}^{(-4)}} = \frac{t_{0, 0, 0, 3}^{(-4)} t_{0, 0, 0, 1}^{(-4)} + t_{0, 0, 2, 0}^{(-2)}}{t_{0, 0, 0, 2}^{(-4)}}.
\end{eqnarray}
We continue this procedure and mutate the vertices of the first $C_1$ in $(C_1, C_1, \ldots, C_1)$ and define $t_{0, 0, 0, k}^{(-2)} = t'{_{0, 0, 0, k}^{(-4)}}$ ($k=3,4,\ldots$) recursively. Therefore
\begin{eqnarray}
t_{0, 0, 0, k}^{(-2)} = t'{_{0, 0, 0, k}^{(-4)}}= \frac{t_{0, 0, 0, k+1}^{(-4)}t_{0, 0, 0, k-1}^{(-4)}+t_{0, 0, k, 0}^{(-2)}}{t_{0, 0, 0, k}^{(-4)}}, \quad k = 3, 4, \ldots
\end{eqnarray}
Now we finish the mutation of the first $C_1$ in $(C_1, C_1, \ldots, C_1)$.

We start to mutate the second $C_1$ in $(C_1, C_1, \ldots, C_1)$. We mutate the first vertex of the second $C_1$ in $(C_1, C_1, \ldots, C_1)$ and define $t_{0, 0, 1, 1}^{(-2)} = {t'}_{0, 0, 0, 1}^{(-2)}$, we obtain a quiver ($Q_{21}$). Therefore
\begin{eqnarray}\label{eqn13}
t_{0, 0, 1, 1}^{(-2)} = {t'}_{0, 0, 0, 1}^{(-2)} = \frac{t_{0, 0, 0, 2}^{(-2)}t_{0, 0, 1, 0}^{(-2)}+t_{0, 0, 2, 0}^{(-2)}}{t_{0, 0, 0, 1}^{(-2)}}.
\end{eqnarray}
We mutate the second vertex of the second $C_{1}$ in $(C_1, C_1, \ldots, C_1)$ and define $t_{0, 0, 1, 2}^{(-2)} = t'{_{0, 0, 0, 2}^{(-2)}}$, the quiver ($Q_{21}$) becomes a quiver ($Q_{22}$). Therefore
\begin{eqnarray}\label{eqn14}
t_{0, 0, 1, 2}^{(-2)} = t'{_{0, 0, 0, 2}^{(-2)}} = \frac{t_{0, 0, 0, 3}^{(-2)}t_{0, 0, 1, 1}^{(-2)}+t_{0, 0, 3, 0}^{(-2)}}{t_{0, 0, 0, 2}^{(-2)}}.
\end{eqnarray}
We continue this procedure and mutate the vertices of the second $C_1$ in $(C_1, C_1, \ldots, C_1)$ and define $t_{0, 0, 1, k}^{(-2)} = t'{_{0, 0, 0, k}^{(-2)}}$ ($k=3,4,\ldots$) recursively. Therefore
\begin{eqnarray}\label{eqn15}
t_{0, 0, 1, k}^{(-2)} = t'{_{0, 0, 0, k}^{(-2)}}= \frac{t_{0, 0, 0, k+1}^{(-2)}t_{0, 0, 1, k-1}^{(-2)}+t_{0, 0, k+1, 0}^{(-2)}}{t_{0, 0, 0, k}^{(-2)}}, \quad k = 3, 4, \ldots
\end{eqnarray}
Now we finish the mutation of the second $C_1$ in $(C_1, C_1, \ldots, C_1)$.

We start to mutate the third $C_1$ in $(C_1, C_1, \ldots, C_1)$. We mutate the first vertex of the third $C_1$ in $(C_1, C_1, \ldots, C_1)$ and define $t_{0, 0, 2, 1}^{(-2)} = {t'}_{0, 0, 1, 1}^{(-2)}$, we obtain a quiver ($Q_{31}$). Therefore
\begin{eqnarray}\label{eqn16}
t_{0, 0, 2, 1}^{(-2)} = {t'}_{0, 0, 1, 1}^{(-2)}= \frac{t_{0, 0, 1, 2}^{(-2)}t_{0, 0, 2, 0}^{(-2)}+t_{0, 0, 1, 0}^{(-2)}t_{0, 0, 3, 0}^{(-2)}}{t_{0, 0, 1, 1}^{(-2)}}.
\end{eqnarray}
We mutate the second vertex of the third $C_1$ in $(C_1, C_1, \ldots, C_1)$ and define $t_{0, 0, 2, 2}^{(-2)} = {t'}_{0, 0, 1, 2}^{(-2)}$, the quiver ($Q_{31}$) becomes a quiver ($Q_{32}$). Therefore
\begin{eqnarray}\label{eqn17}
t_{0, 0, 2, 2}^{(-2)} = {t'}_{0, 0, 1, 2}^{(-2)}=\frac{t_{0, 0, 1, 3}^{(-2)}t_{0, 0, 2, 1}^{(-2)}+t_{0, 0, 1, 0}^{(-2)}t_{0, 0, 4, 0}^{(-2)}}{t_{0, 0, 1, 2}^{(-2)}}.
\end{eqnarray}
We continue this procedure and mutate vertices of the third $C_1$ in $(C_1, C_1, \ldots, C_1)$ and define $t_{0, 0, 2, k}^{(-2)} = t'{_{0, 0, 1, k}^{(-2)}}$ ($k=3,4,\ldots$) recursively. Therefore
\begin{eqnarray}\label{eqn18}
t_{0, 0, 2, k}^{(-2)} = t'{_{0, 0, 1, k}^{(-2)}}=\frac{t_{0, 0, 1, k+1}^{(-2)} t_{0, 0, 2, k-1}^{(-2)}+ t_{0, 0, 1, 0}^{(-2)}t_{0, 0, k+2, 0}^{(-2)}}{t_{0, 0, 1, k}^{(-2)}}, \quad k = 3,4, \ldots
\end{eqnarray}
Now we finish the mutation of the third $C_1$ in $(C_1, C_1, \ldots, C_1)$. We continue this procedure and mutate the $(d+1)$-th $C_1$ $(d= 3, 5, \ldots, l)$ in order. We define $t_{0, 0, d, k}^{(-2)} = t'{_{0, 0, d-1, k}^{(-2)}}$, where $(0, 0, d, k)$=$\{(0, 0, 3, 1), (0, 0, 3, 2), (0, 0, 3, 3), (0, 0, 3, 4), \ldots; (0, 0, 4, 1), (0, 0, 4, 2), (0, 0, 4, 3)$,\\ $(0, 0, 4, 4) \ldots; (0, 0, 5, 1), (0, 0, 5, 2), (0, 0, 5, 3), (0, 0, 5, 4), \ldots ;  (0, 0, l, 1), (0, 0, l, 2), (0, 0, l, 3), (0, 0, l, \\ 4), \ldots \}$ recursively. Therefore
\begin{eqnarray}\label{eqn19}
t_{0, 0, d, k}^{(-2)} = t'{_{0, 0, d-1, k}^{(-2)}} = \frac{t_{0, 0, d-1, k+1}^{(-2)}t_{0, 0, d,k-1}^{(-2)}+t_{0, 0, d-1, 0}^{(-2)}t_{0, 0, k+d, 0}^{(-2)}}
{t{_{0, 0, d-1, k}^{(-2)}}}.
\end{eqnarray}

We write the definition of the variables in (\ref{cluster variables}) and the corresponding mutation sequences in Table \ref{mutation sequence of type F_{4}}.

In Table \ref{mutation sequence of type F_{4}}, we use
\[
\underset{m }{\underbrace{C_{i_1}, C_{i_2},\cdots, C_{i_u}, C_{i_1}, C_{i_2},\cdots, C_{i_u}, \cdots, C_{i_1}, C_{i_2},\cdots, C_{i_u}}},
\]
where $m \in \mathbb{Z}_{\geq 1}$, $u \in \mathbb{Z}_{\geq 1}$, to denote the mutation sequences
\[
C_{i_1}, C_{i_2},\cdots, C_{i_u}, C_{i_1}, C_{i_2},\cdots, C_{i_u}, \cdots, C_{i_1}, C_{i_2},\cdots, C_{i_u},
\]
where the number of $C_{i_{j}}$ $(j \in \{1, 2, \cdots, u\})$ is $m$.

\begin{table}[H] \resizebox{.6\width}{.6\height}{
\begin{tabular}{|c|c|c|}
\hline %
Mutation sequences & Definition of variables in (\ref{cluster variables}) and mutation equations    \\
\hline %
$\substack{(\underset{l+1}{\underbrace{C_{1}, C_{1}, \ldots,  C_{1}}})}$
&$\substack{ t_{0, 0, l, k}^{(-2)} = t'{_{0, 0, l-1, k}^{(-2)}}=\frac{t_{0, 0, l-1, k+1}^{(-2)}t_{0, 0, l, k-1}^{(-2)}+t_{0, 0, l-1, 0}^{(-2)}t_{0, 0, k+l, 0}^{(-2)}}{t_{0, 0, l-1, k}^{(-2)}}\quad (1) }$\\
\hline%
$\substack{(\underset{m+1}{\underbrace{C_{5}, C_{5}, \ldots, C_{5}}})\\
(\underset{m}{\underbrace{C_{6}, C_{6}, \ldots, C_{6}}})}$
& $\substack{\widetilde{t}_{n, m, 0, 0}^{(-4n-4m)}=\widetilde{t'}_{n, m-1, 0, 0}^{(-4n-4m+4)}= \frac{\widetilde{t}_{n+1, m-1, 0, 0}^{(-4n-4m)} \widetilde{t}_{n-1, m, 0, 0}^{(-4n-4m+4)}+\widetilde{t}_{0, m-1, 0, 0}^{(-4m+4)} \widetilde{t}_{0, n+m, 0, 0}^{(-4n-4m)}}{\widetilde{t}_{n, m-1, 0, 0}^{(-4n-4m+4)}} \quad (2) \\
\widetilde{t}_{n, m, 0, 0}^{(-4n-4m+2)} = \widetilde{t}'{_{n, m-1, 0, 0}^{(-4n-4m+6)}}=
\frac{\widetilde{t}_{n+1, m-1, 0, 0}^{(-4n-4m+2)} \widetilde{t}_{n-1, m, 0, 0}^{(-4n-4m+6)}+\widetilde{t}_{0, m-1, 0, 0}^{(-4m+6)} \widetilde{t}_{0, n+m, 0, 0}^{(-4n-4m+2)}}{\widetilde{t}{_{n, m-1, 0, 0}^{(-4n-4m+6)}}} \quad (3) }$\\
\hline%
$\substack{(C_5, \underset{\frac{l+1}{2},\ l \text{ is odd}}{\underbrace{C_{4}, C_{5}, \ldots, C_{4}, C_{5}}}) \\
(\underset{\frac{l+2}{2},\ l \text{ is even} }{\underbrace{C_{3}, C_{6}, \ldots, C_{3}, C_{6}}})}$
&$\substack{\widetilde{t}_{n, 0, 1, 0}^{(-4n-4)} = \widetilde{t'}{_{n, 0, 0, 0}^{(-4n)}}= \frac{\widetilde{t}_{n+1, 0, 0, 0}^{(-4n-4)} \widetilde{t}_{n-1, 0, 1, 0}^{(-4n)}+\widetilde{t}_{0, n, 1, 0}^{(-4n-4)}}{\widetilde{t}_{n, 0, 0, 0}^{(-4n)}} \quad (4)\\
\widetilde{t}_{n, 0, l, 0}^{(-4n-2l-2)} = \widetilde{t'}_{n, 0, l-2, 0}^{(-4n-2l+2)}= \frac{\widetilde{t}_{n+1, 0, l-2, 0}^{(-4n-2l-2)} \widetilde{t}_{n-1, 0, l, 0}^{(-4n-2l+2)}+\widetilde{t}_{0, n, l, 0}^{(-4n-2l-2)} \widetilde{t}_{0, 0, l-2, 0}^{(-2l+2)}}{\widetilde{t}_{n, 0, l-2, 0}^{(-4l-2l+2)}} \quad (5)\\
\widetilde{t}_{0, m, 1, 0}^{(-4m-4)} = \widetilde{t'}_{0, m, 0, 0}^{(-4m)}= \frac{\widetilde{t}_{0, m+1, 0, 0}^{(-4m-4)} \widetilde{t}_{0, m-1, 1, 0}^{(-4m)}+\widetilde{t}_{0, 0, 2m+1, 0}^{(-4m-4)} \widetilde{t}_{m, 0, 0, 0}^{(-4m)}}{\widetilde{t}_{0, m, 0, 0}^{(-4m)}} \quad (6)\\
\widetilde{t}_{0, m, l, 0}^{(-4m-2l-2)} = \widetilde{t'}_{0, m, l-2, 0}^{(-4m-2l+2)}= \frac{\widetilde{t}_{0, m+1, l-2, 0}^{(-4m-2l-2)} \widetilde{t}_{0, m-1, l, 0}^{(-4m-2l+2)}+\widetilde{t}_{0, 0, 2m+l, 0}^{(-4m-2l-2)} \widetilde{t}_{m, 0, l-2, 0}^{(-4m-2l+2)}}{\widetilde{t}_{0, m, l-2, 0}^{(-4m-2l+2)}} \quad (7)}$\\
\hline%
$\substack{(C_{5}, \underset{\frac{l+1}{2},\ l \text{ is odd}}{\underbrace{C_{4}, C_{5}, \ldots, C_{4}, C_{5}}}, \underset{m}{\underbrace{C_{5}, \ldots, C_{5}}})\\
(\underset{\frac{l+2}{2},\ l \text{ is even}}{\underbrace{C_{3}, C_{6}, \ldots, C_{3}, C_{6}}}, \underset{m}{\underbrace{C_6, \ldots, C_6}})}$
& $\substack{\widetilde{t}_{n, m, l, 0}^{(-4n-4m-2l-2)}=\widetilde{t'}_{n, m-1, l, 0}^{(-4n-4m-2l+2)}
=\frac{\widetilde{t}_{n+1, m-1, l, 0}^{(-4n-4m-2l-2)} \widetilde{t}_{n-1, m, l, 0}^{(-4n-4m-2l+2)}
+\widetilde{t}_{0, n+m, l, 0}^{(-4n-4m-2l-2)} \widetilde{t}_{0, m-1, l, 0}^{(-4m-2l+2)}}{\widetilde{t}_{n, m-1, l, 0}^{(-4n-4m-2l+2)}} \quad (8)}$\\
\hline%
$\substack{ (C_3, C_6, C_2, C_5, C_4, C_5) \\
\\  (C_3, C_6, C_2, C_5, C_4, C_5, C_2, C_3, C_6) }$
&$\substack{\widetilde{t}_{n,0,0,k}^{(-4n-2k-2)}=\widetilde{t'}_{n,0,0,0}^{(-4n-2k+2)}=\frac{\widetilde{t}_{n-1,0,0,k}
^{(-4n-2k+2)}\widetilde{t}_{n+1,0,0,0}^{(-4n-2k-2)}+\widetilde{t}_{0,n,0,k}^{(-4n-2k-2)}}{\widetilde{t}_{n,0,0,0}^{(-4n-2k+2)}},
\quad (9) \\
\widetilde{t}_{0,m,0,k}^{(-4m-2k-2)}=\widetilde{t'}_{0,m,0,0}^{(-4m-2k+2)} =\frac{\widetilde{t}_{0,m-1,0,k}^{(-4m-2k+2)} \widetilde{t}_{0,m+1,0,0}^{(-4m-2k-2)}+\widetilde{t}_{0,0,2m,k}^{(-4m-2k-2)}\widetilde{t}_{m,0,0,0}^{(-4m-2k+2)}}
{\widetilde{t}_{0,m,0,0}^{(-4m-2k+2)}},\ k=1,2, \quad (10)}$\\
\hline%
$\substack{ (C_3, C_6, C_2, \underset{\frac{l+1}{2},\ l \text{ is odd}}{\underbrace{C_3, C_6, \ldots, C_3, C_6}}) \\
(C_3, C_6, C_2, C_5, C_4, C_5, \underset{\frac{l+2}{2},\ l \text{ is even}}{\underbrace{C_4, C_5, \ldots, C_4, C_5}})\\\\
(C_3, C_6, C_2, C_5, C_4, C_5, C_2, \underset{\frac{l+1}{2},\ l \text{ is odd}}{\underbrace{C_4, C_5,\ldots, C_4, C_5}})\\
(C_3, C_6, C_2, C_5, C_4, C_5, C_2, C_3, C_6, \underset{\frac{l}{2},\ l \text{ is even}}{\underbrace{C_3, C_6, \ldots, C_3, C_6}})}$
&$\substack{\widetilde{t}_{n,0,1,k}^{(-4n-2k-4)} = \widetilde{t'}_{n,0,0,k-1}^{(-4n-2k)}
= \frac{\widetilde{t}_{n-1,0,1,k}^{(-4n-2k)}\widetilde{t}_{n+1,0,0,k-1}^{(-4n-2k-4)}+\widetilde{t}_{0,n,1,k}^{(-4n-2k-4)}
\widetilde{t}_{0,0,0,k-1}^{(-2k)}}{\widetilde{t}_{n,0,0,k-1}^{(-4n-2n)}},\ k=1,2,\ l=1; \quad (11)\\
\widetilde{t}_{n,0,l,k}^{(-4n-2l-2k-2)}=\widetilde{t'}_{n,0,l-2,k}^{(-4n-2l-2k+2)}=\frac{\widetilde{t}_{n-1,0,l,k}^{(-4n-2l-2k+2)}
\widetilde{t}_{n+1,0,l-2,k}^{(-4n-2l-2k-2)}+\widetilde{t}_{0,n,l,k}^{(-4n-2l-2k-2)}\widetilde{t}_{0,0,l-2,k}^{(-2l-2k+2)}} {\widetilde{t}_{n,0,l-2,k}^{(-4n-2l-2n+2)}},\ k=1,2,\ l\geq 2,  \quad (12)\\
\widetilde{t}_{0,m,1,k}^{(-4m-2k-4)}=\widetilde{t'}_{0,m,0,k-1}^{(-4m-2k)}=\frac{\widetilde{t}_{0,m-1,1,k}^{(-4m-2k)}
\widetilde{t}_{0,m+1,0,k-1}^{(-4m-2k-4)}+\widetilde{t}_{0,0,2m+1,k}^{(-4m-2k-4)}\widetilde{t}_{m,0,0,k-1}^{(-4m-2k)}}
{\widetilde{t}_{0,m,0,k-1}^{(-4m-2k)}},\ k=1,2,\ l=1, \quad (13)\\
\widetilde{t}_{0,m,l,k}^{(-4m-2l-2k-2)}=\widetilde{t'}_{0,m,l-2,k}^{(-4m-2l-2k+2)}=\frac{\widetilde{t}_{0,m-1,l,k}^{(-4m-2l-2k+2)}
\widetilde{t}_{0,m+1,l-2,k}^{(-4m-2l-2k-2)}+\widetilde{t}_{0,0,2m+l,k}^{(-4m-2l-2k-2)}\widetilde{t}_{m,0,l-2,k}^{(-4m-2l-2k+2)}}
{\widetilde{t}_{0,m,l-2,k}^{(-4m-2l-2k+2)}},\  k=1,2,\ l\geq 2, \quad (14)}$\\
\hline%
$\substack{(C_3, C_6, C_2, C_5, C_4, C_5, \underset{\frac{m+1}{2},\ m \text{ is odd} }{\underbrace{C_5, \ldots, C_5}}) \\
(C_3, C_6, C_2, C_3, C_6, \underset{\frac{m}{2},\ m \text{ is even}}{\underbrace{C_6, \ldots, C_6}})\\
(C_3, C_6, C_2, C_5, C_4, C_5, C_2, C_3, C_6, \underset{\frac{m+1}{2},\ m \text{ is odd}}{\underbrace{C_6, \ldots, C_6}})\\
(C_3, C_6, C_2, C_5, C_4, C_5, C_2, C_4, C_5, \underset{\frac{m}{2},\ m \text{ is even}}{\underbrace{C_5, \ldots, C_5}})}$
&$\substack{\widetilde{t}_{n,m,0,k}^{(-4n-4m-2k-2)}=\widetilde{t'}_{n,m-1,0,k}^{(-4n-4m-2k+2)}
=\frac{\widetilde{t}_{n-1,m,0,k}^{(-4n-4m-2k+2)}\widetilde{t}_{n+1,m-1,0,k}^{(-4n-4m-2k-2)}+\widetilde{t}_{0,n+m,0,k}^{(-4n-4m-2k-2)}\widetilde{t}_{0,m-1,0,k}^{(-4m-2k+2)}}
{\widetilde{t}_{n,m-1,0,k}^{(-4n-4m-2k+2)}},\ k=1,2, \quad (15)}$\\
\hline%
$\substack{(C_3, C_6, C_2, \underset{\frac{l+1}{2},\ l \text{ is odd}}{\underbrace{C_3, C_6, \ldots, C_3, C_6}}, \underset{m}{\underbrace{C_6, \ldots, C_6}})\\
(C_3, C_6, C_2, C_5, C_4, C_5, \underset{\frac{l}{2},\ l \text{ is even}}{\underbrace{C_4, C_5, \ldots, C_4, C_5}}, \underset{m}{\underbrace{C_5, \ldots, C_5}})\\
(C_3, C_6, C_2, C_5, C_4, C_5, C_2, C_4, C_5, \underset{\frac{l+1}{2},\ l \text{ is odd}}{\underbrace{C_4, C_5, \ldots, C_4, C_5}},  \underset{m}{\underbrace{C_5, \ldots, C_5}})\\
(C_3, C_6, C_2, C_5, C_4, C_5, C_2, C_3, C_6, \underset{\frac{l}{2},\ l \text{ is even}}{\underbrace{C_3, C_6, \ldots, C_3, C_6}}, \underset{m}{\underbrace{C_6, \ldots, C_6}})}$
& $\substack{\widetilde{t}_{n,m,l,k}^{(-4n-4m-2l-2k-2)}=\widetilde{t'}_{n,m-1,l,k}^{(-4n-4m-2l-2k+2)}
=\frac{\widetilde{t}_{n+1,m-1,l,k}^{(-4n-4m-2l-2k-2)}\widetilde{t}_{n-1,m,l,k}^{(-4n-4m-2l-2k+2)}+
\widetilde{t}_{0,n+m,l,k}^{(-4n-4m-2l-2k-2)}\widetilde{t}_{0,m-1,l,k}^{(-4m-2l-2k+2)}}
{\widetilde{t}_{n,m-1,l,k}^{(-4n-4m-2l-2k+2)}},\ k=1,2, \quad (16)}$\\
\hline%
$\substack{(C_{3}, C_{6}, C_{2}, C_{5}, C_{4}, C_{5}, C_{2}, \\
\underset{ k-2 }{\underbrace{C_{3}, C_{6}, C_{2}, C_{4}, C_{5}, C_{2}, \cdots, C_{3}, C_{6}, C_{2}, C_{4}, C_{5}, C_{2}}})}$
&$\substack{\widetilde{p}^{(-2k-4)}_{0,0,1,k}=\widetilde{p'}^{(-2k-2)}_{0,0,1,k-1}=\frac{\widetilde{t}^{(-2k+6)}_{0,0,0,k-4}
\widetilde{t}^{(-2k+2)}_{0,0,0,k-2}\widetilde{t}^{(-2k-2)}_{0,0,0,k}\widetilde{p}^{(-2k-4)}_{0,0,2,k-1}+\widetilde{t}^{(-2k+8)}_{0,0,0,k-5}
\widetilde{t}^{(-2k+4)}_{0,0,0,k-3} \widetilde{t}^{(-2k)}_{0,0,0,k-1} \widetilde{t}^{(-2k-4)}_{0,0,0,k+1}\widetilde{p}^{(-2k-2)}_{0,1,0,k-2}}
{\widetilde{p}^{(-2k-2)}_{0,0,1,k-1}}, \ k\geq3,\ l=1, \quad (17)\\
\widetilde{p}^{(-2k-2l-2)}_{0,0,l,k}=\widetilde{p'}^{(-2k-2l)}_{0,0,l,k-1}=\frac{\widetilde{p}^{(-2k-2l)}_{0,0,l-1,k}
\widetilde{p}^{(-2k-2l-2)}_{0,0,l+1,k-1} + \widetilde{t}^{(-2k+8)}_{0,0,0,k-5} \widetilde{t}^{(-2k-2l-2)}_{0,0,0,k+l} \widetilde{p}^{(-2k-2l)}_{0,\frac{l+1}{2},0,k-2}\widetilde{p}^{(-2k-2l+2)}_{0,\frac{l-1}{2},0,k-1}}
{\widetilde{p}^{(-2k-2l)}_{0,0,l,k-1}}, \  k\geq3,\ l \text{ is odd}, \ l\geq3, \quad (18)\\
\widetilde{p}^{(-2k-2l-2)}_{0,0,l,k}=\widetilde{p'}^{(-2k-2l)}_{0,0,l,k-1}=\frac{\widetilde{p}^{(-2k-2l)}_{0,0,l-1,k}
\widetilde{p}^{(-2k-2l-2)}_{0,0,l+1,k-1}+ \widetilde{t}^{(-2k+8)}_{0,0,0,k-5} \widetilde{t}^{(-2k-2l-2)}_{0,0,0,k+l} \widetilde{p}^{(-2k-2l+2)}_{0, \frac{l}{2},0,k-2} \widetilde{p}^{(-2k-2l)}_{0,\frac{l}{2},0,k-1}} {\widetilde{p}^{(-2k-2l)}_{0,0,l,k-1}}, \ k\geq3,\ l \text{ is even}, \ l \geq 2, \quad (19)}$\\
\hline%
$\substack{(C_{3}, C_{6}, C_{2}, C_{5}, C_{4}, C_{5}, \\
\underset{\frac{k-1}{2},\ k \text{ is odd}}{\underbrace{ C_{2}, C_{3}, C_{6}, C_{2}, C_{4}, C_{5}, \cdots, C_{2}, C_{3}, C_{6}, C_{2}, C_{4}, C_{5} }}) \\
(C_{3}, C_{6}, C_{2}, C_{5}, C_{4}, C_{5}, C_{2}, C_{3}, C_{6}, \\
\underset{\frac{k-2}{2},\ k \text{ is even }}{\underbrace{C_{2}, C_{4}, C_{5}, C_{2}, C_{3}, C_{6}, \cdots, C_{2}, C_{4}, C_{5}, C_{2}, C_{3}, C_{6}}})}$
&$\substack{\widetilde{p}^{(-2k-6)}_{0,1,0,k} = \widetilde{p'}^{(-2k-2)}_{0,1,0,k-2} = \frac{\widetilde{t}^{(-2k+2)}_{0,0,0,k-2} \widetilde{t}^{(s+4)}_{0,0,0,k} \widetilde{p}^{(-2k-6)}_{0,2,0,k-2}+\widetilde{p}^{(-2k-6)}_{0,0,2,k} \widetilde{t}^{(-2k-2)}_{1,0,0,k-2}}
{\widetilde{p}^{(-2k-2)}_{0,1,0,k-2}}, \ k\geq 3,\ m=1, \quad (20)\\
\widetilde{p}^{(-4m-2k-2)}_{0,m,0,k} =\widetilde{p'}^{(-4m-2k+2)}_{0,m,0,k-2}
=\frac{\widetilde{p}^{(-4m-2k+2)}_{0,m-1,0,k}\widetilde{p}^{(-4m-2k-2)}_{0,m+1,0,k-2} +
\widetilde{p}^{(-4m-2k-2)}_{0,0,2m,k} \widetilde{t}^{(-4m-2k+2)}_{m,0,0,k-2}}
{\widetilde{p}^{(-4m-2k+2)}_{0,m,0,k-2}}, \ k\geq 3, \ m\geq 2, \quad (21)\\
\widetilde{t}^{(-4n-2k-2)}_{n,0,0,k} = \widetilde{t'}^{(-4n-2k+2)}_{n,0,0,k-2}= \frac{\widetilde{t}^{(-4n-2k+2)}_{n-1,0,0,k} \widetilde{t}^{(-4n-2k-2)}_{n+1,0,0,k-2} + \widetilde{p}^{(-4n-2k-2)}_{0,n,0,k}}{\widetilde{t}^{(-2k-2n+2)}_{n,0,0,k-2}}, \ k\geq3. \quad (22)}$\\
\hline%
\end{tabular}}
\caption{Mutation sequences.}
\label{mutation sequence of type F_{4}}
\end{table}

\subsection{The equations in the system of type $F_4$ correspond to mutations in the cluster algebra $\mathscr{A}$}

Equation (1) corresponds to Equation (\ref{eqn 1}) in Theorem \ref{M-system of type F4}, Equations (2) and (3) correspond to Equation (\ref{eqn 2}) in Theorem \ref{M-system of type F4}. Equations (4)--(22) correspond to Equations (\ref{eqn 3a})--(\ref{eqn 19}) in Theorem \ref{M-system of type F4} respectively. Therefore we have the following theorem.

\begin{theorem}\label{connection with cluster algebra 1}
Each minimal affinizations in Theorem \ref{M-system of type F4} corresponds to a cluster variable in $\mathscr{A}$ defined in Section \ref{definition of cluster algebra A}.
\end{theorem}

\section{The dual system of Theorem \ref{M-system of type F4}} \label{dual M system}
In this section, we study the dual system of Theorem \ref{M-system of type F4}.

\begin{theorem}[Theorem 3.9, \cite{Her07}] \label{her 07}
For $l, m, n\in \mathbb{Z}_{\geq 1}$, $s\in \mathbb{Z}$, the modules $\mathcal{T}_{n,m,0,0}^{(s)}$, $\mathcal{T}_{n,0,l,0}^{(s)}$, $\mathcal{T}_{0,m,l,0}^{(s)}$, $\mathcal{T}_{n,m,l,0}^{(s)}$ are anti-special.
\end{theorem}

We have the following theorem.

\begin{theorem}\label{anti-special}
The modules
\begin{align*}
&\mathcal{\widetilde{T}}_{0,0,l,k}^{(s)}, \  \mathcal{T}_{n,0,l,0}^{(s)}, \  \mathcal{T}_{0,m,l,0}^{(s)},
\ \mathcal{T}_{n,m,0,0}^{(s)},
\ \mathcal{T}_{n,m,l,0}^{(s)},
\ \mathcal{T}_{n,0,0,k}^{(s)},
\mathcal{T}_{n,m,0,k}^{(s)}(k \leq 2),\\
&\mathcal{T}_{n,m,l,k}^{(s)}(k \leq 2),
\ \mathcal{T}_{n,0,l,k}^{(s)}(k \leq 2),
\ \mathcal{T}_{0,m,l,k}^{(s)}(k \leq 2),
\ \mathcal{P}^{(s)}_{0,0,l,k},
\ \mathcal{P}^{(s)}_{0,m,0,k},
\end{align*}
where $k, l, m, n\in \mathbb{Z}_{\geq 1}$, $s\in \mathbb{Z}$, are anti-special.
\end{theorem}
\begin{proof}
The proof of the theorem follows from dual arguments in the proof of Theorem \ref{special}.
\end{proof}

\begin{lemma}\label{lemma2}
Let $\iota: \mathbb{Z}\mathcal{P}\rightarrow \mathbb{Z}\mathcal{P}$ be a homomorphism of rings such that $Y_{1, aq^{s}}\mapsto Y_{1, aq^{18-s}}^{-1}$, $Y_{2, aq^{s}}\mapsto Y_{2, aq^{18-s}}^{-1}$, $Y_{3, aq^{s}}\mapsto Y_{3, aq^{18-s}}^{-1}$, $Y_{4, aq^{s}}\mapsto Y_{4, aq^{18-s}}^{-1}$ for all $a\in \mathbb{C}^{\times}, s\in \mathbb{Z}$. Then
$$\chi_{q}(\widetilde{\mathcal T}_{k, l, m, n}^{(s)})=\iota(\chi_{q}(\mathcal T_{k, l, m, n}^{(s)})), \quad \chi_{q}(\mathcal T_{k, l, m, n}^{(s)})=\iota(\chi_{q}(\widetilde{\mathcal T}_{k, l, m, n}^{(s)})).$$
\end{lemma}
\begin{proof}
The proof is similar to Lemma 7.3 in \cite{LM13}.
\end{proof}

\begin{theorem}\label{The dual M-system}
For $s\in \mathbb{Z}$, $k, l, m, n \in \mathbb{Z}_{\geq 0}$, we have the following system of equations.
\begin{align}\label{eqn 100}
[\mathcal{\widetilde{T}}_{0,0,l-1,k}^{(-2)}][\mathcal{\widetilde{T}}_{0,0,l,k}^{(-2)}]=[\mathcal{\widetilde{T}}_{0,0,l,k-1}^{(-2)}][\mathcal{\widetilde{T}}_{0,0,l-1,k+1}^{(-2)}]
+[\mathcal{\widetilde{T}}_{0,0,k+l,0}^{(-2)}][\mathcal{\widetilde{T}}_{0,0,l-1,0}^{(-2)}],
\end{align}
\begin{align}\label{eqn 102}
[\mathcal{T}_{n,m-1,0,0}^{(s+4)}][\mathcal{T}_{n,m,0,0}^{(s)}]=[\mathcal{T}_{n-1,m,0,0}^{(s+4)}][\mathcal{T}_{n+1,m-1,0,0}^{(s)}]+[\mathcal{T}_{0,m-1,0,0}^{(s+4n+4)}]
[\mathcal{T}_{0,n+m,0,0}^{(s)}],
\end{align}

\begin{align}\label{eqn 103a}
[\mathcal{T}_{n,0,0,0}^{(s+4)}][\mathcal{T}_{n,0,1,0}^{(s)}]&=[\mathcal{T}_{n-1,0,1,0}^{(s+4)}][\mathcal{T}_{n+1,0,0,0}^{(s)}]+[\mathcal{T}_{0,n,1,0}^{(s)}],
\end{align}

\begin{align}\label{eqn 103b}
[\mathcal{T}_{n,0,l-2,0}^{(s+4)}][\mathcal{T}_{n,0,l,0}^{(s)}]&=[\mathcal{T}_{n-1,0,l,0}^{(s+4)}][\mathcal{T}_{n+1,0,l-2,0}^{(s)}]+[\mathcal{T}_{0,0,l-2,0}^{(s+4n+4)}]
[\mathcal{T}_{0,n,l,0}^{(s)}],\  l \geq 2,
\end{align}

\begin{align}\label{eqn 104a}
[\mathcal{T}_{0,m,0,0}^{(s+4)}][\mathcal{T}_{0,m,1,0}^{(s)}]&=[\mathcal{T}_{0,m-1,1,0}^{(s+4)}][\mathcal{T}_{0,m+1,0,0}^{(s)}]+[\mathcal{T}_{m,0,0,0}^{(s+4)}]
[\mathcal{T}_{0,0,1+2m,0}^{(s)}],
\end{align}

\begin{align}\label{eqn 104b}
[\mathcal{T}_{0,m,l-2,0}^{(s+4)}][\mathcal{T}_{0,m,l,0}^{(s)}]&=[\mathcal{T}_{0,m-1,l,0}^{(s+4)}][\mathcal{T}_{0,m+1,l-2,0}^{(s)}]+[\mathcal{T}_{m,0,l-2,0}^{(s+4)}]
[\mathcal{T}_{0,0,l+2m,0}^{(s)}],\ l \geq 2
\end{align}

\begin{align}\label{eqn 105}
[\mathcal{T}_{n,m-1,l,0}^{(s+4)}][\mathcal{T}_{n,m,l,0}^{(s)}]=[\mathcal{T}_{n-1,m,l,0}^{(s+4)}][\mathcal{T}_{n+1,m-1,l,0}^{(s)}]+[\mathcal{T}_{0,m-1,l,0}^{(s+4n+4)}]
[\mathcal{T}_{0,n+m,l,0}^{(s)}],
\end{align}

\begin{align}\label{eqn 106a}
[\mathcal{T}_{n,0,0,0}^{(s+4)}][\mathcal{T}_{n,0,0,k}^{(s)}]=[\mathcal{T}_{n-1,0,0,k}^{(s+4)}][\mathcal{T}_{n+1,0,0,0}^{(s)}]+[\mathcal{T}_{0,n,0,k}^{(s)}],
\ k=1, 2,
\end{align}

\begin{align}\label{eqn 107a}
[\mathcal{T}_{0,m,0,0}^{(s+4)}][\mathcal{T}_{0,m,0,k}^{(s)}]&=[\mathcal{T}_{0,m-1,0,k}^{(s+4)}][\mathcal{T}_{0,m+1,0,0}^{(s)}]+[\mathcal{T}_{0,0,2m,k}^{(s)}]
[\mathcal{T}_{m,0,0,0}^{(s+4)}], \ k=1, 2,
\end{align}

\begin{align}\label{eqn 108}
[\mathcal{T}_{n,0,0,k-1}^{(s+4)}][\mathcal{T}_{n,0,1,k}^{(s)}]=[\mathcal{T}_{n-1,0,1,k}^{(s+4)}][\mathcal{T}_{n+1,0,0,k-1}^{(s)}]+[\mathcal{T}_{0,n,1,k}^{(s)}],
\ k=1,2,
\end{align}

\begin{align}\label{eqn 109}
[\mathcal{T}_{0,m,0,k-1}^{(s+4)}][\mathcal{T}_{0,m,1,k}^{(s)}]=[\mathcal{T}_{0,m-1,1,k}^{(s+4)}][\mathcal{T}_{0,m+1,0,k-1}^{(s)}]+[\mathcal{T}_{0,0,1+2m,k}^{(s)}][\mathcal{T}_{m,0,0,k-1}^{(s+4)}],
\ k=1,2,
\end{align}

\begin{align}\label{eqn 110}
&[\mathcal{T}_{n,0,l-2,k}^{(s+4)}][\mathcal{T}_{n,0,l,k}^{(s)}]=[\mathcal{T}_{n-1,0,l,k}^{(s+4)}][\mathcal{T}_{n+1,0,l-2,k}^{(s)}]+[\mathcal{T}_{0,n,l,k}^{(s)}]
[\mathcal{T}_{0,0,l-2,k}^{(s+4n+4)}],\ k=1,2,\ l \geq 2,
\end{align}
\begin{align}\label{eqn 111}
&[\mathcal{T}_{0,m,l-2,k}^{(s+4)}][\mathcal{T}_{0,m,l,k}^{(s)}]=[\mathcal{T}_{0,m-1,l,k}^{(s+4)}][\mathcal{T}_{0,m+1,l-2,k}^{(s)}]+[\mathcal{T}_{0,0,l+2m,k}^{(s)}]
[\mathcal{T}_{m,0,l-2,k}^{(s+4)}], \ k=1,2,\ l \geq 2,
\end{align}

\begin{align}\label{eqn 112}
[\mathcal{T}_{n,m-1,0,k}^{(s+4)}][\mathcal{T}_{n,m,0,k}^{(s)}]=[\mathcal{T}_{n-1,m,0,k}^{(s+4)}][\mathcal{T}_{n+1,m-1,0,k}^{(s)}]+[\mathcal{T}_{0,n+m,0,k}^{(s)}]
[\mathcal{T}_{0,m-1,0,k}^{(s+4)}], \ k=1,2,
\end{align}

\begin{align}\label{eqn 113}
[\mathcal{T}_{n,m-1,l,k}^{(s+4)}][\mathcal{T}_{n,m,l,k}^{(s)}]=[\mathcal{T}_{n-1,m,l,k}^{(s+4)}][\mathcal{T}_{n+1,m-1,l,k}^{(s)}]+[\mathcal{T}_{0,n+m,l,k}^{(s)}]
[\mathcal{T}_{0,m-1,l,k}^{(s+4n+4)}],\ k=1,2,
\end{align}

\begin{align}
\begin{split}
[\mathcal{P}^{(s+2)}_{0,0,1,k-1}][\mathcal{P}^{(s)}_{0,0,1,k}]&=[\mathcal{T}^{(s+10)}_{0,0,0,k-4}][\mathcal{T}^{(s+6)}_{0,0,0,k-2}][\mathcal{T}^{(s+2)}_{0,0,0,k}][\mathcal{P}^{(s)}_{0,0,2,k-1}]\\
& \quad + [\mathcal{T}^{(s+12)}_{0,0,0,k-5}][\mathcal{T}^{(s+8)}_{0,0,0,k-3}][\mathcal{T}^{(s+4)}_{0,0,0,k-1}]
[\mathcal{T}^{(s)}_{0,0,0,k+1}][\mathcal{P}^{(s+2)}_{0,1,0,k-2}],\  k \geq 3,
\end{split}
\end{align}

\begin{gather}
\begin{align}
[\mathcal{P}^{(s+2)}_{0,0,l,k-1}][\mathcal{P}^{(s)}_{0,0,l,k}]=[\mathcal{P}^{(s+2)}_{0,0,l-1,k}][\mathcal{P}^{(s)}_{0,0,l+1,k-1}]+[\mathcal{T}^{(s+2l+10)}_{0,0,0,k-5}][\mathcal{T}^{(s)}_{0,0,0,k+l}] [\mathcal{P}^{(s+2)}_{0,\frac{l+1}{2},0,k-2}][\mathcal{P}^{(s+4)}_{0,\frac{l-1}{2},0,k-1}],
\end{align}
\end{gather}
where $k \geq 3, \text{ l is odd}, \ l\geq 3$,

\begin{gather}
\begin{align}
[\mathcal{P}^{(s+2)}_{0,0,l,k-1}][\mathcal{P}^{(s)}_{0,0,l,k}]=[\mathcal{P}^{(s+2)}_{0,0,l-1,k}][\mathcal{P}^{(s)}_{0,0,l+1,k-1}]+[\mathcal{T}^{(s+2l+10)}_{0,0,0,k-5}][\mathcal{T}^{(s)}_{0,0,0,k+l}] [\mathcal{P}^{(s+4)}_{0,\frac{l}{2},0,k-2}][\mathcal{P}^{(s+2)}_{0,\frac{l}{2},0,k-1}],
\end{align}
\end{gather}
where $k\geq3, \text{ l is even}, \ l \geq 2$,

\begin{gather}
\begin{align}
[\mathcal{P}^{(s+4)}_{0,1,0,k-2}][\mathcal{P}^{(s)}_{0,1,0,k}]= [\mathcal{T}^{(s+8)}_{0,0,0,k-2}][\mathcal{T}^{(s+4)}_{0,0,0,k}][\mathcal{P}^{(s)}_{0,2,0,k-2}]+[\mathcal{P}^{(s)}_{0,0,2,k}][\mathcal{T}^{(s+4)}_{1,0,0,k-2}],
\ k \geq 3,
\end{align}
\end{gather}

\begin{gather}
\begin{align}
[\mathcal{P}^{(s+4)}_{0,m,0,k-2}][\mathcal{P}^{(s)}_{0,m,0,k}]= [\mathcal{P}^{(s+4)}_{0,m-1,0,k}][\mathcal{P}^{(s)}_{0,m+1,0,k-2}]+[\mathcal{P}^{(s)}_{0,0,2m,k}][\mathcal{T}^{(s+4)}_{m,0,0,k-2}], \ k\geq 3, \ m\geq 2,
\end{align}
\end{gather}

\begin{gather}
\begin{align}
[\mathcal{T}^{(s+4)}_{n,0,0,k-2}][\mathcal{T}^{(s)}_{n,0,0,k}]&=[\mathcal{T}^{(s+4)}_{n-1,0,0,k}][\mathcal{T}^{(s)}_{n+1,0,0,k-2}]+[\mathcal{P}^{(s)}_{0,n,0,k}], \ k \geq 3.
\end{align}
\end{gather}

Moreover, every module in the summands on the right hand side of above equation is simple.
\end{theorem}
\begin{proof}

The lowest weight monomial of $ \chi_q(\mathcal{\widetilde{T}}_{n, m, l, k}^{(s)}) $ is obtained from the highest weight monomial of $ \chi_q(\mathcal{\widetilde{T}}_{n, m, l, k}^{(s)}) $ by the substitutions: $1_s \mapsto 1^{-1}_{18+s}$, $2_s \mapsto 2^{-1}_{18+s}$, $3_s \mapsto 3^{-1}_{18+s}$, $4_s \mapsto 4^{-1}_{18+s}$. After we apply $\iota$ to $\chi_{q}(\mathcal{\widetilde{T}}_{n, m, l, k}^{(s)})$, the lowest weight monomial of $ \chi_q(\mathcal{\widetilde{T}}_{n, m, l, k}^{(s)}) $ becomes the highest weight monomial of $\iota(\chi_{q}(\mathcal{\widetilde{T}}_{n, m, l, k}^{(s)}))$. Therefore the highest weight monomial of $\iota(\chi_{q}(\mathcal T_{n, m, l, k}^{(s)}))$ is obtained from the lowest weight monomial of $ \chi_q(\mathcal{\widetilde{T}}_{n, m, l, k}^{(s)}) $ by the substitutions: $1_s \mapsto 1^{-1}_{18-s}$, $2_s \mapsto 2^{-1}_{18-s}$, $3_s \mapsto 3^{-1}_{18-s}$, $4_s \mapsto 4^{-1}_{18-s}$. It follows that the highest weight monomial of $\iota(\chi_{q}(\mathcal{\widetilde{T}}_{n, m, l, k}^{(s)}))$ is obtained from the highest weight monomial of $ \chi_q(\mathcal{\widetilde{T}}_{n, m, l, k}^{(s)}) $ by the substitutions: $1_s \mapsto 1_{-s}$, $2_s \mapsto 2_{-s}$, $3_s \mapsto 3_{-s}$, $4_s \mapsto 4_{-s}$. Therefore the dual system is obtained applying $\iota$ to both sides of every equation of the system in Theorem \ref{M-system of type F4}.

The simplify of every module in the summands on the right hand side of every equation follows from Theorem $\ref{irreducible}$ and Lemma $\ref{lemma2}$.
\end{proof}
\begin{example}
The following are some equations in the system in Theorem \ref{The dual M-system}.
\begin{gather}
\begin{align*}
&[1_{2}][1_{4}2_{1}]=[1_{4}1_{2}][2_{1}]+ [2_{3}2_{1}], \\
&[1_{4}1_{2}][1_{6}1_{4}2_{1}]=[1_{4}2_{1}][1_{6}1_{4}1_{2}]+[2_{5}2_{3}2_{1}], \\
&[1_{6}1_{4}1_{2}][1_{8}1_{6}1_{4} 2_{1}]=[1_{6}1_{4}2_{1}][1_{8}1_{6}1_{4}1_{2}]+[2_{7}2_{5}2_{3}2_{1}],\\
&[3_{4}][1_{0}2_{3}3_{8}]=[1_{0}2_{3}][3_{4}3_{8}]+[1_{0}2_{3}2_{5}2_{7}][4_{6}],\\
&[3_{4}3_{8}][1_{0}2_{3}3_{8}3_{12}=[1_{0}2_{3}3_{8}][3_{4}3_{8}3_{12}]+[1_{0}2_{3}2_{5}2_{7}2_{9}2_{11}][4_{6}4_{10}],\\
&[3_{4}3_{8}3_{12}][1_{0}2_{3}3_{8}3_{12}3_{16}]=[1_{0}2_{3}3_{8}3_{12}][3_{4}3_{8}3_{12}3_{16}]+[1_{0}2_{3}2_{5}2_{7}2_{9}2_{11}2_{13}2_{15}][4_{6}  4_{10}4_{14}],\\
&[4_{6}][1_{0}2_{3}4_{10}]=[1_{0}2_{3}][4_{6}4_{10}]+[1_{0}2_{3}3_{8}],\\
&[4_{6}4_{10}][1_{0}2_{3}4_{10}4_{14}]=[1_{0}2_{3}4_{10}][4_{6}4_{10}4_{14}]+[1_{0}2_{3}3_{8}3_{12}],\\
&[4_{6}4_{10}4_{14}][1_{0}2_{3}4_{10}4_{14}4_{18}]=[1_{0}2_{3}4_{10}4_{14}][4_{6}4_{10}4_{14}4_{18}]+[1_{0}2_{3}3_{8}3_{12}3_{16}].
\end{align*}
\end{gather}
\end{example}
\subsection{The system in Theorem \ref{The dual M-system}}
By replacing each $[\mathcal{\widetilde{T}}_{n, m, l, k}^{(s)}]$ (resp. $[\mathcal{T}_{n, m, l, k}^{(s)}]$) in the system of Theorem \ref{The dual M-system} with  $\chi(\widetilde{\mathfrak{m}}_{n, m, l, k})$ (resp. $\chi(\mathfrak{m}_{n, m, l, k})$), we obtain a system of equations consisting of the characters of $U_{q}\mathfrak{g}$-modules. The following are two equations in the system.
\begin{equation*}
\begin{split}
&\chi(\widetilde{\mathfrak{m}}_{0,0,l-1,k})\chi(\widetilde{\mathfrak{m}}_{0,0,l,k})=\chi(\widetilde{\mathfrak{m}}_{0,0,l,k-1})\chi(\widetilde{\mathfrak{m}}_{0,0,l-1,k+1})+\chi(\widetilde{\mathfrak{m}}_{0,0,k+l,0})\chi(\widetilde{\mathfrak{m}}_{0,0,l-1,0}),\\
&\chi(\mathfrak{m}_{n,m-1,0,0})\chi(\mathfrak{m}_{n,m,0,0})=\chi(\mathfrak{m}_{n-1,m,0,0})\chi_(\mathfrak{m}_{n+1,m-1,0,0})+\chi(\mathfrak{m}_{0,m-1,0,0})\chi(\mathfrak{m}_{0,n+m,0,0}).
\end{split}
\end{equation*}

\subsection{Relation between the system in Theorem \ref{The dual M-system} and cluster algebras}
Let $I = \{ 1,2,$ $3,4 \}$ and

\begin{align*}
&S=\{2u \mid u\in \mathbb{Z}_{\geq 0}\},\\
&S'=\{2u+1 \mid u\in \mathbb{Z}_{\geq 0}\}.
\end{align*}
Let
\begin{align*}
V & =(\{1\}\times S) \cup (\{2\}\times S')\cup (\{3\}\times S)\cup (\{4\}\times S).
\end{align*}
A quiver $\widetilde{Q}$ with vertex set $V$ will be defined as follows. The arrows of $\widetilde{Q}$ are given by the following rule: there is an arrow from the vertex $(i,r)$ to the vertex $(j,s)$ if and only if $b_{ij}\neq 0$ and $s=r-b_{ij}+d_{i}-d_{j}$.

Let ${\bf \widetilde{t}} = {\bf \widetilde{t}}_1 \cup {\bf \widetilde{t}}_2$, where
\begin{align*}
{\bf \widetilde{t}_1}=\{\widetilde{t}_{0, 0, l, 0}^{(-2)},\ \widetilde{t}_{0, 0, 0, k}^{(-4)} \mid k, l \in \mathbb{Z}_{\geq 1}\}
\end{align*}
and
\begin{gather}
\begin{align*}
{\bf \widetilde{t}}_2=\{t_{n, 0, 0, 0}^{(-4n+4)},\ t_{0, m, 0, 0}^{(-4m)},\ t_{n, 0, 0, 0}^{(-4n+2)}, \ t_{0, m, 0, 0}^{(-4m+2)}, \ t_{0, 0, l, 0}^{(-2l-2)},\ t_{0, 0, 0, k}^{(2k-2)} \mid k, l, m, n \in \mathbb{Z}_{\geq 1}\}.
\end{align*}
\end{gather}

Let $\widetilde{\mathscr{A}}$ be the cluster algebra defined by the initial seed $({\bf \widetilde{t}}, \widetilde{Q})$.
By similar arguments in Section \ref{Relation between the M-systems and cluster algebras}, we have the following theorem.
\begin{theorem}\label{minimal affinizations correspond to cluster variablesII}
Every equation in the system in Theorem \ref{The dual M-system} corresponds to a mutation equation in the cluster algebra $\widetilde{\mathscr{A}}$. Every minimal affinization in the system in Theorem \ref{The dual M-system} corresponds to a cluster variable of the cluster algebra $\widetilde{\mathscr{A}}$.
\end{theorem}

\section{Proof of theorem \ref{special}}\label{proof of special}
In this section, we prove Theorem \ref{special}. Namely, we will prove that for $s \in \mathbb{Z}, k, l, m, n \in \mathbb{Z}_{\geq 1},$ the modules

\begin{equation}\label{special modules 1}
\begin{split}
\mathcal{T}_{0,0,l,k}^{(-2)}, \
\mathcal{\widetilde{T}}_{n,0,l,0}^{(s)}, \
\mathcal{\widetilde{T}}_{0,m,l,0}^{(s)}, \
\mathcal{\widetilde{T}}_{n,m,0,0}^{(s)}, \
\mathcal{\widetilde{T}}_{n,m,l,0}^{(s)}, \
\mathcal{\widetilde{T}}_{n,0,0,k}^{(s)}, \
\mathcal{\widetilde{T}}_{n,m,0,k}^{(s)} \ (k \leq 2), \\
\mathcal{\widetilde{T}}_{n,m,l,k}^{(s)} \ (k \leq 2), \
\mathcal{\widetilde{T}}_{n,0,l,k}^{(s)} \ (k \leq 2), \
\mathcal{\widetilde{T}}_{0,m,l,k}^{(s)} \ (k \leq 2), \
\widetilde{P}^{(s)}_{0,0,l,k}, \
\widetilde{P}^{(s)}_{0,m,0,k},
\end{split}
\end{equation}
are special. Since the modules
\begin{align*}
&\mathcal{T}_{0,0,0,k}^{(s)}, \ \mathcal{T}_{0,0,l,0}^{(s)}, \ \mathcal{T}_{0,m,0,0}^{(s)}, \ \mathcal{T}_{m,0,0,0}^{(s)}, \ \mathcal{\widetilde{T}}_{0,0,0,k}^{(s)}, \ \mathcal{\widetilde{T}}_{0,0,l,0}^{(s)}, \ \mathcal{\widetilde{T}}_{0,m,0,0}^{(s)}, \
\mathcal{\widetilde{T}}_{n,0,0,0}^{(s)},
\end{align*}
are Kirillov-Reshetikhin modules, they are special. By Theorem \ref{Her 07}, the modules $\mathcal{\widetilde{T}}_{n,m,0,0}^{(s)}$, $\mathcal{\widetilde{T}}_{n,0,l,0}^{(s)}$, $\mathcal{\widetilde{T}}_{0,m,l,0}^{(s)}$, $\mathcal{\widetilde{T}}_{n,m,l,0}^{(s)}$ are special.
In the following, we will prove that the other modules in (\ref{special modules 1}) are special. Without loss of generality, we may assume that $s=0$ in $\mathcal{\widetilde{T}}^{(s)}$, where $\mathcal{\widetilde{T}}$ is a module in (\ref{special modules 1}).

\subsection{The cases of $\mathcal{\widetilde{T}}_{n, 0, 0, k}^{(0)}$, $\mathcal{T}_{0,0,l,k}^{(-2)}$, $\widetilde{P}^{(s)}_{0,0,l,k}$, and $\widetilde{P}^{(s)}_{0,m,0,k}$ }\label{cases 6.1}

Since the proof of each case is similar to each other, we give a detailed proof of the case of $\mathcal{\widetilde{T}}_{n, 0, 0, k}^{(0)}$ as follows.

Let $m_+=\widetilde{T}_{n, 0, 0, k}^{(0)}$ with $n,k \in\mathbb{Z}_{\geq 1}$. Then
\begin{align*}
m_+=(4_{0} 4_{4}\cdots 4_{4n-4}) ( 1_{4n+4} 1_{4n+6} \cdots 1_{4n+2k+2} ).
\end{align*}

Suppose that $n=1$. Let $$U=I \times \{aq^s : s \in \mathbb{Z}, s \leq 2k+6 \}.$$

Since all monomials in $\mathscr{M}(\chi_q(m_+)-\text{trunc}_{m_+ \mathcal{Q}_{U}^{-}} \: \chi_q(m_+))$ are right-negative, it is sufficient to show that $\text{trunc}_{m_+ \mathcal{Q}_{U}^{-}} \: \chi_q(m_+)$ is special.

Let
\begin{gather}
\begin{align*}
\mathscr{M}= \{ m_0=m_+, m_1=m_0 A_{4,2}^{-1},  m_{2} = m_{1} A_{3, 4}^{-1},  m_{3} = m_{2} A_{2, 6}^{-1},  m_{4}=m_{3}A_{2, 4}^{-1},  m_{5} = m_{4}A_{3, 6}^{-1},  m_{6}=m_{5}A_{4, 8}^{-1} \}.
\end{align*}
\end{gather}

It is easy to see that $\mathscr{M}$ satisfies the conditions in Theorem \ref{truncated}. Therefore
\begin{align*}
\text{trunc}_{m_+ \mathcal{Q}_{U}^{-}} \: \chi_q(m_+)=\sum_{m\in \mathscr{M}} m
\end{align*}
and hence $\text{trunc}_{m_+ \mathcal{Q}_{U}^{-}} \: \chi_q(m_+)$ is special.

Suppose that $n\geq 2$. In the following, we write $m_+ = m_1'm_2' = m_1''m_2''$ for some monomials $m_1',m_2',m_1'',m_2''$. We will show that the only dominant monomial in $\mathscr{M}(\chi_q(m_1')\chi_q(m_2')) \cap \mathscr{M}(\chi_q(m_1'')\chi_q(m_2''))$ is $m_+$ which implies that $L(m_+)$ is special.

Let $m_+ = m_1'm_2'$, where
\begin{align*}
m'_1=4_{0} 4_{4}\cdots 4_{4n-8}, \ m'_2=4_{4n-4} 1_{4n+4} 1_{4n+6} \cdots 1_{4n+2k+2}.
\end{align*}

We have shown that $L(m'_2)$ is special. Therefore the Frenkel-Mukhin algorithm works for $L(m'_2)$. We will use the Frenkel-Mukhin algorithm to compute $\chi_q(L(m'_1)),  \chi_q(L(m'_2))$ and classify all dominant monomials in $\chi_q(L(m'_1))  \chi_q(L(m'_2))$. Let $m=m_1m_2$ be a dominant monomial, where $m_i \in \mathscr{M}(L(m'_i))$, $i=1, 2$.

Suppose that $m_2 \neq m'_2$. If $m_2$ is right-negative, then $m$ is a right negative monomial and therefore $m$ is not dominant, this is a contradiction. Hence $m_2$ is not right-negative. Through the above discussion, $m_2$ is one of the following monomials
\begin{align*}
& \bar{m}_1=m'_2 A_{4, 4n-2}^{-1} =4_{4n}^{-1} 3_{4n-2} 1_{4n+4} 1_{4n+6} \cdots 1_{4n+2k+2}, \\
& \bar{m}_2=\bar{m}_1A_{3, 4n}^{-1} =3^{-1}_{4n+2} 2_{4n-1} 2_{4n+1} 1_{4n+4} 1_{4n+6} \cdots 1_{4n+2k+2},\\
& \bar{m}_3=\bar{m}_2A_{2, 4n+2}^{-1}=2_{4n-1} 2^{-1}_{4n+3} 1_{4n+2} 1_{4n+4} 1_{4n+6} \cdots 1_{4n+2k+2},\\
& \bar{m}_4=\bar{m}_2A_{2, 4n}^{-1}=3_{4n}2^{-1}_{4n+1}2^{-1}_{4n+3}1_{4n}1_{4n+2} 1_{4n+4} 1_{4n+6} \cdots 1_{4n+2k+2},\\
& \bar{m}_5=\bar{m}_4A_{3, 4n+2}^{-1}=4_{4n+2}3^{-1}_{4n+4}1_{4n}1_{4n+2} 1_{4n+4} 1_{4n+6} \cdots 1_{4n+2k+2},\\
& \bar{m}_6=\bar{m}_4A_{4, 4n+4}^{-1}=4^{-1}_{4n+6} 1_{4n}1_{4n+2} 1_{4n+4} 1_{4n+6} \cdots 1_{4n+2k+2}.
\end{align*}
We nextly discuss $m_2$ as follows:

\textbf{Case 1.} The factor $4_{4n}$  can only come from the monomials in $\chi_q(4_{4n-8})$, the monomials in $\chi_q(4_{4n-8})$ which contain a factor $4_{4n}$ are
\begin{equation}\label{negative factors 1}
\begin{split}
&3^{-1}_{4n+2}4_{4n-2}4_{4n},\ 2_{4n+1}2_{4n+3}3^{-1}_{4n+2}3^{-1}_{4n+4}4_{4n}, \
1_{4n+4}2_{4n+1}2^{-1}_{4n+5}3^{-1}_{4n+2}4_{4n}, \ 1^{-1}_{4n+4}1^{-1}_{4n+6}4_{4n},\\
&1_{4n+2}1_{4n+4}2^{-1}_{4n+3}2^{-1}_{4n+5}4_{4n}, \ 1^{-1}_{4n+6}2_{4n+1}3^{-1}_{4n+2}4_{4n},
\ 1_{4n+2}1^{-1}_{4n+6}2^{-1}_{4n+3}4_{4n}.
\end{split}
\end{equation}

The negative factors in (\ref{negative factors 1}) can not be canceled by monomials in $\chi_q(4_{4n-12})$. However, the negative factor $1^{-1}_{4n+4}1^{-1}_{4n+6}4_{4n}$ can be canceled by $\bar{m}_1$. Therefore $m_{2}=\bar{m}_1$ in the situation.

\textbf{Case 2.} The factor $3_{4n+2}$ can only come from the monomials in $\chi_q(4_{4n-8})$, the monomials in $\chi_q(4_{4n-8})$ which contain a factor $3_{4n+2}$ are
\begin{align}\label{negative factors 2}
1_{4n+2} 1_{4n+4} 2^{-1}_{4n+3} 2^{-1}_{4n+5} 3_{4n+2} 4^{-1}_{4n+4},\ 1_{4n+2} 1^{-1}_{4n+6} 2^{-1}_{4n+3} 3_{4n+2} 4^{-1}_{4n+4},\ 1^{-1}_{4n+4} 1^{-1}_{4n+6} 3_{4n+2} 4^{-1}_{4n+4}.
\end{align}

The negative factors in (\ref{negative factors 2}) can not be canceled by monomials in $\chi_q(4_{4n-12})$.

\textbf{Case 3.}  The factors $2_{4n+3}, 3_{4n+4}, 4_{4n+6}$ can only come from the monomials in $\chi_q(4_{4n-8})$, the monomials in $\chi_q(4_{4n-8})$ which contain factors $2_{4n+3}, 3_{4n+4}, 4_{4n+6}$ are
\begin{equation}\label{negative factors 3}
\begin{split}
&1_{4n-4}2_{4n+3}3^{-1}_{4n+4}, \quad 1^{-1}_{4n-2}2_{4n-3}2_{4n+3}3^{-1}_{4n+4}, \quad 2^{-1}_{4n-1}2_{4n+3}3_{4n-2}3^{-1}_{4n+4}, \\
&2_{4n+1}2_{4n+3}3^{-1}_{4n+2}3^{-1}_{4n+4}4_{4n}, \quad  2_{4n+1}2_{4n+3}3^{-1}_{4n+4}4^{-1}_{4n+4},\quad 1^{-1}_{4n+4}2_{4n+3}2^{-1}_{4n+7}, \\
&1^{-1}_{4n+4}1^{-1}_{4n+6}2_{4n+3}2_{4n+5}3^{-1}_{4n+6}, \quad 3^{-1}_{4n+8}4_{4n+6}, \quad 2^{-1}_{4n+5}2^{-1}_{4n+7}3_{4n+4}.
\end{split}
\end{equation}

The negative factors in (\ref{negative factors 3}) can not be canceled by monomials  in $\chi_q(4_{4n-12})$.


Since
\begin{align*}
\chi_q(4_{0}4_{4}\cdots 4_{4n-8})\subseteq \chi_q(4_{0}4_{4}\cdots 4_{4n-12})\chi_q(4_{4n-8})\subseteq \chi_q(4_{0}4_{4}\cdots 4_{4n-16})\chi_q(4_{4n-12})\chi_q(4_{4n-8}),
\end{align*}
and $m=m_1m_2$ is dominant, $m_2=m'_2$ or $\bar{m}_1$.

Suppose that $m_{2}=\bar{m}_1$. By the above discussion, $m$ is in
\begin{align}\label{m 1}
\chi_q(4_{0}4_{4}\cdots 4_{4n-12})3_{4n-2}  1_{4n+8} \cdots 1_{4n+2k+2}.
\end{align}

Suppose that $m_{2}=m'_2$. If $m_1 \neq m'_1$, then $m_1$ is right negative. Since $m$ is dominant, each factor with a negative power in $m_1$ needs to be canceled by a factor in $m'_2$. We have $\mathscr{M}(L(m'_1)) \subset \mathscr{M}(\chi_q(4_{0}4_{4}\cdots 4_{4n-12}))$ $\chi_q(L(4_{4n-8}))$. Only monomials in $\chi_q(L(4_{4n-8}))$ can cancel $4_{4n-4}$, $1_{4n+4}$, $1_{4n+6}$.
Therefore
$m_1$ is one of the sets
\begin{align*}
&\mathscr{M}(\chi_q(4_{0}4_{4}\cdots 4_{4n-12}))3_{4n-6}4^{-1}_{4n-4},\\
&\mathscr{M}(\chi_q(4_{0}4_{4}\cdots 4_{4n-12}))1^{-1}_{4n+4}1^{-1}_{4n+6}4_{4n},\\
&\mathscr{M}(\chi_q(4_{0}4_{4}\cdots 4_{4n-12}))1^{-1}_{4n+4}1^{-1}_{4n+6}3_{4n+2}4_{4n+4},\\
&\mathscr{M}(\chi_q(4_{0}4_{4}\cdots 4_{4n-12}))1^{-1}_{4n+4}1^{-1}_{4n+6}2_{4n+3}2_{4n+5}3^{-1}_{4n+6},\\
&\mathscr{M}(\chi_q(4_{0}4_{4}\cdots 4_{4n-12}))1^{-1}_{4n+4}2_{4n+3}2^{-1}_{4n+7},\\
&\mathscr{M}(\chi_q(4_{0}4_{4}\cdots 4_{4n-12}))1_{4n-4}1^{-1}_{4n+6},\\
&\mathscr{M}(\chi_q(4_{0}4_{4}\cdots 4_{4n-12}))1^{-1}_{4n-2}1^{-1}_{4n+6}2_{4n-3},\\
&\mathscr{M}(\chi_q(4_{0}4_{4}\cdots 4_{4n-12}))1^{-1}_{4n+6}2^{-1}_{4n-1}3_{4n-2},\\
&\mathscr{M}(\chi_q(4_{0}4_{4}\cdots 4_{4n-12}))1^{-1}_{4n+6}2_{4n+1}3^{-1}_{4n+2}4_{4n},\\
&\mathscr{M}(\chi_q(4_{0}4_{4}\cdots 4_{4n-12}))1_{4n+2}1^{-1}_{4n+6}2^{-1}_{4n+3}4_{4n},\\
&\mathscr{M}(\chi_q(4_{0}4_{4}\cdots 4_{4n-12}))1^{-1}_{4n+6}2_{4n+1}4^{-1}_{4n+4},\\
&\mathscr{M}(\chi_q(4_{0}4_{4}\cdots 4_{4n-12}))1_{4n+2}1^{-1}_{4n+6}2^{-1}_{2n+3}3_{4n+2}4^{-1}_{4n+4},\\
&\mathscr{M}(\chi_q(4_{0}4_{4}\cdots 4_{4n-12}))1_{4n+2}1^{-1}_{4n+6}2_{4n+5}3^{-1}_{4n+6}.
\end{align*}

By (\ref{negative factors 3}), we know that $m_{1} \notin \mathscr{M}(\chi_q(4_{0}4_{4}\cdots 4_{4n-12}))1^{-1}_{4n+6}2^{-1}_{4n-1}3_{4n-2}$. The factors which contain $1_{4n-2}$ in $\chi_q(4n-12)$ are
\begin{equation}\label{negative factors 4}
\begin{split}
&1_{4n-2}1_{4n}2^{-1}_{4n-1}2^{-1}_{4n+1}4_{4n-4},\  1_{4n-2}1_{4n}2^{-1}_{4n-1}2^{-1}_{4n+1}3_{4n-2}4^{-1}_{4n},\ 1_{4n-2}1^{-1}_{4n+2}2^{-1}_{4n-1}4_{4n-4},\\
&1_{4n-2}1^{-1}_{4n+2}2^{-1}_{4n-1}3_{4n-2}4^{-1}_{4n},\  1_{4n-2}1_{4n}3^{-1}_{4n+2},\ 1_{4n-2}1^{-1}_{4n+2}2_{4n+1}3^{-1}_{4n+2},\  1_{4n-2}2^{-1}_{4n+3}.
\end{split}
\end{equation}

The negative factors in (\ref{negative factors 4}) can not be canceled by monomials  in $\chi_q(4_{4n-16})$.

Since $\chi_q(4_{0}4_{4}\cdots 4_{4n-12})\subseteq \chi_q(4_{0}4_{4}\cdots 4_{4n-16})\chi_q(4n-12)$ and $1_{4n+4},$ $1_{4n+6},$ $2_{4n+7}$ $\notin \chi_q(4n-12)$. Therefore $m_1$ is in one of the sets
\begin{align*}
&\mathscr{M}(\chi_q(4_{0}4_{4}\cdots 4_{4n-12}))3_{4n-6}4^{-1}_{4n-4},\\
&\mathscr{M}(\chi_q(4_{0}4_{4}\cdots 4_{4n-12}))1_{4n-4}1^{-1}_{4n+6}.
\end{align*}

Therefore $m$ is in one of the sets
\begin{align}
&\mathscr{M}(\chi_q(4_{0}4_{4}\cdots 4_{4n-12}))3_{4n-6}1_{4n+4} 1_{4n+6} \cdots 1_{4n+2k+2},\label{m 2}\\
&\mathscr{M}(\chi_q(4_{0}4_{4}\cdots 4_{4n-12}))1_{4n-4}4_{4n-4} 1_{4n+4} 1_{4n+8} \cdots 1_{4n+2k+2}. \label{m 3}
\end{align}

Let $m_+ = m_1''m_2''$, where
\begin{align*}
m''_1=4_{0} 4_{4}\cdots 4_{4n-4}, \ m''_2=1_{4n+4} 1_{4n+6} \cdots 1_{4n+2k+2}.
\end{align*}

If $m$ is the expressions of (\ref{m 1}) (\ref{m 2}), (\ref{m 3}), we know that $m \notin  \mathscr{M}(\chi_q(m''_1)\chi_q(m''_2))$ by the Frenkel-Mukhin algorithm.





Therefore the only dominant monomial in $\mathscr{M}(\chi_q(m'_1)\chi_q(m'_2)) \cap \mathscr{M}(\chi_q(m''_1)\chi_q(m''_2))$ is $m_+$. Hence the only dominant monomial in $\chi_q(m_+)$ is $m_+$.

\subsection{The case of $\mathcal{\widetilde{T}}_{n, m, 0 , k}^{(0)}$ $(k \leq2)$, $\mathcal{\widetilde{T}}_{n, 0, l , k}^{(0)}$ $(k\leq 2)$, and $\mathcal{\widetilde{T}}_{0, m, l , k}^{(0)}$ $(k\leq 2)$}

Since the proof of each case is similar to each other, we give a detailed proof of the case of $\mathcal{\widetilde{T}}_{n, m, 0 , k}^{(0)}$ $(k \leq2)$ as follows.

Let $m_+=\widetilde{T}_{n, m, 0 , k}^{(0)}$ with $n, m, k \in\mathbb{Z}_{\geq 1}$ and $k \leq 2$. Then

\begin{align*}
m_+ = (4_{0}4_{4}\cdots 4_{4n-4}) (3_{4n+2} 3_{4n+6} \cdots 3_{4n+4m-2}) (1_{4n+4m+4}\cdots 1_{4n+4m+2k+2}).
\end{align*}

Let
\begin{align*}
& m'_1=(4_{0}4_{4}\cdots 4_{4n-4}) (3_{4n+2} 3_{4n+6} \cdots 3_{4n+4m-2}), \ m'_2= (1_{4n+4m+4}\cdots 1_{4n+4m+2k+2}), \\
& m''_1=(4_{0}4_{4}\cdots 4_{4n-4}), \ m''_2=(3_{4n+2} 3_{4n+6} \cdots 3_{4n+4m-2}) (1_{4n+4m+4}\cdots 1_{4n+4m+2k+2}).
\end{align*}

Then $\mathscr{M}(L(m_+)) \subset \mathscr{M}(\chi_q(m'_1)\chi_q(m'_2)) \cap \mathscr{M}(\chi_q(m''_1)\chi_q(m''_2))$.

By using similar arguments as in Subsection \ref{cases 6.1}, we can show that the only possible dominant monomial in $\chi_q(m'_1)\chi_q(m'_2)$ $\cap$ $\chi_q(m''_1)\chi_q(m''_2)$ is $m_+$. Hence the only dominant monomial in $\chi_q(m_+)$ is $m_+$.


\subsection{The case of $\mathcal{\widetilde{T}}_{n, m, l , k}^{(0)}(k \leq2)$}

Let $m_+=\widetilde{T}_{n, m, l , k}^{(0)}$ with $n, m, l, k \in\mathbb{Z}_{\geq 1}$ and $k \leq 2$. Then

\begin{gather}
\begin{align*}
m_+=& (4_{0}4_{4}\cdots 4_{4n-4}) (3_{4n+2} 3_{4n+6} \cdots 3_{4n+4m-2})(2_{4n+4m+3}2_{4n+4m+5}\cdots2_{4n+4m+2l+1})\\
&(1_{4n+4m+2l+4}1_{4n+4m+2l+1+6}\cdots 1_{4n+4m+2l+2k+2}).
\end{align*}
\end{gather}

Let

\begin{gather}
\begin{align*}
m'_1&=(4_{0}4_{4}\cdots 4_{4n-4}) (3_{4n+2} 3_{4n+6} \cdots 3_{4n+4m-2})(2_{4n+4m+3}2_{4n+4m+5}\cdots 2_{4n+4m+2l+1}),\\
m'_2&=(1_{4n+4m+2l+4}1_{4n+4m+2l+1+6}\cdots 1_{4n+4m+2l+2k+2}), \\
m''_1&=(4_{0}4_{4}\cdots 4_{4n-4}),\\
m''_2&=(3_{4n+2} 3_{4n+6} \cdots 3_{4n+4m-2})(2_{4n+4m+3}2_{4n+4m+5}\cdots 2_{4n+4m+2l+1})(1_{4n+4m+2l+4}1_{4n+4m+2l+1+6}\cdots 1_{4n+4m+2l+2k+2}).
\end{align*}
\end{gather}

Then $\mathscr{M}(L(m_+)) \subset \mathscr{M}(\chi_q(m'_1)\chi_q(m'_2)) \cap \mathscr{M}(\chi_q(m''_1)\chi_q(m''_2))$.

By using similar arguments as in Subsection \ref{cases 6.1}, we can show that the only possible dominant monomial in $\chi_q(m'_1)\chi_q(m'_2)$ $\cap$ $\chi_q(m''_1)\chi_q(m''_2)$ is $m_+$. Hence the only dominant monomial in $\chi_q(m_+)$ is $m_+$.

\section{Proof of Theorem $\ref{M-system of type F4}$} \label{proof main1}

In this section, we prove Theorem $\ref{M-system of type F4}$.

\subsection{Classification of dominant monomials in the summands on both sides of the system}

In Section \ref{proof of special}, we have shown that for $s \in \mathbb{Z}, k, l, m, n \in \mathbb{Z}_{\geq 1},$ the modules
in Theorem \ref{special} are special.
Now we use the Frenkel-Mukhin algorithm to classify dominant monomials in the summands on both sides of the system in Theorem $\ref{M-system of type F4}$.

\begin{lemma}\label{lemma1}
The dominant monomials in each summand on the left and right hand sides of every equation in the system of Theorem \ref{M-system of type F4} are given in Table \ref{dominant monomials in the M-system of type F_4 1} and \ref{dominant monomials in the M-system of type F_4 2}.
\end{lemma}

\begin{table}[!htbp] \resizebox{.7\width}{.7\height}{
\begin{tabular}{|c|c|c|c|}
\hline %
Equations & Summands in the equations & $M$ & Dominant monomials \\
\hline %
(\ref{eqn 1}) &$\substack{\chi_{q}(\mathcal{T}^{(-2)}_{0,0,l-1,k})  \chi_{q}(\mathcal{T}^{(-2)}_{0,0,l,k})}$
& $\substack{M=T^{(-2)}_{0,0,l-1,k}  T^{(-2)}_{0,0,l,k}}$
& $\substack{M_{r}=M\prod _{i=0}^{r-1}A^{-1}_{1,aq^{-2l-1-2i}}, \\ 0\leq r\leq k}$ \\
\hline %
(\ref{eqn 1}) & $\substack{\chi_q(\mathcal T_{0,0,l,k-1}^{(-2)})  \chi_q(\mathcal{T}_{0,0,l-1,k+1}^{(-2)})}$
& $\substack{M=T_{0,0,l,k-1}^{(-2)}   T_{0,0,l-1,k+1}^{(-2)}}$
& $\substack{M_{r}=M\prod _{i=0}^{r-1}A^{-1}_{1,aq^{-2l-1-2i}},\\ 0\leq r\leq k-1}$ \\
\hline %
(\ref{eqn 1}) & $\substack{\chi_q(\mathcal T_{0,0,k+l,0}^{(-2)}) \chi_q(\mathcal T_{0,0,l-1,0}^{(-2)})}$
& $\substack{M= T_{0,0,k+l,0}^{(-2)} T_{0,0,l-1,0}^{(-2)}}$
& $M$\\
\hline %
(\ref{eqn 2})&$\substack{\chi_{q}(\mathcal{\widetilde{T}}_{n,m-1,0,0}^{(s+4)})
\chi_{q}(\mathcal{\widetilde{T}}_{n,m,0,0}^{(s)})}$
& $\substack{M=\widetilde{T}_{n,m-1,0,0}^{(s+4)} \widetilde{T}_{n,m,0,0}^{(s)}}$
& $ \substack{M_{r}=M\prod _{i=0}^{r-1}A^{-1}_{4,aq^{s+4n-2-4i}},\\ 0\leq r\leq n}$ \\
\hline %
(\ref{eqn 2})&$\substack{\chi_{q}(\mathcal{\widetilde{T}}_{n-1,m,0,0}^{(s+4)})
\chi_{q}(\mathcal{\widetilde{T}}_{n+1,m+1,0,0}^{(s)})}$
& $\substack{M=\widetilde{T}_{n-1,m,0,0}^{(s+4)} \widetilde{T}_{n+1,m-1,0,0}^{(s)}}$
& $\substack{M_{r}=M\prod _{i=0}^{r-1}A^{-1}_{4,aq^{s+4n-2-4i}}, \\-1\leq r\leq n-1}$\\
\hline %
(\ref{eqn 2}) & $\substack{\chi_{q}(\mathcal{\widetilde{T}}_{0,m-1,0,0}^{(s+4n+4)}) \chi_{q}(\mathcal{\widetilde{T}}_{0,n+m,0,0}^{(s)})}$
&$\substack{M=\widetilde{T}_{0,0,l-2,0}^{(s+4n+4)}  \widetilde{T}_{0,n,l,0}^{(s)}}$
& $M$\\
\hline %
(\ref{eqn 3a})&$\substack{\chi_{q}(\mathcal{\widetilde{T}}_{n,0,0,0}^{(s+4)})
\chi_{q}(\mathcal{\widetilde{T}}_{n,0,1,0}^{(s)})}$
& $\substack{M=\widetilde{T}_{n,0,0,0}^{(s+4)} \widetilde{T}_{n,0,1,0}^{(s)}}$
& $ \substack{M_{r}=M\prod _{i=0}^{r-1}A^{-1}_{4,aq^{s+4n-2-4i}},\\ 0\leq r\leq n}$ \\
\hline %
(\ref{eqn 3a})&$\substack{\chi_{q}(\mathcal{\widetilde{T}}_{n-1,0,1,0}^{(s+4)})
\chi_{q}(\mathcal{\widetilde{T}}_{n+1,0,0,0}^{(s)})}$
& $\substack{M=\widetilde{T}_{n-1,0,1,0}^{(s+4)} \widetilde{T}_{n+1,0,0,0}^{(s)}}$
& $\substack{M_{r}=M\prod _{i=0}^{r-1}A^{-1}_{4,aq^{s+4n-2-4i}}, \\-1\leq r\leq n-1}$\\
\hline %
(\ref{eqn 3a}) & $\substack{ \chi_{q}( \mathcal{ \widetilde{T}}_{0,n,l,0}^{(s)})}$
&$\substack{M=\widetilde{T}_{0,n,l,0}^{(s)}}$
& $M$\\
\hline %
(\ref{eqn 3b})&$\substack{\chi_{q}(\mathcal{\widetilde{T}}_{n,0,l-2,0}^{(s+4)})
\chi_{q}(\mathcal{\widetilde{T}}_{n,0,l,0}^{(s)})}$
& $\substack{M=\widetilde{T}_{n,0,l-2,0}^{(s+4)} \widetilde{T}_{n,0,l,0}^{(s)}}$
& $ \substack{M_{r}=M\prod _{i=0}^{r-1}A^{-1}_{4,aq^{s+4n-2-4i}},\\ 0\leq r\leq n}$ \\
\hline %
(\ref{eqn 3b})&$\substack{\chi_{q}(\mathcal{\widetilde{T}}_{n-1,0,l,0}^{(s+4)})
\chi_{q}(\mathcal{\widetilde{T}}_{n+1,0,l-2,0}^{(s)})}$
& $\substack{M=\widetilde{T}_{n-1,0,l,0}^{(s+4)} \widetilde{T}_{n+1,0,l-2,0}^{(s)}}$
& $\substack{M_{r}=M\prod _{i=0}^{r-1}A^{-1}_{4,aq^{s+4n-2-4i}}, \\-1\leq r\leq n-1}$\\
\hline %
(\ref{eqn 3b}) & $\substack{\chi_{q}(\mathcal{\widetilde{T}}_{0,0,l-2,0}^{(s+4n+4)}) \chi_{q}(\mathcal{\widetilde{T}}_{0,n,l,0}^{(s)})}$
&$\substack{M=\widetilde{T}_{0,0,l-2,0}^{(s+4n+4)}  \widetilde{T}_{0,n,l,0}^{(s)}}$
& $M$\\
\hline %
(\ref{eqn 4a}) & $\substack{\chi_{q}(\mathcal{\widetilde{T}}_{0,m,0,0}^{(s+4)}) \chi_{q}(\mathcal{\widetilde{T}}_{0,m,1,0}^{(s)})}$
& $\substack{M=\widetilde{T}_{0,m,0,0}^{(s+4)}  \widetilde{T}_{0,m,1,0}^{(s)}}$
& $\substack{M_{r}=M\prod _{i=0}^{r-1}A^{-1}_{3,aq^{s+4m-4i}}, \\
0\leq r \leq m}$ \\
\hline %
(\ref{eqn 4a}) & $\substack{\chi_{q}(\mathcal{\widetilde{T}}_{0,m-1,1,0}^{(s+4)}) \chi_{q}(\mathcal{\widetilde{T}}_{0,m+1,0,0}^{(s)})}$
&$\substack{M=\widetilde{T}_{0,m,0,0}^{(s+4)}  \widetilde{T}_{0,m,1,0}^{(s)}}$
&$\substack{M_{r}=M\prod _{i=0}^{r-1}A^{-1}_{3,aq^{s+4m-4i}}, \\0\leq r \leq m-1}$ \\
\hline %
(\ref{eqn 4a}) & $\substack{\chi_{q}(\mathcal{\widetilde{T}}_{m,0,0,0}^{(s+4)}) \chi_{q}(\mathcal{\widetilde{T}}_{0,0,1+2m,0}^{(s)})}$
&$\substack{M=\widetilde{T}_{m,0,0,0}^{(s+4)}  \widetilde{T}_{0,0,1+2m,0}^{(s)}}$
& $M$\\
\hline %
(\ref{eqn 4b}) & $\substack{\chi_{q}(\mathcal{\widetilde{T}}_{0,m,l-2,0}^{(s+4)}) \chi_{q}(\mathcal{\widetilde{T}}_{0,m,l,0}^{(s)})}$
& $\substack{M=\widetilde{T}_{0,m,l-2,0}^{(s+4)}  \widetilde{T}_{0,m,l,0}^{(s)}}$
& $\substack{M_{r}=M\prod _{i=0}^{r-1}A^{-1}_{3,aq^{s+4m-4i}}, \\
0\leq r \leq m}$ \\
\hline %
(\ref{eqn 4b}) & $\substack{\chi_{q}(\mathcal{\widetilde{T}}_{0,m-1,l,0}^{(s+4)}) \chi_{q}(\mathcal{\widetilde{T}}_{0,m+1,l-2,0}^{(s)})}$
&$\substack{M=\widetilde{T}_{0,m,l-2,0}^{(s+4)}  \widetilde{T}_{0,m,l,0}^{(s)}}$
&$\substack{M_{r}=M\prod _{i=0}^{r-1}A^{-1}_{3,aq^{s+4m-4i}}, \\0\leq r \leq m-1}$ \\
\hline %
(\ref{eqn 4b}) & $\substack{\chi_{q}(\mathcal{\widetilde{T}}_{m,0,0,l-2}^{(s+4)}) \chi_{q}(\mathcal{\widetilde{T}}_{0,0,l+2m,0}^{(s)})}$
&$\substack{M=\widetilde{T}_{m,0,0,l-2}^{(s+4)}  \widetilde{T}_{0,0,l+2m,0}^{(s)}}$
& $M$\\
\hline %
(\ref{eqn 5}) & $\substack{\chi_{q}(\mathcal{\widetilde{T}}_{n,m-1,m,0}^{(s+4)}) \chi_{q}(\mathcal{\widetilde{T}}_{n,m,l,0}^{(s)})}$
&$\substack{M=\widetilde{T}_{n,m-1,l,0}^{(s+4)}  \widetilde{T}_{n,m,l,0}^{(s)}}$
&$\substack{M_{r}=M \prod _{i=0}^{r-1}A^{-1}_{4,aq^{s+4n-2-4i}}, \\ 0\leq r \leq n}$ \\
\hline %
(\ref{eqn 5}) & $\substack{\chi_{q}(\mathcal{\widetilde{T}}_{n-1,m,l,0}^{(s+4)})  \chi_{q}(\mathcal{\widetilde{T}}_{n+1,m-1,l,0}^{(s)})}$
&$\substack{M=\widetilde{T}_{n-1,m,l,0}^{(s+4)}  \widetilde{T}_{n+1,m-1,l,0}^{(s)}}$
&$\substack{M_{r}=M\prod _{i=0}^{r-1}A^{-1}_{4,aq^{s+4n-2-4i}}, \\0\leq r \leq n-1}$ \\
\hline %
(\ref{eqn 5}) & $\substack{\chi_{q}(\mathcal{\widetilde{T}}_{0,m-1,l,0}^{(s+4n+4)}) \chi_{q}(\mathcal{\widetilde{T}}_{0,n+m,l,0}^{(s)})}$
&$\substack{M=\widetilde{T}_{0,m-1,l,0}^{(s+4n+4)})  \widetilde{T}_{0,n+m,l,0}^{(s)}}$
&$M$\\
\hline %
(\ref{eqn 6a}) & $\substack{\chi_{q}(\mathcal{\widetilde{T}}_{n,0,0,0}^{(s+4)}) \chi_{q}(\mathcal{\widetilde{T}}_{n,0,0,k}^{(s)})}$
&$\substack{M=\widetilde{T}_{n,0,0,0}^{(s+4)}  \widetilde{T}_{n,0,0,k}^{(s)}}$
&$\substack{M_{r}=M\prod _{i=0}^{r-1}A^{-1}_{4,aq^{s+4n-2-4i}},\\ 0\leq r \leq n}$ \\
\hline %
(\ref{eqn 6a}) & $\substack{\chi_{q}(\mathcal{\widetilde{T}}_{n-1,0,0,k}^{(s+4)}) \chi_{q}(\mathcal{\widetilde{T}}_{n+1,0,0,0}^{(s)})}$
&$\substack{M=\widetilde{T}_{n-1,0,0,k}^{(s+4)} \widetilde{T}_{n+1,0,0,0}^{(s)}}$
&$\substack{M_{r}=M\prod _{i=0}^{r-1}A^{-1}_{4,aq^{s+4n-2-4i}},\\ 0 \leq r \leq n-1}$ \\
\hline %
(\ref{eqn 6a}) & $\substack{\chi_{q}(\mathcal{\widetilde{T}}_{0,n,0,k}^{(s)})}$
&$\substack{M=\widetilde{T}_{0,n,0,k}^{(s)}}$
&$M$ \\
\hline %
(\ref{eqn 7a}) & $\substack{\chi_{q}(\mathcal{\widetilde{T}}_{0,m,0,0}^{(s+4)}) \chi_{q}(\mathcal{\widetilde{T}}_{0,m,0,k}^{(s)})}$
&$\substack{M=\widetilde{T}_{0,m,0,0}^{(s+4)} \widetilde{T}_{0,m,0,k}^{(s)}}$
&$ \substack{M_{r}=M\prod _{i=0}^{r-1}A^{-1}_{3,aq^{s+4m-4i}}, \\0 \leq r \leq m}$ \\
\hline %
(\ref{eqn 7a}) & $\substack{\chi_{q}(\mathcal{\widetilde{T}}_{0,m-1,0,k}^{(s+4)}) \chi_{q}(\mathcal{\widetilde{T}}_{0,m+1,0,0}^{(s)})}$
&$\substack{M=\widetilde{T}_{0,m-1,0,k}^{(s+4)} \widetilde{T}_{0,m+1,0,0}^{(s)}}$
&$\substack{M_{r}=M\prod _{i=0}^{r-1}A^{-1}_{3,aq^{s+4m-4i}}, \\ 0 \leq r \leq m-1}$ \\
\hline %
(\ref{eqn 7a}) & $\substack{\chi_{q}(\mathcal{\widetilde{T}}_{m,0,0,0}^{(s+4)})
\chi_{q}(\mathcal{\widetilde{T}}_{0,0,2m,k}^{(s)})}$
&$\substack{M=\widetilde{T}_{m,0,0,0}^{(s+4)} \widetilde{T}_{0,0,2m,k}^{(s)}}$
&$M$ \\
\hline %
(\ref{eqn 8}) & $\substack{\chi_{q}(\mathcal{\widetilde{T}}_{n,0,0,k-1}^{(s+4)}) \chi_{q}(\mathcal{\widetilde{T}}_{n,0,1,k}^{(s)})}$
& $ \substack{M=\widetilde{T}_{n,0,0,k-1}^{(s+4)} \widetilde{T}_{n,0,1,k}^{(s)}}$
& $ \substack{M_{r}=M\prod _{i=0}^{r-1}A^{-1}_{4,aq^{s+4n-2-4i}},\\ 0 \leq r \leq n}$\\
\hline %
(\ref{eqn 8}) & $\substack{\chi_{q}(\mathcal{\widetilde{T}}_{n-1,0,1,k}^{(s+4)}) \chi_{q}(\mathcal{\widetilde{T}}_{n+1,0,0,k-1}^{(s)})}$
&$\substack{M=\widetilde{T}_{n-1,0,1,k}^{(s+4)} \widetilde{T}_{n+1,0,0,k-1}^{(s)}}$
&$\substack{M_{r}=M \prod _{i=0}^{r-1}A^{-1}_{4,aq^{s+4n-2-4i}},\\ 0 \leq r \leq n-1}$\\
\hline %
(\ref{eqn 8}) & $\substack{\chi_{q}(\mathcal{\widetilde{T}}_{0,0,0,k-1}^{(s+4n+4)})
\chi_{q}(\mathcal{\widetilde{T}}_{0,n,1,k}^{(s)})}$
&$\substack{M=\widetilde{T}_{0,0,0,k-1}^{(s+4n+4)} \widetilde{T}_{0,n,l,k}^{(s)}}$
&$M$\\
\hline %
(\ref{eqn 9}) & $\substack{\chi_{q}(\mathcal{\widetilde{T}}_{n,0,l-2,k}^{(s+4)}) \chi_{q}(\mathcal{\widetilde{T}}_{n,0,l,k}^{(s)})}$
&$\substack{M=\widetilde{T}_{n,0,l-2,k}^{(s+4)} \widetilde{T}_{n,0,l,k}^{(s)}}$
&$\substack{M_{r}=M\prod _{i=0}^{r-1}A^{-1}_{4,aq^{s+4n-2-4i}},\\ 0 \leq r \leq n}$\\
\hline %
(\ref{eqn 9}) &$\substack{\chi_{q}(\mathcal{\widetilde{T}}_{n-1,0,l,k}^{(s+4)}) \chi_{q}(\mathcal{\widetilde{T}}_{n+1,0,l-2,k}^{(s)})}$
&$\substack{M=\widetilde{T}_{n-1,0,l,k}^{(s+4)} \widetilde{T}_{n+1,0,l-2,k}^{(s)}}$
&$ \substack{M_{r}=M\prod _{i=0}^{r-1}A^{-1}_{4,aq^{s+4n-2-4i}},\\ 0 \leq r \leq n-1}$\\
\hline %
(\ref{eqn 9}) &$\substack{\chi_{q}(\mathcal{\widetilde{T}}_{0,0,l-2,k}^{(s+4n+4)})
\chi_{q}(\mathcal{\widetilde{T}}_{0,n,l,k}^{(s)})}$
&$\substack{M=\widetilde{T}_{0,0,l-2,k}^{(s+4n+4)} \widetilde{T}_{0,n,l,k}^{(s)}}$
&$M$\\
\hline %
(\ref{eqn 10}) & $\substack{\chi_{q}(\mathcal{\widetilde{T}}_{0,m,0,k-1}^{(s+4)}) \chi_{q}(\mathcal{\widetilde{T}}_{0,m,1,k}^{(s)})}$
&$ \substack{M=\widetilde{T}_{0,m,0,k-1}^{(s+4)} \widetilde{T}_{0,m,1,k}^{(s)}}$
&$\substack{M_{r}=M\prod _{i=0}^{r-1}A^{-1}_{3,aq^{s+4m-4i}}, \\0 \leq r \leq m}$\\
\hline %
(\ref{eqn 10}) & $\substack{\chi_{q}(\mathcal{\widetilde{T}}_{0,m-1,1,k}^{(s+4)}) \chi_{q}(\mathcal{\widetilde{T}}_{0,m+1,0,k-1}^{(s)})}$
&$\substack{M=\widetilde{T}_{0,m-1,1,k}^{(s+4)} \widetilde{T}_{0,m+1,0,k-1}^{(s)}}$
&$\substack{M_{r}=M\prod _{i=0}^{r-1}A^{-1}_{3,aq^{s+4m-4i}}, \\0 \leq r \leq m-1}$\\
\hline %
(\ref{eqn 10}) & $\substack{\chi_{q}(\mathcal{\widetilde{T}}_{m,0,0,k-1}^{(s+4)})
\chi_{q}(\mathcal{\widetilde{T}}_{0,0,1+2m,k}^{(s)})}$
&$\substack{M=\widetilde{T}_{m,0,0,k-1}^{(s+4)} \widetilde{T}_{0,0,1+2m,k}^{(s)}} $
&$M$ \\
\hline %
(\ref{eqn 11}) & $\substack{\chi_{q}(\mathcal{\widetilde{T}}_{0,m,l-2,k}^{(s+4)}) \chi_{q}(\mathcal{\widetilde{T}}_{0,m,l,k}^{(s)})}$
& $\substack{M=\widetilde{T}_{0,m,l-2,k}^{(s+4)} \widetilde{T}_{0,m,l,k}^{(s)}}$
& $\substack{M_{r}=M\prod _{i=0}^{r-1}A^{-1}_{3,aq^{s+4m-4i}},\\ 0 \leq r \leq m}$\\
\hline %
(\ref{eqn 11}) & $\substack{\chi_{q}(\mathcal{\widetilde{T}}_{0,m-1,l,k}^{(s+4)}) \chi_{q}(\mathcal{\widetilde{T}}_{0,m+1,l-2,k}^{(s)})}$
& $\substack{M=\widetilde{T}_{0,m-1,l,k}^{(s+4)} \widetilde{T}_{0,m+1,l-2,k}^{(s)}}$
& $\substack{M_{r}=M\prod _{i=0}^{r-1}A^{-1}_{3,aq^{s+4m-4i}}, \\0 \leq r \leq m-1}$\\
\hline %
(\ref{eqn 11}) & $\substack{\chi_{q}(\mathcal{\widetilde{T}}_{m,0,l-2,k}^{(s+4)})
\chi_{q}(\mathcal{\widetilde{T}}_{0,0,l+2m,k}^{(s)}) }$
& $\substack{M=\widetilde{T}_{m,0,l-2,k}^{(s+4)} \widetilde{T}_{0,0,l+2m,k}^{(s)} }$
& $M$\\
\hline %
(\ref{eqn 12}) & $\substack{\chi_{q}(\mathcal{\widetilde{T}}_{n,m-1,0,k}^{(s+4)}) \chi_{q}(\mathcal{\widetilde{T}}_{n,m,0,k}^{(s)})}$
& $\substack{M=\widetilde{T}_{n,m-1,0,k}^{(s+4)} \widetilde{T}_{n,m,0,k}^{(s)}}$
& $\substack{M_{r}=M\prod _{i=0}^{r-1}A^{-1}_{4,aq^{s+4n-2-4i}}, \\0 \leq r \leq n}$ \\
\hline %
(\ref{eqn 12}) & $\substack{\chi_{q}(\mathcal{\widetilde{T}}_{n-1,m,0,k}^{(s+4)}) \chi_{q}(\mathcal{\widetilde{T}}_{n+1,m-1,0,k}^{(s)})}$
&$\substack{M=\widetilde{T}_{n-1,m,0,k}^{(s+4)} \widetilde{T}_{n+1,m-1,0,k}^{(s)}}$
&$\substack{M_{r}=M\prod _{i=0}^{r-1}A^{-1}_{4,aq^{s+4n-2-4i}}, \\0 \leq r \leq n-1}$ \\
\hline %
(\ref{eqn 12}) & $\substack{\chi_{q}(\mathcal{\widetilde{T}}_{0,m-1,0,k}^{(s+4)})
\chi_{q}(\mathcal{\widetilde{T}}_{0,n+m,0,k}^{(s)})}$
&$\substack{M=\widetilde{T}_{0,m-1,0,k}^{(s+4)} \widetilde{T}_{0,n+m,0,k}^{(s)} }$
&$M$ \\
\hline %
(\ref{eqn 13}) & $\substack{ \chi_{q}(\mathcal{\widetilde{T}}_{n, m-1, l, k}^{(s+4)})
\chi_{q}(\mathcal{\widetilde{T}}_{n, m, l, k}^{(s)})}$
&$\substack{M=\widetilde{T}_{n,m-1,l,k}^{(s+4)} \widetilde{T}_{n,m,l,k}^{(s)}}$
&$\substack{M_{r}=M\prod _{i=0}^{r-1}A^{-1}_{4,aq^{s+4n-2-4i}}, \\0 \leq r \leq n} $ \\
\hline %
(\ref{eqn 13}) & $\substack{\chi_{q}(\mathcal{\widetilde{T}}_{n-1,m,l,k}^{(s+4)}) \chi_{q}(\mathcal{\widetilde{T}}_{n+1,m-1,l,k}^{(s)})}$
&$\substack{M=\widetilde{T}_{n-1,m,l,k}^{(s+4)} \widetilde{T}_{n+1,m-1,l,k}^{(s)}}$
&$\substack{M_{r}=M\prod _{i=0}^{r-1}A^{-1}_{4,aq^{s+4n-2-4i}}, \\0 \leq r \leq n-1} $ \\
\hline %
(\ref{eqn 13}) & $\substack{\chi_{q}(\mathcal{\widetilde{T}}_{0,m-1,l,k}^{(s+4n+4)}) \chi_{q}(\mathcal{\widetilde{T}}_{0,n+m,l,k}^{(s)})}$
&$\substack{M=\widetilde{T}_{0,n+m,l,k}^{(s)} \widetilde{T}_{0,m-1,l,k}^{(s+4n+4)}}$
&$M$ \\
\hline %
\end{tabular}}
\caption{Classification of dominant monomials in the system in Theorem \ref{M-system of type F4}.}
\label{dominant monomials in the M-system of type F_4 1}
\end{table}

\begin{table}[!htbp] \resizebox{.65\width}{.65\height}{
\begin{tabular}{|c|c|c|c|}
\hline %
Equations & Summands in the equations & $M$ & Dominant monomials \\
\hline %
(\ref{eqn 14}) & $\substack{ \chi_{q}(\widetilde{\mathcal{P}}^{(s+2)}_{0,0,1,k-1}) \chi_{q}(\widetilde{\mathcal{P}}^{(s)}_{0,0,1,k})}$
&$\substack{M=\widetilde{P}^{(s+2)}_{0,0,1,k-1}\widetilde{P}^{(s)}_{0,0,1,k}}$
&$\substack{M_{0}=M, M_{1}=M A^{-1}_{2,aq^{s+4}}} $ \\
\hline %
(\ref{eqn 14}) & $\substack{\chi_{q}(\mathcal{T}^{(s+10)}_{0,0,0,k-4}) \chi_{q}(\mathcal{T}^{(s+6)}_{0,0,0,k-2})\chi_{q}(\mathcal{T}^{(s+2)}_{0,0,0,k})\chi_{q}(\widetilde{\mathcal{P}}^{(s)}_{0,0,2,k-1})}$
&$\substack{M=T^{(s+10)}_{0,0,0,k-4} T^{(s+6)}_{0,0,0,k-2} T^{(s+2)}_{0,0,0,k}\widetilde{P}^{(s)}_{0,0,2,k-1}}$
&$\substack{ M } $ \\
\hline %
(\ref{eqn 14}) & $\substack{\chi_{q}(\mathcal{T}^{(s+12)}_{0,0,0,k-5}) \chi_{q}(\mathcal{T}^{(s+8)}_{0,0,0,k-3}) \chi_{q}(\mathcal{T}^{(s+4)}_{0,0,0,k-1}) \chi_{q}(\mathcal{T}^{(s)}_{0,0,0,k+1}) \chi_{q}(\widetilde{\mathcal{P}}^{(s+2)}_{0,1,0,k-2})}$
&$\substack{M=T^{(s+12)}_{0,0,0,k-5} T^{(s+8)}_{0,0,0,k-3} T^{(s+4)}_{0,0,0,k-1}
T^{(s)}_{0,0,0,k+1} \widetilde{P}^{(s+2)}_{0,1,0,k-2}}$
&$\substack{ M } $ \\
\hline %
(\ref{eqn 15}) & $\substack{ \chi_{q}(\widetilde{\mathcal{P}}^{(s+2)}_{0,0,l,k-1}) \chi_{q}(\widetilde{\mathcal{P}}^{(s)}_{0,0,l,k})}$
&$\substack{M=\widetilde{P}^{(s+2)}_{0,0,l,k-1} \widetilde{P}^{(s)}_{0,0,l,k}}$
&$\substack{M_{r}=M\prod _{i=0}^{r-1}A^{-1}_{2,aq^{s+2l+2-2i}}, \\0 \leq r \leq l} $ \\
\hline %
(\ref{eqn 15}) & $\substack{\chi_{q}(\widetilde{\mathcal{P}}^{(s+2)}_{0,0,l-1,k}) \chi_{q}(\widetilde{\mathcal{P}}^{(s)}_{0,0,l+1,k-1})}$
&$\substack{M=\widetilde{P}^{(s+2)}_{0,0,l-1,k} \widetilde{P}^{(s)}_{0,0,l+1,k-1}}$
&$\substack{M_{r}=M\prod _{i=0}^{r-1}A^{-1}_{2,aq^{s+2l+2-2i}}, \\0 \leq r \leq l-1} $ \\
\hline %
(\ref{eqn 15}) & $\substack{\chi_{q}(\mathcal{T}^{(s+2l+10)}_{0,0,0,k-5}) \chi_{q}(\mathcal{T}^{(s)}_{0,0,0,k+l})
\chi_{q}(\widetilde{\mathcal{P}}^{(s+2)}_{0,\frac{l+1}{2},0,k-2}) \chi_{q}(\widetilde{\mathcal{P}}^{(s+4)}_{0,\frac{l-1}{2},0,k-1})}$
&$\substack{M=T^{(s+2l+10)}_{0,0,0,k-5} T^{(s)}_{0,0,0,k+l} \widetilde{P}^{(s+2)}_{0,\frac{l+1}{2},0,k-2}
\widetilde{\mathcal{P}}^{(s+4)}_{0,\frac{l-1}{2},0,k-1}}$
&$\substack{ M } $ \\
\hline %
(\ref{eqn 16}) & $\substack{\chi_{q}(\widetilde{\mathcal{P}}^{(s+2)}_{0,0,l,k-1})
 \chi_{q}(\widetilde{\mathcal{P}}^{(s)}_{0,0,l,k})}$
&$\substack{M=\widetilde{P}^{(s+2)}_{0,0,l,k-1} \widetilde{P}^{(s)}_{0,0,l,k}}$
&$\substack{M_{r}=M_{r}=M\prod _{i=0}^{r-1}A^{-1}_{2,aq^{s+2l+2-2i}}, \\0 \leq r \leq l} $ \\
\hline %
(\ref{eqn 16}) & $\substack{\chi_{q}(\widetilde{\mathcal{P}}^{(s+2)}_{0,0,l-1,k}) \chi_{q}(\widetilde{\mathcal{P}}^{(s)}_{0,0,l+1,k-1})}$
&$\substack{M=\widetilde{P}^{(s+2)}_{0,0,l-1,k} \widetilde{P}^{(s)}_{0,0,l+1,k-1}}$
&$\substack{M_{r}=M_{r}=M\prod _{i=0}^{r-1}A^{-1}_{2,aq^{s+2l+2-2i}}, \\0 \leq r \leq l-1} $ \\
\hline %
(\ref{eqn 16}) & $\substack{\chi_{q}(\mathcal{T}^{(s+2l+10)}_{0,0,0,k-5})\chi_{q}(\mathcal{T}^{(s)}_{0,0,0,k+l})
\chi_{q}(\widetilde{\mathcal{P}}^{(s+4)}_{0,\frac{l}{2},0,k-2}) \chi_{q}(\widetilde{\mathcal{P}}^{(s+2)}_{0,\frac{l}{2},0,k-1})}$
&$\substack{M=T^{(s+2l+10)}_{0,0,0,k-5} T^{(s)}_{0,0,0,k+l} \widetilde{P}^{(s+4)}_{0,\frac{l}{2},0,k-2}
\widetilde{\mathcal{P}}^{(s+2)}_{0,\frac{l}{2},0,k-1}}$
&$\substack{ M } $ \\
\hline %
\ref{eqn 17}) & $\substack{\chi_{q}(\widetilde{\mathcal{P}}^{(s+4)}_{0,1,0,k-2}) \chi_{q}(\widetilde{\mathcal{P}}^{(s)}_{0,1,0,k})}$
&$\substack{M=\widetilde{P}^{(s+4)}_{0,1,0,k-2} \widetilde{P}^{(s)}_{0,1,0,k}}$
&$ \substack{M_{r}=M\prod _{i=0}^{r-1}A^{-1}_{3,aq^{s+4m-4i}}, \\0 \leq r \leq m }$ \\
\hline %
(\ref{eqn 17}) & $\substack{\chi_{q}(\mathcal{T}^{(s+8)}_{0,0,0,k-2}) \chi_{q}(\mathcal{T}^{(s+4)}_{0,0,0,k}) \chi_{q}(\widetilde{\mathcal{P}}^{(s)}_{0,2,0,k-2})}$
&$\substack{M=T^{(s+8)}_{0,0,0,k-2} T^{(s+4)}_{0,0,0,k} \widetilde{P}^{(s)}_{0,2,0,k-2}}$
&$\substack{M_{r}=M\prod _{i=0}^{r-1}A^{-1}_{3,aq^{s+4m-4i}}, \\ 0 \leq r \leq m-1 }$ \\
\hline %
(\ref{eqn 17}) & $\substack{\chi_{q}(\widetilde{\mathcal{T}}^{(s+4)}_{1,0,0,k-2}) \chi_{q}(\widetilde{\mathcal{P}}^{(s)}_{0,0,2,k})}$
&$\substack{M=\widetilde{T}^{(s+4)}_{1,0,0,k-2} \widetilde{P}^{(s)}_{0,0,2,k} }$
&$\substack{ M } $ \\
\hline %
(\ref{eqn 18}) & $\substack{\chi_{q}(\widetilde{\mathcal{P}}^{(s+4)}_{0,m,0,k-2}) \chi_{q}(\widetilde{\mathcal{P}}^{(s)}_{0,m,0,k})}$
& $ \substack{M=\widetilde{P}^{(s+4)}_{0,1,0,k-2} \widetilde{P}^{(s)}_{0,1,0,k}}$
& $ \substack{M_{r}=M\prod _{i=0}^{r-1}A^{-1}_{3,aq^{s+4m-4i}}, \\0 \leq r \leq m }$\\
\hline %
(\ref{eqn 18}) & $\substack{\chi_{q}(\widetilde{\mathcal{P}}^{(s+4)}_{0,m-1,0,k}) \chi_{q}(\widetilde{\mathcal{P}}^{(s)}_{0,m+1,0,k-2})}$
&$\substack{M=\widetilde{P}^{(s+4)}_{0,m-1,0,k} \widetilde{P}^{(s)}_{0,m+1,0,k-2}}$
&$\substack{M_{r}=M\prod _{i=0}^{r-1}A^{-1}_{3,aq^{s+4m-4i}}, \\ 0 \leq r \leq m-1 }$\\
\hline %
(\ref{eqn 18}) & $\substack{\chi_{q}(\widetilde{\mathcal{T}}^{(s+4)}_{m,0,0,k-2}) \chi_{q}(\widetilde{\mathcal{P}}^{(s)}_{0,0,2m,k})}$
&$\substack{M= \widetilde{T}^{(s+4)}_{m,0,0,k-2} \widetilde{P}^{(s)}_{0,0,2m,k}}$
&$\substack{ M } $ \\
\hline %
(\ref{eqn 19})&$\substack{\chi_{q}(\widetilde{\mathcal{T}}^{(s+4)}_{n,0,0,k-2})
\chi_{q}(\widetilde{\mathcal{T}}^{(s)}_{n,0,0,k})}$
& $\substack{M=\widetilde{T}^{(s+4)}_{n,0,0,k-2} \widetilde{T}^{(s)}_{n,0,0,k}}$
& $ \substack{M_{r}=M\prod _{i=0}^{r-1}A^{-1}_{4,aq^{s+4n-2-4i}},\\ 0\leq r\leq n}$ \\
\hline %
(\ref{eqn 19})&$\substack{\chi_{q}(\widetilde{\mathcal{T}}^{(s+4)}_{n-1,0,0,k})
\chi_{q}(\widetilde{\mathcal{T}}^{(s)}_{n+1,0,0,k-2})}$
& $\substack{M=\widetilde{T}^{(s+4)}_{n-1,0,0,k} \widetilde{T}^{(s)}_{n+1,0,0,k-2}}$
& $\substack{M_{r}=M\prod _{i=0}^{r-1}A^{-1}_{4,aq^{s+4n-2-4i}}, \\ 0\leq r\leq n-1}$\\
\hline %
(\ref{eqn 19}) & $\substack{\chi_{q}(\widetilde{\mathcal{P}}^{(s)}_{0,n,0,k})}$
&$\substack{M=\widetilde{P}^{(s)}_{0,n,0,k}}$
&$\substack{ M } $ \\
\hline %
\end{tabular}}
\caption{Classification of dominant monomials in the system in Theorem \ref{M-system of type F4} (continued).}
\label{dominant monomials in the M-system of type F_4 2}
\end{table}

\begin{proof}
We will prove the case of (\ref{eqn 13}), the other cases are similar.

Let $m'_{1}=\widetilde{T}_{n,m-1,l,k}^{(s+4)}$, $m'_{2}=\widetilde{T}_{n,m,l,k}^{(s)}$. Without loss of generality, we may assume that $s=0$. Then

\begin{align*}
m'_{1}=&(4_{4}\cdots 4_{4n})(3_{4n+6}3_{4n+10}\cdots 3_{4n+4m-2})(2_{4n+4m+3}2_{4n+4m+5}\cdots 2_{4n+4m+2l+1})\\
&(1_{4n+4m+2l+4} 1_{4n+4m+2l+6} \cdots 1_{4n+4m+2l+2k+2}),\\
m'_{2}=&(4_{0}4_{4} \cdots4_{4n-4})(3_{4n+2}3_{4n+6}\cdots3_{4n+4m-2})(2_{4n+4m+3}2_{4n+4m+5}\cdots2_{4n+4m+2l+1})\\
&(1_{4n+4m+2l+4} 1_{4n+4m+2l+6} \cdots 1_{4n+4m+2l+2k+2}).
\end{align*}

Let $m=m_{1}m_{2}$ be a dominant monomial, where $m_{i} \in \chi_{q}(m'_{i}), i=1, 2$. We denote

\begin{align*}
m_{3}=&(3_{4n+6}3_{4n+10}\cdots 3_{4n+4m-2})(2_{4n+4m+3}2_{4n+4m+5}\cdots 2_{4n+4m+2l+1})\\
&(1_{4n+4m+2l+4} 1_{4n+4m+2l+6} \cdots 1_{4n+4m+2l+2k+2}),\\
m_{4}=&(3_{4n+2}3_{4n+6}\cdots3_{4n+4m-2})(2_{4n+4m+3}2_{4n+4m+5}\cdots2_{4n+4m+2l+1})\\
&(1_{4n+4m+2l+4} 1_{4n+4m+2l+6} \cdots 1_{4n+4m+2l+2k+2}).
\end{align*}

Suppose that $m_{1} \in \chi_{q}(4_{4}\cdots 4_{4n})(\chi_{q}(m_{3})-m_{3})$, then $m=m_{1}m_{2}$ is right negative and hence $m$ is not dominant. Therefor $m_{1} \in \chi_{q}(4_{4}\cdots 4_{4n})m_{3}$. Similarly, if $m_{2}\in \chi_{q}(4_{0}4_{4} \cdots4_{4n-4})(\chi_{q}(m_{4})-m_{4})$, then $m=m_{1}m_{2}$ is right negative and hence $m$ is not dominant. This contradicts our assumption. Therefore $m_{2}\in \chi_{q}(4_{0}4_{4} \cdots4_{4n-4})m_{4}$.

Suppose that $m_{1}\in \mathscr{M}(L(m_{1}'))\cap \mathscr{M}(\chi_{q}(4_{4}\cdots 4_{4n-4})(\chi_{q}(4_{4n})-4_{4n})m_{3})$. By the Frenkel-Mukhin algorithm for $L(m_{1}')$, $m_{1}$ must have the factor $4_{4n+4}^{-1}$. But by the Frenkel-Mukhin algorithm and the fact that $m_{2}\in \chi_{q}(4_{0}4_{4} \cdots
4_{4n-4})m_{4}$, $m_{2}$ does not have the factor $4_{4n+4}$. Therefore $m_{1}m_{2}$ is not dominant. Hence $m_{1}\in \chi_q(4_{4}\cdots 4_{4n-4})4_{4n}m_{3}$. It follows that $m_{1}=m_{1}'$.

By the Frenkel-Mukhin algorithm and the fact that $m_{2}\in \chi_{q}(4_{0}4_{4} \cdots4_{4n-4})m_{4}$, $m_{2}$ must be one of the following monomials,

\begin{align*}
& v_{1} =m_{2}'A_{4, 4n-2}^{-1}=4_{0}4_{4}\cdots 4_{4n-8}4_{4n}^{-1}3_{4n-2}m_{4},\\
& v_{2} =m_{2}'A_{4, 4n-2}^{-1}A_{4, 4n-6}^{-1}=4_{0}4_{4}\cdots 4_{4n-12}4_{4n-4}^{-1}4_{4n}^{-1}3_{4n-6}3_{4n-2}m_{4},\\
& \cdots\\
& v_{n} =m_{2}'A_{4, 4n-2}^{-1}A_{4, 4n-6}^{-1} \cdots A_{4, 2}^{-1}=4_{4}^{-1} \cdots 4_{4n-4}^{-1}4_{4n}^{-1}3_{2}\cdots 3_{4n-6}3_{4n-2}m_{4}.
\end{align*}

It follows that the dominant monomials in $\chi_{q}(\widetilde{T}_{n,m-1,l,k}^{(s+4)}) \chi_{q}(\widetilde{T}_{n,m,l,k}^{(s)})$ are

\begin{eqnarray*}
\begin{split}
&M=m_{1}'m_{2}', \ M_{1}=v_{1}m_{1}'=MA_{4, 4n-2}^{-1}, \ M_{2}=v_{2}m_{1}'=M\prod_{i=0}^{1}A_{4, 4n-4i-2}^{-1},\
\ldots, \\
&M_{n-1}=v_{n-1}m_{1}'=M\prod_{i=0}^{n-2}A_{4, 4n-4i-2}^{-1}, \ M_{n}=v_{n}m_{1}'=M\prod_{i=0}^{n-1}A_{4, 4n-4i-2}^{-1}.
\end{split}
\end{eqnarray*}
\end{proof}

\subsection{Proof of Theorem $\ref{M-system of type F4}$}\label{Proof of M system of type F4}

By Lemma $\ref{lemma1}$, the dominant monomials in the $q$-characters of the right hand side and of the left hand side of every equation in Theorem $\ref{M-system of type F4}$ are the same. Therefore the theorem is true.

\section{Proof of Theorem $\ref{irreducible}$}  \label{proof irreducible}

In this section, we prove Theorem \ref{irreducible}.

By Lemma $\ref{lemma1}$, we have known that the modules in the second summand on the right hand side of every equation in Theorem $\ref{M-system of type F4}$ are special, and hence they are simple. Therefore in order to prove Theorem \ref{irreducible}, we only need to show that the modules in the first summand on the right hand side of every equation in Theorem $\ref{M-system of type F4}$ are simple.

\begin{table}[!htbp] \resizebox{.75\width}{.75\height}{
\begin{tabular}{|c|c|c|c|}
\hline %
Equations & Non-highest dominant monomial & $n_{r}$ & Relation \\
\hline %
(\ref{eqn 1})&$\substack{M_{r}, \quad r=1, 2, \ldots, k-1}$
&$\substack{M_{r}A_{1, aq^{-2l-2r+1}}^{-1}A_{2, aq^{-2l-2r+2}}^{-1},\\ r=1, 2, \ldots, k-1}$
&$\substack{n_{r}\in \chi_{q}(M_{r}),\\ n_{r}\not \in \chi_{q}(\mathcal T_{0, 0, l-1, k }^{(-2)})\chi_{q}(\mathcal T_{0, 0, l,k }^{(-2)}), \\ r=1, 2, \ldots, k-1}$\\
\hline%
\ref{eqn 2})&$\substack{M_{r}, \quad r=1, 2, \ldots, n-1}$
&$\substack{M_{r}A_{4, aq^{s+4n-4r+2}}^{-1}A_{3, aq^{s+4n-4r+4}}^{-1}, \\ r=1, 2, \ldots, n-1}$
&$\substack{n_{r}\in \chi_{q}(M_{r}),\\ n_{r}\not \in \chi_{q}(\widetilde{\mathcal{T}}_{n, m-1, 0, 0}^{(s+4)})\chi_{q}(\widetilde{\mathcal{T}}_{n, m, 0, 0}^{(s)}), \\r=1, 2, \ldots, n-1}$\\
\hline%
(\ref{eqn 3a}, \ref{eqn 3b})&$\substack{M_{r}, \quad r=1, 2, \ldots, n-1}$
&$\substack{M_{1}A^{-1}_{4,aq^{s+4n-2}}A^{-1}_{3,aq^{s+4n}}A^{-1}_{2,aq^{s+4n+2}}\\M_{r}A^{-1}_{4, aq^{s+4n-4r+2}}A^{-1}_{3, aq^{s+4n-4r+4}}, \\ r=2, \ldots, n-1}$
&$\substack{n_{r}\in \chi_{q}(M_{r}),\\ n_{r}\not \in \chi_{q}(\widetilde{\mathcal{T}}_{n, 0, l-2, 0}^{(s+4)})\chi_{q}(\widetilde{\mathcal{T}}_{n, 0, l, 0}^{(s)}),\\ r=1, 2, \ldots, n-1}$\\
\hline%
(\ref{eqn 4a}, \ref{eqn 4b})&$\substack{M_{r}, \quad r=1, 2, \ldots, m-1}$
&$\substack{M_{r}A_{3, aq^{s+4m-4r+4}}^{-1}A_{2, aq^{s+4m-4r+6}}^{-1}, \\ r=1, 2, \ldots, m-1}$
&$\substack{n_{r}\in \chi_{q}(M_{r}),\\ n_{r}\not \in \chi_{q}(\widetilde{\mathcal{T}}_{0, m, l-2, 0}^{(s+4)})\chi_{q}(\widetilde{\mathcal{T}}_{0, m, l, 0}^{(s)}), \\ r=1, 2, \ldots, m-1}$\\
\hline%
(\ref{eqn 5})&$\substack{M_{r}, \quad r=1, 2, \ldots, n-1}$
&$\substack{M_{r}A_{4, aq^{s+4n-4r+2}}^{-1}A_{3, aq^{s+4n-4r+4}}^{-1}, \\ r=1, 2, \ldots, n-1}$
&$\substack{n_{r}\in \chi_{q}(M_{r}), \\n_{r}\not \in \chi_{q}(\widetilde{\mathcal{T}}_{n, m-1, l, 0}^{(s+4)})\chi_{q}(\widetilde{\mathcal{T}}_{n, m, l, 0}^{(s)}), \\r=1, 2, \ldots, n-1}$\\
\hline%
(\ref{eqn 6a})& $\substack{M_{r}, \quad r=1, 2, \ldots, n-1}$
&$\substack{M_{1}A_{4, aq^{s+4n-2}}^{-1}A_{3, aq^{s+4n}}^{-1}A_{2, aq^{s+4n+2}}^{-1}A_{1, aq^{s+4n+3}}^{-1},\\
M_{r}A_{4, aq^{s+4n-4r+2}}^{-1}A_{3, aq^{s+4n-4r+4}}^{-1},\\r=2, \ldots, n-1}$
&$\substack{n_{r}\in \chi_{q}(M_{r}), \\n_{r}\not \in \chi_{q}(\widetilde{\mathcal{T}}_{n, 0, 0, k-2}^{(s+4)})\chi_{q}(\widetilde{\mathcal{T}}_{n, 0, 0, k}^{(s)}), \\r=1, 2, \ldots, n-1}$\\
\hline%
(\ref{eqn 7a})&$\substack{M_{r}, \quad r=1, 2, \ldots, m-1}$
&$\substack{M_{1}A_{3, aq^{s+4m}}^{-1}A_{2, aq^{s+4m+2}}^{-1}A_{1, aq^{s+4m+3}}^{-1},\\
M_{r}A_{3, aq^{s+4m-4r+4}}^{-1}A_{2, aq^{s+4m-4r+6}}^{-1},\\r=2, \ldots, m-1}$
&$\substack{n_{r}\in \chi_{q}(M_{r}), \\n_{r}\not \in \chi_{q}(\widetilde{\mathcal{T}}_{0, m, 0, k-2}^{(s+4)})\chi_{q}(\widetilde{\mathcal{T}}_{0, m, 0, k}^{(s)}),\\ r=1, 2, \ldots, m-1}$\\
\hline%
(\ref{eqn 8}, \ref{eqn 9})&$\substack{M_{r}, \quad r=1, 2, \ldots, n-1}$
&$\substack{M_{1}A_{4, aq^{s+4n-2}}^{-1}A_{3, aq^{s+4n}}^{-1}A_{2, aq^{s+4n+2}}^{-1},\\
 M_{r}A_{4, aq^{s+4n-4r+2}}^{-1}A_{3, aq^{s+4n-4r+4}}^{-1}, \\ r= 2, \ldots, n-1}$
&$\substack{n_{r}\in \chi_{q}(M_{r}), \\n_{r}\not \in \chi_{q}(\widetilde{\mathcal{T}}_{n, 0, 0, k-1}^{(s+4)})\chi_{q}(\widetilde{\mathcal{T}}_{n, 0, l, k}^{(s)}), \\r=1, 2, \ldots, n-1 }$\\
\hline%
(\ref{eqn 10}, \ref{eqn 11})&$\substack{M_{r}, \quad r=1, 2, \ldots, m-1}$
&$\substack{M_{r}A_{3, aq^{s+4m-4r+4}}^{-1}A_{2, aq^{s+4m-4r+6}}^{-1},\\ r=1, 2, \ldots, m-1}$
&$\substack{n_{r}\in \chi_{q}(M_{r}), \\ n_{r}\not \in \chi_{q}(\widetilde{\mathcal{T}}_{0, m, 0, k-1}^{(s+4)})\chi_{q}(\widetilde{\mathcal{T}}_{0, m, 1, k}^{(s)}), \\ r=1, 2, \ldots, m-1}$\\
\hline%
(\ref{eqn 12})&$\substack{M_{r}, \quad r=1, 2, \ldots, n-1}$
&$\substack{M_{r}A_{4, aq^{s+4n-4r+2}}^{-1}A_{3, aq^{s+4n-4r+4}}^{-1},\\ r=1, 2, \ldots, n-1}$
&$\substack{n_{r}\in \chi_{q}(M_{r}), \\n_{r}\not \in \chi_{q}(\widetilde{\mathcal{T}}_{n, m-1, 0, k}^{(s+4)})\chi_{q}(\widetilde{\mathcal{T}}_{n, m, 0, k}^{(s)}), \\ r=1, 2, \ldots, n-1}$\\
\hline%
(\ref{eqn 13})&$\substack{M_{r}, \quad r=1, 2, \ldots, n-1}$
&$\substack{M_{r}A_{4, aq^{s+4n-4r+2}}^{-1}A_{3, aq^{s+4n-4r+4}}^{-1},\\ r=1, 2, \ldots, n-1}$
&$\substack{n_{r}\in \chi_{q}(M_{r}), \\n_{r}\not \in \chi_{q}(\widetilde{\mathcal{T}}_{n, m-1, l, k}^{(s+4)})\chi_{q}(\widetilde{\mathcal{T}}_{n, m, l, k}^{(s)}), \\ r=1, 2, \ldots, n-1}$\\
\hline%
(\ref{eqn 14}, \ref{eqn 15}, \ref{eqn 16}) & $\substack{M_{r}, \quad r=1, 2, \ldots, l-1}$
&$\substack{M_{r} A_{2, aq^{s+2l+4-2r}}^{-1} A_{1, aq^{s+2l+5-2r}}^{-1},\\ r=1, 2, \ldots, l-1}$
&$\substack{n_{r}\in \chi_{q}(M_{r}), \\n_{r}\not \in \chi_{q}(\widetilde{\mathcal{P}}^{(s+2)}_{0,0,l,k-1})
\chi_{q}(\widetilde{\mathcal{P}}^{(s)}_{0,0,l,k}), \\ r=1, 2, \ldots, l-1}$\\
\hline%
(\ref{eqn 17}, \ref{eqn 18}) & $\substack{M_{r}, \quad r=1, 2, \ldots, m-1}$
&$\substack{M_{1}A_{3, aq^{s+4m}}^{-1}A_{2, aq^{s+4m+2}}^{-1}A_{1, aq^{s+4m+3}}^{-1},\\
M_{r}A_{3, aq^{s+4m-4r+4}}^{-1}A_{2, aq^{s+4m-4r+6}}^{-1},\\r=2, \ldots, m-1}$
&$\substack{n_{r}\in \chi_{q}(M_{r}), \\n_{r}\not \in \chi_{q}(\widetilde{\mathcal{P}}^{(s+4)}_{0,m,0,k-2})
\chi_{q}(\widetilde{\mathcal{P}}^{(s)}_{0,m,0,k}),\\ r=1, 2, \ldots, m-1}$\\
\hline
(\ref{eqn 19})& $\substack{M_{r}, \quad r=1, 2, \ldots, n-1}$
& $\substack{M_{1}A_{4, aq^{s+4n-2}}^{-1}A_{3, aq^{s+4n}}^{-1}A_{2, aq^{s+4n+2}}^{-1}A_{1, aq^{s+4n+3}}^{-1},\\
M_{r}A_{4, aq^{s+4n-4r+2}}^{-1}A_{3, aq^{s+4n-4r+4}}^{-1},\\r=2, \ldots, n-1}$
& $\substack{n_{r}\in \chi_{q}(M_{r}), \\n_{r}\not \in \chi_{q}(\widetilde{\mathcal{T}}_{n, 0, 0, k-2}^{(s+4)})\chi_{q}(\widetilde{\mathcal{T}}_{n, 0, 0, k}^{(s)}), \\r=1, 2, \ldots, n-1}$\\
\hline%
\end{tabular}}
\caption{Irreducible in Theorem \ref{M-system of type F4}.}
\label{irreducible in the system of type F_{4}}
\end{table}

\subsection{Proof of Theorem \ref{irreducible}}

We will prove the case of $\chi_{q}(\mathcal{\widetilde{T}}_{n-1,m,l,k}^{(s+4)}) \chi_{q}(\mathcal{\widetilde{T}}_{n+1,m-1,l,k}^{(s)})$. The other cases are similar. By definition, we have

\begin{gather}
\begin{align*}
\widetilde{T}_{n, m-1, l, k}^{(s+4)}=&(4_{s+4}4_{4+8}\cdots 4_{s+4n})(3_{s+4n+6}3_{s+4n+10} \cdots 3_{s+4n+4m-2})
(2_{s+4n+4m+3} 2_{s+4n+4m+5}\cdots \\
&2_{s+4n+4m+2l+1})(1_{s+4n+4m+2l+4} 1_{s+4n+4m+2l+6} \cdots 1_{s+4n+4m+2l+2k+2}),\\
\widetilde{T}_{n, m, l, k}^{(s)}=&(4_{s}4_{s+4} \cdots4_{s+4n-4})(3_{s+4n+2}3_{s+4n+6}\cdots3_{s+4n+4m-2})(2_{s+4n+4m+3}2_{s+4n+4m+5}\cdots \\
&2_{s+4n+4m+2l+1})(1_{s+4n+4m+2l+4} 1_{s+4n+4m+2l+6} \cdots 1_{s+4n+4m+2l+2k+2}).
\end{align*}
\end{gather}

By Lemma \ref{lemma1}, the dominant monomials of $\chi_{q}(\mathcal{\widetilde{T}}_{n-1,m,l,k}^{(s+4)}) \chi_{q}(\mathcal{\widetilde{T}}_{n+1,m-1,l,k}^{(s)})$ are

\begin{align*}
M_{r}=M \prod _{i=0}^{r-1}A^{-1}_{4,aq^{s+4n-2-4i}}, 0 \leq r \leq n-1, \text { where } M=\widetilde{T}_{n-1,m,l,k}^{(s+4)}\widetilde{T}_{n+1,m-1,l,k}^{(s)}.
\end{align*}

We need to show that $\chi_q(M_{r})\not \subseteq \chi_q(M)$ for $1 \leq r \leq n-1$. We will prove the case of $r=1$, the other cases are similar.

\begin{align*}
M_{1}=&MA_{4, aq^{s+4n-2}}^{-1}=M4_{s+4n-4}^{-1}4_{s+4n}^{-1}3_{s+4n-2}.
\end{align*}

By $U_{q^{2}}(\widehat {\mathfrak sl}_{2})$ argument, it is clear that $n_{1}=M_{1}A_{4, aq^{s+4n-2}}^{-1}A_{3, aq^{s+4n}}^{-1}$ is in $\chi_{q}(M_{1})$.

If $n_{1}$ is in $\chi_{q}(\widetilde{T}_{n, m-1, l, k}^{(s+4)})\chi_{q}(\widetilde{T}_{n, m, l, k}^{(s)})$, then $\widetilde{T}_{n, m-1, l, k}^{(s+4)}A_{4, aq^{s+4n-2}}^{-1}A_{3, aq^{s+4n}}^{-1}$ is in $\chi_{q}(\widetilde{T}_{n, m-1, l, k}^{(s+4)})$ which is impossible by the Frenkel-Mukhin algorithm for $\widetilde{T}_{n, m-1, l, k}^{(s+4)}$. Hence $\chi_q(M_{1})\not \subseteq \chi_q(M)$.

\section{Conjectural equations satisfied by the $q$-characters of other minimal affinizations in type $F_4$} \label{conjectural equations for F_4}

In this section, we give some conjectural equations satisfied by the $q$-characters of the minimal affinizations in type $F_4$ which are not in Theorem \ref{M-system of type F4} and Theorem \ref{The dual M-system}. In order to study equations satisfied by $q$-characters, we introduce the concept of dominant monomial graphs for a tensor product of simple $U_q \widehat{\mathfrak{g}}$-modules.

\subsection{Conjecture about the minimal affinizations which are not in Theorem \ref{M-system of type F4} and Theorem \ref{The dual M-system}}

Let $l, m, n \in \mathbb{Z}_{\geq 1}$, $k\in \mathbb{Z}_{\geq 3}$, $s\in \mathbb{Z}$. We define
\begin{align*}
\widetilde{S}_{0,m,0,k}^{(s)}&=\widetilde{T}_{0,0,2m,k}^{(s)}\widetilde{T}_{m,0,0,k-2}^{(s+4)},\\
\widetilde{S}_{n,0,l,k}^{(s)}&=\left\{
                                              \begin{array}{ll}
                                              \widetilde{T}_{0,n,1,k}^{(s)}\widetilde{T}_{0,0,0,k-1}^{(s+4n+4)}, & \hbox{$l=1,$} \\
                                              \widetilde{T}_{0,n,l,k}^{(s)}\widetilde{T}_{0,0,l-2,k}^{(s+4n+4)}, & \hbox{$l\geq 2,$}
                                              \end{array}
                                              \right.                    \\
\widetilde{S}_{0,m,l,k}^{(s)}&=\left\{
                                              \begin{array}{ll}
                                              \widetilde{T}_{0,0,1+2m,k}^{(s)}\widetilde{T}_{m,0,0,k-1}^{(s+4)}, & \hbox{$l=1,$} \\
                                              \widetilde{T}_{0,0,l+2m,k}^{(s)}\widetilde{T}_{m,0,l-2,k}^{(s+4)}, & \hbox{$l\geq 2,$}
                                              \end{array}
                                              \right.              \\
\widetilde{S}_{n,m,0,k}^{(s)}&=\widetilde{T}_{0,n+m,0,k}^{(s)}\widetilde{T}_{0,m-1,0,k}^{(s+4n+4)},\\
\widetilde{S}_{n,m,l,k}^{(s)}&=\widetilde{T}_{0,n+m,l,k}^{(s)}\widetilde{T}_{0,m-1,l,k}^{(s+4n+4)}.
\end{align*}

We use $\widetilde{\mathcal{S}}^{(s)}_{n,m,l,k}$ to denote the simple $U_q \widehat{\mathfrak{g}}$-module with the highest weight monomial $\widetilde{S}^{(s)}_{n,m,l,k}$.

\begin{conjecture} \label{the conjecture 1}
For $s \in \mathbb{Z}$, $n, m, l \in \mathbb{Z}_{\geq 1}$, $k\in \mathbb{Z}_{\geq 3}$, we have the following equations in $\rep(U_q \widehat{\mathfrak{g}})$.

\begin{align}\label{eqn 49}
[\mathcal{\widetilde{T}}_{0,m,0,k-2}^{(s+4)}][\mathcal{\widetilde{T}}_{0,m,0,k}^{(s)}]&=[\mathcal{\widetilde{T}}_{0,m-1,0,k}^{(s+4)}][\mathcal{\widetilde{T}}_{0,m+1,0,k-2}^{(s)}]+[\mathcal{\widetilde{S}}_{0,m,0,k}^{(s)}],
\end{align}

\begin{align}\label{eqn 54}
[\mathcal{\widetilde{T}}_{n,m-1,0,k}^{(s+4)}][\mathcal{\widetilde{T}}_{n,m,0,k}^{(s)}]&=[\mathcal{\widetilde{T}}_{n-1,m,0,k}^{(s+4)}][\mathcal{\widetilde{T}}_{n+1,m-1,0,k}^{(s)}]+[\mathcal{\widetilde{S}}_{n,m,0,k}^{(s)}],
\end{align}

\begin{align}\label{eqn 55}
[\mathcal{\widetilde{T}}_{n,m-1,l,k}^{(s+4)}][\mathcal{\widetilde{T}}_{n,m,l,k}^{(s)}]&=[\mathcal{\widetilde{T}}_{n-1,m,l,k}^{(s+4)}][\mathcal{\widetilde{T}}_{n+1,m-1,l,k}^{(s)}]+[\mathcal{\widetilde{S}}_{n,m,l,k}^{(s)}],
\end{align}

\begin{align}\label{eqn 56a}
[\mathcal{\widetilde{T}}^{(s+4)}_{n,0,0,k-1}][\mathcal{\widetilde{T}}^{(s)}_{n,0,1,k}]&=[\mathcal{\widetilde{T}}^{(s+4)}_{n-1,0,1,k}][\mathcal{\widetilde{T}}^{(s)}_{n+1,0,0,k-1}]+[\mathcal{\widetilde{S}}^{(s)}_{n,0,1,k}], \quad  l=1,
\end{align}

\begin{align}\label{eqn 56b}
[\mathcal{\widetilde{T}}^{(s+4)}_{n,0,l-2,k}][\mathcal{\widetilde{T}}^{(s)}_{n,0,l,k}]&=[\mathcal{\widetilde{T}}^{(s+4)}_{n-1,0,l,k}][\mathcal{\widetilde{T}}^{(s)}_{n+1,0,l-2,k}]+[\mathcal{\widetilde{S}}^{(s)}_{n,0,l,k}], \quad l\geq2,
\end{align}

\begin{align}\label{eqn 57a}
[\mathcal{\widetilde{T}}_{0,m,0,k-1}^{(s+4)}][\mathcal{\widetilde{T}}_{0,m,1,k}^{(s)}]=[\mathcal{\widetilde{T}}_{0,m-1,1,k}^{(s+4)}][\mathcal{\widetilde{T}}_{0,m+1,0,k-1}^{(s)}]+[\mathcal{\widetilde{S}}^{(s)}_{0,m,1,k}],
\quad l=1,
\end{align}

\begin{align}\label{eqn 57b}
[\mathcal{\widetilde{T}}_{0,m,l-2,k}^{(s+4)}][\mathcal{\widetilde{T}}_{0,m,l,k}^{(s)}]=[\mathcal{\widetilde{T}}_{0,m-1,l,k}^{(s+4)}][\mathcal{\widetilde{T}}_{0,m+1,l-2,k}^{(s)}]+[\mathcal{\widetilde{S}}^{(s)}_{0,m,l,k}],
\quad l\geq2.
\end{align}
\end{conjecture}

\begin{example}\label{example 1}
The following are some equations of Equation (\ref{eqn 49}) in Conjecture \ref{the conjecture 1}.
{\tiny \begin{align}
&[1_{0}3_{-6}][1_{-4}1_{-2}1_{0}3_{-10}]=[1_{-4}1_{-2}1_{0}][1_{0}3_{-10}3_{-6}]+[1_{-4}1_{-2}1^{2}_{0}2_{-9}2_{-7}4_{-8}],
\label{eqn 70}\\
&[1_{0}3_{-10}3_{-6}][1_{-4}1_{-2}1_{0}3_{-14}3_{-10}]=[1_{-4}1_{-2}1_{0}3_{-10}][1_{0}3_{-14}3_{-10}3_{-6}]+[1_{-4}1_{-2}1^{2}_{0}2_{-13}2_{-11}2_{-9}2_{-7}4_{-12}4_{-8}],
\label{eqn 71}\\
&[1_{-4}1_{-2}1_{0}3_{-10}][1_{-8}1_{-6}1_{-4}1_{-2}1_{0}3_{-14}]=[1_{-8}1_{-6}1_{-4}1_{-2}1_{0}][1_{-4}1_{-2}1_{0}3_{-14}3_{-10}]+[1_{-8}1_{-6}1^{2}_{-4}1^{2}_{-2}1^{2}_{0}2_{-13}2_{-11}4_{-12}].
\label{eqn 72}
\end{align}}
\end{example}

\begin{example}\label{example 2}
The following are some equations of Equation (\ref{eqn 56a}) and  Equation (\ref{eqn 56b}) in Conjecture \ref{the conjecture 1}.

{\tiny
\begin{align}
[1_{-2}1_{0}4_{-10}][1_{-4}1_{-2}1_{0}2_{-7}4_{-14}] & =[1_{-4}1_{-2}1_{0}2_{-7}][1_{-2}1_{0}4_{-14}4_{-10}]+[1_{-4}1^{2}_{-2}
1^{2}_{0}2_{-7}3_{-12}], \label{eqn 73}\\
[1_{-2}1_{0}4_{-14}4_{-10}][1_{-4}1_{-2}1_{0}2_{-7}4_{-18}4_{-14}] & =[1_{-4}1_{-2}1_{0}2_{-7}4_{-14}][1_{-2}1_{0}4_{-18}4_{-14}
4_{-10}]+[1_{-4}1^{2}_{-2}1^{2}_{0}2_{-7}3_{-16}3_{-12}], \label{eqn 74}\\
[1_{-4}1_{-2}1_{0}4_{-12}][1_{-4}1_{-2}1_{0}2_{-9}2_{-7}4_{-16}] & =[1_{-4}1_{-2}1_{0}2_{-9}2_{-7}][1_{-4}1_{-2}1_{0}4_{-16}
4_{-12}]+[1^{2}_{-4}1^{2}_{-2}1^{2}_{0}2_{-9}2_{-7}3_{-14}], \label{eqn 75}\\
[1_{-4}1_{-2}1_{0}4_{-16}4_{-12}][1_{-4}1_{-2}1_{0}2_{-9}2_{-7}4_{-20}4_{-16}] & =[1_{-4}1_{-2}1_{0}2_{-9}2_{-7}4_{-16}][1_{-4}
1_{-2}1_{0}4_{-20}4_{-16}4_{-12}] \nonumber \\
& \quad +[1^{2}_{-4}1^{2}_{-2}1^{2}_{0}2_{-9}2_{-7}3_{-18}3_{-14}], \label{eqn 76}\\
[1_{-4}1_{-2}1_{0}2_{-7}4_{-14}][1_{-4}1_{-2}1_{0}2_{-11}2_{-9}2_{-7}4_{-18}] & =[1_{-4}1_{-2}1_{0}2_{-11}2_{-9}2_{-7}][1_{-4}
1_{-2}1_{0}2_{-7}4_{-18}4_{-14}]\nonumber\\
& \quad +[1^{2}_{-4}1^{2}_{-2}1^{2}_{0}2_{-11}2_{-9}2^{2}_{-7}3_{-16}], \label{eqn 77}\\
[1_{-4}1_{-2}1_{0}2_{-7}4_{-18}4_{-14}][1_{-4}1_{-2}1_{0}2_{-11}2_{-9}2_{-7}4_{-22}4_{-18}] & =[1_{-4}1_{-2}1_{0}2_{-11}2_{-9}
2_{-7}4_{-18}][1_{-4}1_{-2}1_{0}2_{-7}4_{-22}4_{-18}4_{-14}]\nonumber \\
& \quad +[1^{2}_{-4}1^{2}_{-2} 1^{2}_{0}2_{-11}2_{-9}2^{2}_{-7}3_{-20}3_{-16}],\label{eqn 78}\\
[1_{-4}1_{-2}1_{0}2_{-9}2_{-7}4_{-16}][1_{-4}1_{-2}1_{0}2_{-13}2_{-11}2_{-9}2_{-7}4_{-20}] & =[1_{-4}1_{-2}1_{0}2_{-13}2_{-11}
2_{-9}2_{-7}][1_{-4}1_{-2}1_{0}2_{-9}2_{-7}4_{-16}]\nonumber\\
& \quad +[1^{2}_{-4}1^{2}_{-2}1^{2}_{0}2_{-13}2_{-11}2^{2}_{-9}2^{2}_{-7}3_{-18}].\label{eqn 79}
\end{align}
}

\end{example}

\subsection{Dominant monomial graphs}

In order to study equations satisfied by $q$-characters, we introduce dominant monomial graphs for a tensor product of simple $U_q \widehat{\mathfrak{g}}$-modules.

\begin{definition}\label{Dominant monomial graphs}
Let $\mathcal{T} = \mathcal{T}_1 \otimes \cdots \otimes \mathcal{T}_k$ be a tensor product of simple $U_q \widehat{\mathfrak{g}}$-modules. We define the dominant monomial graph $G(\mathcal{T})$ for $\mathcal{T}$ as follows. The vertices of $G(\mathcal{T})$ are dominant monomials in $\chi_q(\mathcal{T}) = \chi_q(\mathcal{T}_1) \cdots \chi_q(\mathcal{T}_k)$. For two vertices $v_1, v_2$ in $G(\mathcal{T})$, there is an arrow from $v_1$ to $v_2$ if and only if $v_2 < v_1$.
\end{definition}

Let $G$ be a dominant monomial graph. Suppose that $a, b$ are two vertices in $G$ and $b<a$. Then $b = m a$ for some $m \in \mathcal{Q}^{-}$. We draw
$$\begin{xy}
(50,0)*+{a}="1";(70,0)*+{b}="2";%
{\ar@{->}|{m} "1";"2"};%
\end{xy}$$
when we draw the graph $G$.

In the following, we draw the dominant monomial graphs for the modules in the equivalence classes on the left hand side of the equations in Examples \ref{example 1} and \ref{example 2}. Figure \ref{The Graph of dominant 1} -- Figure \ref{The Graph of dominant 11} correspond Equation (\ref{eqn 70}) -- Equation (\ref{eqn 79}) respectively.

In all examples of dominant monomial graphs, we find that every graph can be divided into two parts. The vertices in the first (resp. second) part of the graph are dominant monomials in the first  (resp. second) summand of the right hand side of the correpsonding equation. For example, in Figure \ref{The Graph of dominant 2}, the monomials $M$, $a_2$, $a_3$, $a_4$ (resp. $a_5$, $a_6$, $a_7$, $a_8$, $a_9$, $a_{10}$, $a_{11}$) are the dominant monomials of the first (resp. second) summand of the right hand side of Equation (\ref{eqn 71}).

These graphs are also conjectural, since we are not able to show that the Frenkel-Mukhin algorithm works for the modules which are not special. If we can show that these graphs are indeed the dominant monomial graphs for the corresponding modules, then the corresponding conjectural equations are true.

For $k\in\mathbb{Z}$, let
\begin{align*}
N^{(1)}_{k}&=A^{-1}_{4,k}A^{-1}_{3,k+2}A^{-1}_{2,k+4}A^{-1}_{1,k+5},\\
N^{(2)}_{k}&=A^{-1}_{2, k}A^{-1}_{3, k+2}A^{-1}_{2, k+4}A^{-1}_{1, k+5},\\
N^{(3)}_{k}&=A^{-1}_{4,k-2}A^{-1}_{3,k},\\
N^{(4)}_{k}&=A^{-1}_{2,k-2}A^{-1}_{3,k}A^{-1}_{4,k-2},\\
M^{(1)}_{k}&=A^{-1}_{2,k-2}A^{-1}_{3,k}A^{-1}_{4,k+2},
\end{align*}
where $A_{i,k}=A_{i,aq^{k}}$ which is defined in \ref{eqution a}.

We have the following relations
\begin{align*}
A_{2,k}N^{(2)}_{k}=A_{4,k}N^{(1)}_{k}, \quad A_{2,k+2}A_{1,k+3}N^{(1)}_{k-2}=N^{(3)}_{k},\\
N^{(3)}_{k}=A_{2,k-2}N^{(4)}_{k},  \quad   A_{4,k-2}N^{(4)}_{k}=A_{4,k+2}M^{(1)}_{k}.
\end{align*}

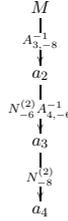
\begin{figure}[H]
\resizebox{.6\width}{.6\height}{
$$\begin{xy}
(50,0)*+{M}="1";(50,-15)*+{a_2}="2";%
(50,-30)*+{a_3}="3";(50,-45)*+{a_4}="4";%
{\ar@{->}|{A^{-1}_{3,-8}} "1";"2"};{\ar@{->}|{N^{(2)}_{-6}A^{-1}_{4,-6}} "2";"3"};
{\ar@{->}|{N^{(2)}_{-8}} "3";"4"};%
\end{xy}$$}
\caption{The dominant monomial graph for $L(m_{1})\bigotimes L(m_{2})$ ($M=m_{1}m_{2}$ where $m_{1}=1_{0}3_{-6}, m_{2}=1_{-4}1_{-2}1_{0}3_{-10})$}.\label{The Graph of dominant 1}
\end{figure}

\begin{figure}[H]
\resizebox{.6\width}{.6\height}{
$$\begin{xy}
(50,0)*+{M}="1";(80,0)*+{a_2}="2";(110,0)*+{a_3}="3";(140,0)*+{a_4}="4";%
(80,-20)*+{a_5}="5";(110,-20)*+{a_6}="6";(140,-20)*+{a_7}="7";(170,-40)*+{a_{10}}="10";
(200,-40)*+{a_{11}}="11";%
(110,-40)*+{a_8}="8";(140,-40)*+{a_9}="9";%
{\ar@{->}|{A^{-1}_{3,-8}} "1";"2"};{\ar@{->}|{N^{(2)}_{-6}A^{-1}_{4,-6}} "2";"3"};{\ar@{->}|{N^{(2)}_{-8}} "3";"4"};%
{\ar@{->}|{N^{(2)}_{-6}A^{-1}_{4,-6}} "5";"6"};{\ar@{->}|{N^{(2)}_{-8}} "6";"7"};{\ar@{->}|{N^{(4)}_{-8}} "7";"9"};
{\ar@{->}|{A^{-1}_{2,-6}A^{-1}_{1,-5}} "9";"10"};{\ar@{->}|{N^{(2)}_{-12}} "10";"11"};
{\ar@{->}|{A^{-1}_{3,-12}} "2";"5"};{\ar@{->}|{A^{-1}_{3,-12}} "3";"6"};{\ar@{->}|{A^{-1}_{3,-12}} "4";"7"};
{\ar@{->}|{N^{(4)}_{-8}A^{-1}_{2,-8}} "6";"8"};
{\ar@{->}|{N^{(2)}_{-8}A_{2,-8}} "8";"9"};
\end{xy}$$}
\caption{The dominant monomial graph for $L(m_{1})\bigotimes L(m_{2})$ ($M=m_{1}m_{2}$ where $m_{1}=1_{0}3_{-10}3_{-6}, m_{2}=1_{-4}1_{-2}1_{0}3_{-14}3_{-10})$.}\label{The Graph of dominant 2}
\end{figure}
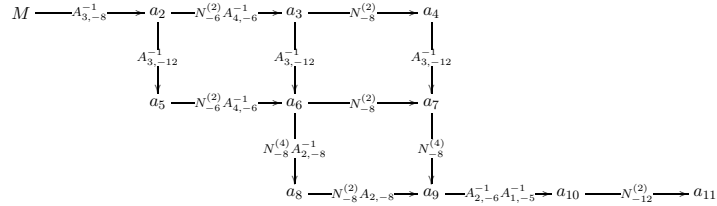

\begin{figure}[H]
\resizebox{.6\width}{.6\height}{
$$\begin{xy}
(45,0)*+{M}="1";(80,0)*+{a_2}="2";(110,0)*+{a_3}="3";(150,0)*+{a_4}="4";(170,0)*+{a_5}="5";%
(45,-20)*+{a_6}="6";(80,-20)*+{a_7}="7";(110,-20)*+{a_8}="8";(110,-40)*+{a_{9}}="9";(170,-40)*+{a_{10}}="10";%
(45,-40)*+{a_{11}}="11";(80,-40)*+{a_{12}}="12";%
(45,-60)*+{a_{13}}="13";(80,-60)*+{a_{14}}="14";%
{\ar@{->}|{A^{-1}_{3,-12}} "1";"6"};{\ar@{->}|{N^{(2)}_{-10}A^{-1}_{4,-10}} "6";"11"};{\ar@{->}|{N^{(2)}_{-12}} "11";"13"};%
{\ar@{->}|{N^{(2)}_{-6}A^{-1}_{3,-8}A^{-1}_{4,-6}} "1";"2"};{\ar@{->}|{N^{(2)}_{-8}} "2";"3"};
{\ar@{->}|{N^{(2)}_{-10}A^{-1}_{3,-12}A^{-1}_{4,-10}} "3";"4"};{\ar@{->}|{N^{(2)}_{-12}} "4";"5"};
{\ar@{->}|{A^{-1}_{3,-12}A^{-1}_{2,-10}} "2";"7"};{\ar@{->}|{N^{(1)}_{-10}} "7";"12"};
{\ar@{->}|{N^{(2)}_{-12}} "12";"14"};%
{\ar@{->}|{A^{-1}_{3,-12}A^{-1}_{2,-10}M^{(1)}_{-10}} "3";"8"};
{\ar@{->}|{M^{(1)}_{-10}} "4";"9"};{\ar@{->}|{A^{-1}_{4,-8}} "5";"10"};
{\ar@{->}|{N^{(1)}_{-10}} "8";"9"};
{\ar@{->}|{A^{-1}_{2,-8}A^{-1}_{1,-7}} "9";"10"};%
{\ar@{->}|{N^{(2)}_{-6}M^{(1)}_{-8}} "6";"7"};{\ar@{->}|{N^{(2)}_{-6}A^{-1}_{3,-8}A^{-1}_{4,-6}} "11";"12"};
{\ar@{->}|{N^{(2)}_{-6}A^{-1}_{3,-8}A^{-1}_{4,-6}}  "13";"14"};%
{\ar@{->}|{N^{(2)}_{-8}M^{(1)}_{-10}} "7";"8"};{\ar@{->}|{N^{(2)}_{-8}M^{(1)}_{-10}} "12";"9"};%
\end{xy}$$}
\caption{The dominant monomial graph for $L(m_{1})\bigotimes L(m_{2})$ ($M=m_{1}m_{2}$ where $m_{1}=1_{-4}1_{-2}1_{0}3_{-10}, m_{2}=1_{-8}1_{-6}1_{-4}1_{-2}1_{0}3_{-14}$).}\label{The Graph of dominant 3}
\end{figure}

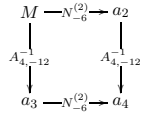
\begin{figure}[H]
\resizebox{.6\width}{.6\height}{
$$\begin{xy}
(50,0)*+{M}="1";(70,0)*+{a_2}="2";%
(50,-20)*+{a_3}="3";(70,-20)*+{a_4}="4";%
{\ar@{->}|{N^{(2)}_{-6}} "1";"2"};{\ar@{->}|{A^{-1}_{4,-12}} "1";"3"};
{\ar@{->}|{N^{(2)}_{-6}} "3";"4"};{\ar@{->}|{A^{-1}_{4,-12}} "2";"4"};%
\end{xy}$$}
\caption{The dominant monomial graph for $L(m_{1})\bigotimes L(m_{2})$ ($M=m_{1}m_{2}$ where $m_{1}=1_{-2}1_{0}4_{-10}, m_{2}=1_{-4}1_{-2}1_{0}2_{-7}4_{-14})$.}\label{The Graph of dominant 5}
\end{figure}

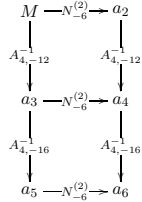
\begin{figure}[H]
\resizebox{.6\width}{.6\height}{
$$\begin{xy}
(50,0)*+{M}="1";(70,0)*+{a_2}="2";%
(50,-20)*+{a_3}="3";(70,-20)*+{a_4}="4";%
(50,-40)*+{a_5}="5";(70,-40)*+{a_6}="6";%
{\ar@{->}|{N^{(2)}_{-6}} "1";"2"};{\ar@{->}|{A^{-1}_{4,-12}} "1";"3"};%
{\ar@{->}|{N^{(2)}_{-6}} "3";"4"};{\ar@{->}|{A^{-1}_{4,-12}} "2";"4"};%
{\ar@{->}|{N^{(2)}_{-6}} "5";"6"};{\ar@{->}|{A^{-1}_{4,-16}} "3";"5"};{\ar@{->}|{A^{-1}_{4,-16}} "4";"6"};%
\end{xy}$$}
\caption{The dominant monomial graph for $L(m_{1})\bigotimes L(m_{2})$ ($M=m_{1}m_{2}$ where $m_{1}=1_{-2}1_{0}4_{-14}4_{-10}, m_{2}=1_{-4}1_{-2}1_{0}2_{-7}4_{-18}4_{-14})$.}\label{The Graph of dominant 6}
\end{figure}

\begin{figure}[H]
\resizebox{.6\width}{.6\height}{
$$\begin{xy}
(50,0)*+{M}="1";(70,0)*+{a_2}="2";(90,0)*+{a_3}="3";%
(50,-20)*+{a_4}="4";(70,-20)*+{a_5}="5";(90,-20)*+{a_6}="6";%
{\ar@{->}|{N^{(2)}_{-6}} "1";"2"};{\ar@{->}|{N^{(2)}_{-8}} "2";"3"};
{\ar@{->}|{A^{-1}_{4,-14}} "1";"4"};{\ar@{->}|{A^{-1}_{4,-14}} "2";"5"};
{\ar@{->}|{N^{(2)}_{-6}} "4";"5"};{\ar@{->}|{N^{(2)}_{-8}} "5";"6"};{\ar@{->}|{A^{-1}_{4,-14}} "3";"6"};%
\end{xy}$$}
\caption{The dominant monomial graph for $L(m_{1})\bigotimes L(m_{2})$ ($M=m_{1}m_{2}$ where $m_{1}=1_{-4}1_{-2}1_{0}4_{-12}, m_{1}=1_{-4}1_{-2}1_{0}2_{-9}2_{-7}4_{-16}$).}\label{The Graph of dominant 7}
\end{figure}
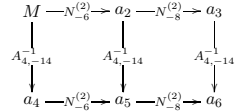

\begin{figure}[H]
\resizebox{.6\width}{.6\height}{
$$\begin{xy}
(50,0)*+{M}="1";(70,0)*+{a_2}="2";(90,0)*+{a_3}="3";%
(50,-20)*+{a_4}="4";(70,-20)*+{a_5}="5";(90,-20)*+{a_6}="6";%
(50,-40)*+{a_7}="7";(70,-40)*+{a_8}="8";(90,-40)*+{a_9}="9";%
{\ar@{->}|{N^{(2)}_{-6}} "1";"2"};{\ar@{->}|{N^{(2)}_{-8}} "2";"3"};
{\ar@{->}|{A^{-1}_{4,-14}} "1";"4"};{\ar@{->}|{A^{-1}_{4,-14}} "2";"5"};
{\ar@{->}|{N^{(2)}_{-6}} "4";"5"};{\ar@{->}|{N^{(2)}_{-8}} "5";"6"};{\ar@{->}|{A^{-1}_{4,-14}} "3";"6"};%
{\ar@{->}|{N^{(2)}_{-6}} "7";"8"};{\ar@{->}|{N^{(2)}_{-8}} "8";"9"};
{\ar@{->}|{A^{-1}_{4,-18}} "4";"7"};{\ar@{->}|{A^{-1}_{4,-18}} "5";"8"};{\ar@{->}|{A^{-1}_{4,-18}} "6";"9"};%
\end{xy}$$}
\caption{The dominant monomial graph for $L(m_{1})\bigotimes L(m_{2})$ ($M=m_{1}m_{2}$ where $m_{1}=1_{-4}1_{-2}1_{0}4_{-16}4_{-12}, m_{2}=1_{-4}1_{-2}1_{0}2_{-9}2_{-7}4_{-20}4_{-16})$.}\label{The Graph of dominant 8}
\end{figure}
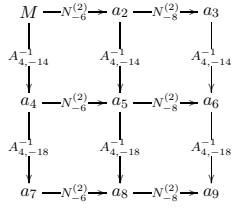

\begin{figure}[H]
\resizebox{.6\width}{.6\height}{
\begin{xy}
(0,0)*+{M}="1";(0,-30)*+{2a_2}="2";(40,-30)*+{a_3}="8";(0,-60)*+{a_4}="16";(40,-60)*+{a_5}="17";(0,-90)*+{a_6}="21";(40,-90)*+{a_7}="20";%
(10,-20)*+{a_{12}}="5";(10,-50)*+{2a_{13}}="6";(50,-50)*+{a_{14}}="13";(10,-80)*+{a_{15}}="18";(50,-80)*+{a_{16}}="19";(10,-110)*+{a_{17}}="23";(50,-110)*+{a_{18}}="22";%
(-40,-30)*+{a_{8}}="9";(-70,-30)*+{2a_{9}}="3";(-100,-30)*+{2a_{10}}="4";(-150,-30)*+{a_{11}}="15";(-30,-50)*+{a_{19}}="14";%
{\ar@{->}|{A^{-1}_{2,-8}} "2";"9"};{\ar@{->}|{N^{(3)}_{-10}} "9";"3"};{\ar@{->}|{A^{-1}_{2,-10}} "3";"4"};
{\ar@{->}|{N^{(3)}_{-14}A^{-1}_{2,-12}} "4";"15"};{\ar@{-->}|{A^{-1}_{4,-16}} "9";"14"};
{\ar@{-->}|(.3){A^{-1}_{2,-8}}"6";"14"};{\ar@{->}|{N^{(2)}_{-6}} "1";"2"};{\ar@{->}|{N^{(2)}_{-8}} "8";"17"};
{\ar@{->}|{N^{(2)}_{-10}} "17";"20"};{\ar@{->}|(.6){N^{(2)}_{-6}} "2";"8"};{\ar@{->}|(.6){N_{-6}}"16";"17"};
{\ar@{->} |(.6){N^{(2)}_{-6}}"21";"20"};%
{\ar@{-->}|{N^{(2)}_{-6}} "5";"6"};{\ar@{-->}|{N^{(2)}_{-10}} "18";"23"};%
{\ar@{-->}|{N^{(2)}_{-6}} "6";"13"};{\ar@{-->}|{N^{(2)}_{-6}} "18";"19"};{\ar@{->}|{N^{(2)}_{-6}} "23";"22"};%
{\ar@{->} |{A^{-1}_{4,-16}} "1";"5"};{\ar@{-->}|{A^{-1}_{4,-16}} "2";"6"};%
{\ar@{->}|{A^{-1}_{4,-16}} "8";"13"};%
{\ar@{->}|(.6){N^{(2)}_{-10}}"16";"21"};{\ar@{->} |{A^{-1}_{4,-16}} "17";"19"};%
{\ar@{-->}|{A^{-1}_{4,-16}} "16";"18"};{\ar@{->}|{N^{(2)}_{-8}} "13";"19"};%
{\ar@{->}|{A^{-1}_{4,-16}} "21";"23"};{\ar@{->}|{N^{(2)}_{-10}}"19";"22"};%
{\ar@{->}|{A^{-1}_{4,-16}} "20";"22"};{\ar@{->}|{N^{(3)}_{-10}A^{-1}_{3,-14}A^{-1}_{2,-12}A^{-1}_{2,-10}} "14";"15"};%
{\ar@{->}|{N^{(2)}_{-8}A_{2,-8}}"9";"16"};{\ar@{-->}|(.4){N^{(2)}_{-8}A_{2,-8}} "14";"18"};%
\end{xy}}
\caption{The dominant monomial graph for $L(m_{1})\bigotimes L(m_{2})$ ($M=m_{1}m_{2}$ where $m_{1}=1_{-4}1_{-2}1_{0}2_{-7}4_{-14}, m_{2}=1_{-4}1_{-2}1_{0}2_{-11}2_{-9}2_{-7}4_{-18})$.}
\label{The Graph of dominant 9}
\end{figure}
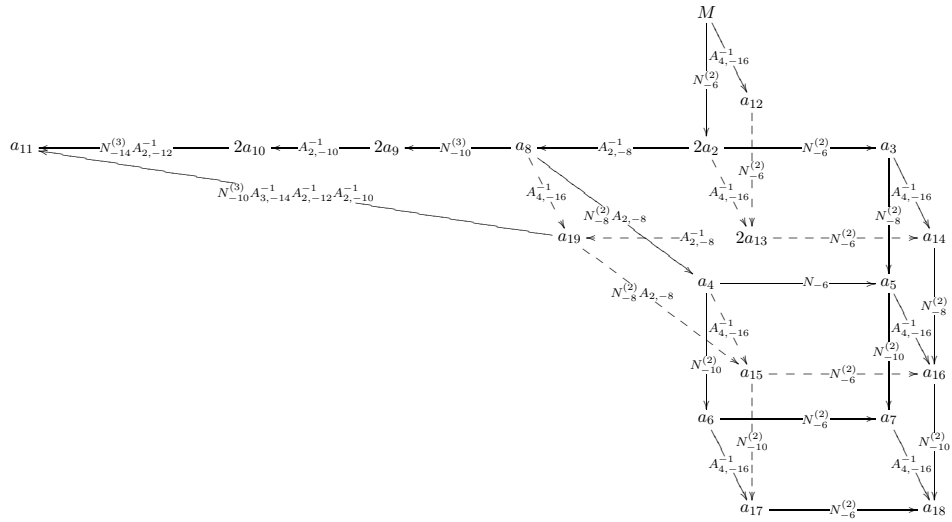

\begin{figure}[H]
\resizebox{.5\width}{.5\height}{
\begin{xy}
(0,0)*+{M}="1";(0,-30)*+{2a_2}="2";(50,-30)*+{a_3}="3";(0,-60)*+{a_4}="4";(50,-60)*+{a_5}="5";(0,-105)*+{a_6}="6";(50,-105)*+{a_7}="7";(-130,-50)*+{a_{28}}="28";%
(10,-20)*+{a_{12}}="12";(10,-50)*+{2a_{13}}="13";(60,-50)*+{a_{14}}="14";(10,-80)*+{a_{15}}="15";(60,-80)*+{a_{16}}="16";(10,-125)*+{a_{17}}="17";(60,-125)*+{a_{18}}="18";%
(-50,-30)*+{a_{8}}="8";(-80,-30)*+{2a_{9}}="9";(-105,-30)*+{2a_{10}}="10";(-140,-30)*+{2a_{11}}="11";(-40,-50)*+{a_{19}}="19";(-30,-70)*+{a_{27}}="27";%
(20,-40)*+{a_{20}}="20";(20,-70)*+{a_{21}}="21";(70,-70)*+{a_{22}}="22";(20,-100)*+{a_{23}}="23";(70,-100)*+{a_{24}}="24";(20,-145)*+{a_{25}}="25";(70,-145)*+{a_{26}}="26";%
{\ar@{->}|{N^{(2)}_{-6}} "1";"2"};{\ar@{->}|{N^{(2)}_{-10}} "4";"6"};%
{\ar@{->}|{N^{(2)}_{-8}} "3";"5"};{\ar@{->}|{N^{(2)}_{-10}} "5";"7"};%
{\ar@{-->}|{N^{(2)}_{-6}} "12";"13"};{\ar@{-->}|{N^{(2)}_{-10}} "15";"17"};%
{\ar@{->}|{N^{(2)}_{-8}} "14";"16"};{\ar@{->}|{N^{(2)}_{-10}} "16";"18"};%
{\ar@{-->}|{N^{(2)}_{-6}} "20";"21"};{\ar@{-->}|{N^{(2)}_{-10}} "23";"25"};%
{\ar@{->}|{N^{(2)}_{-8}} "22";"24"};{\ar@{->}|{N^{(2)}_{-10}} "24";"26"};%
{\ar@{->}|{A^{-1}_{4,-16}} "1";"12"};{\ar@{-->}|(.6){A^{-1}_{4,-20}} "12";"20"};
{\ar@{-->}|{A^{-1}_{4,-16}} "2";"13"};{\ar@{-->}|(.6){A^{-1}_{4,-20}} "13";"21"};
{\ar@{-->}|{A^{-1}_{4,-16}} "4";"15"};{\ar@{-->}|(.6){A^{-1}_{4,-20}} "15";"23"};
{\ar@{->}|{A^{-1}_{4,-16}} "6";"17"};{\ar@{->}|{A^{-1}_{4,-20}} "17";"25"};%
{\ar@{->}|{A^{-1}_{4,-16}} "3";"14"};{\ar@{->}|{A^{-1}_{4,-20}} "14";"22"};
{\ar@{->}|{A^{-1}_{4,-16}} "5";"16"};{\ar@{->}|{A^{-1}_{4,-20}} "16";"24"};
{\ar@{->}|{A^{-1}_{4,-16}} "7";"18"};{\ar@{->}|{A^{-1}_{4,-20}} "18";"26"};%
{\ar@{->}|(.6){A^{-1}_{4,-16}} "2";"3"};{\ar@{->}|(.6){N^{(2)}_{-6}} "4";"5"};{\ar@{->}|(.6){N^{(2)}_{-6}} "6";"7"};
{\ar@{-->}|{N^{(2)}_{-6}} "13";"14"};{\ar@{-->}|{N^{(2)}_{-6}} "15";"16"};{\ar@{->}|{N^{(2)}_{-6}} "17";"18"};
{\ar@{-->}|(.4){N^{(2)}_{-6}} "21";"22"};{\ar@{-->}|(.4){N^{(2)}_{-6}} "23";"24"};
{\ar@{->}|(.4){N^{(2)}_{-6}} "25";"26"};%
{\ar@{->}|{A^{-1}_{2,-8}} "2";"8"};{\ar@{->}|{N^{(3)}_{-10}} "8";"9"};{\ar@{->}|{A^{-1}_{4,-16}} "8";"19"};
{\ar@{->}|{A^{-1}_{2,-10}} "9";"10"};{\ar@{->}|{N^{(3)}_{-14}A^{-1}_{2,-12}} "10";"11"};
{\ar@{->}|{A^{-1}_{4,-20}} "11";"28"};%
{\ar@{->}|{A^{-1}_{4,-20}} "19";"27"};{\ar@{->}|{N^{(3)}_{-10}A^{-1}_{3,-14}A^{-1}_{2,-12}A^{-1}_{2,-10}} "27";"28"};%
{\ar@{-->}|(.4){A^{-1}_{2,-8}} "13";"19"};{\ar@{-->}|(.725){A^{-1}_{2,-8}} "21";"27"};
{\ar@{->}|{N^{(3)}_{-10}A^{-1}_{3,-14}A^{-1}_{2,-12}A^{-1}_{2,-10}} "19";"11"};%
{\ar@{->}|{N^{(2)}_{-8}A_{2,-8}} "8";"4"};{\ar@{->}|(.4){N^{(2)}_{-8}A_{2,-8}} "19";"15"};
{\ar@{->}|(.3){N^{(2)}_{-8}A_{2,-8}} "27";"23"};%
\end{xy}}
\caption{The dominant monomial graph for $L(m_{1})\bigotimes L(m_{2})$ ($M=m_{1}m_{2}$ where $m_{1}=1_{-4}1_{-2}1_{0}2_{-7}4_{-18}4_{-14}, m_{2}=1_{-4}1_{-2}1_{0}2_{-11}2_{-9}2_{-7}4_{-22}4_{-18})$.}\label{The Graph of dominant 10}
\end{figure}
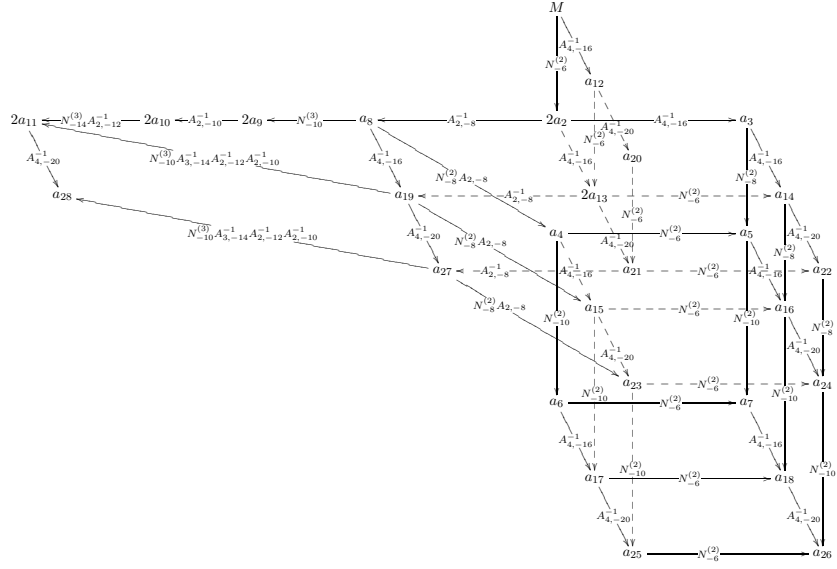

\begin{figure}[H]
\resizebox{.5\width}{.5\height}{
\begin{xy}
(0,0)*+{M}="1";(0,-30)*+{2a_2}="2";(40,-30)*+{a_3}="3";(0,-60)*+{2a_4}="4";(40,-60)*+{2a_5}="5";(80,-60)*+{a_6}="6";(0,-90)*+{a_7}="7";%
(40,-90)*+{a_{8}}="8";(80,-90)*+{a_9}="9";(0,-120)*+{a_{10}}="10";(40,-120)*+{a_{11}}="11";(80,-120)*+{a_{12}}="12";%
(-30,-30)*+{2a_{13}}="13";(-60,-30)*+{a_{14}}="14";(-90,-30)*+{2a_{15}}="15";(-120,-30)*+{2a_{16}}="16";(-150,-30)*+{a_{17}}="17";%
(-60,-60)*+{a_{18}}="18";(-90,-60)*+{2a_{19}}="19";(-120,-60)*+{2a_{20}}="20";(-150,-60)*+{a_{21}}="21";%
(-60,-90)*+{a_{22}}="22";(-90,-90)*+{2a_{23}}="23";(-120,-90)*+{2a_{24}}="24";(-150,-90)*+{a_{25}}="25";%
(-40,-90)*+{a_{26}}="26";%
(10,-20)*+{a_{27}}="27";(10,-50)*+{2a_{28}}="28";(50,-50)*+{a_{29}}="29";(10,-80)*+{2a_{30}}="30";(50,-80)*+{2a_{31}}="31";(90,-80)*+{a_{32}}="32";(10,-110)*+{a_{33}}="33";%
(50,-110)*+{a_{34}}="34";(90,-110)*+{a_{35}}="35";(10,-140)*+{a_{36}}="36";(50,-140)*+{a_{37}}="37";(90,-140)*+{a_{38}}="38";%
(-20,-50)*+{2a_{39}}="39";(-50,-50)*+{a_{40}}="40";%
(-50,-80)*+{a_{41}}="41";(-50,-110)*+{a_{42}}="42";%
(-30,-110)*+{a_{43}}="43";%
{\ar@{-->}|{A^{(-1)}_{4,-18}} "26";"43"};{\ar@{-->}|(.4){N^{(2)}_{-6}} "41";"42"};%
{\ar@{-->}|{A^{(-1)}_{4,-18}} "18";"41"};{\ar@{->}|{A^{(-1)}_{4,-18}} "22";"42"};%
{\ar@{-->}|{A^{(-1)}_{4,-18}} "13";"39"};{\ar@{-->}|{A^{(-1)}_{4,-18}} "14";"40"};%
{\ar@{->}|{N^{(2)}_{-6}} "1";"2"};{\ar@{->}|(.4){N^{(2)}_{-12}A_{1,-7}} "7";"26"};
{\ar@{-->}|(.32){N^{(2)}_{-12}A_{1,-7}} "33";"43"};{\ar@{->}|{A^{-1}_{1,-7}} "26";"10"};
{\ar@{->}|{A^{-1}_{1,-7}} "43";"36"};%
{\ar@{->}|(.6){N^{(2)}_{-8}}"3";"5"};{\ar@{->} |(.75){N^{(2)}_{-10}}"5";"8"};{\ar@{->} |(.6){N^{(2)}_{-12}}"8";"11"};%
{\ar@{->}|{N^{(2)}_{-10}} "6";"9"};{\ar@{->}|{N^{(2)}_{-12}} "9";"12"};%
{\ar@{->}|{N^{(2)}_{-6}} "2";"3"};{\ar@{->}|{N^{(2)}_{-6}} "4";"5"};{\ar@{->}|{N^{(2)}_{-8}} "5";"6"};
{\ar@{->}|{N^{(2)}_{-6}} "7";"8"};{\ar@{->}|{N^{(2)}_{-8}} "8";"9"};{\ar@{->}|{N^{(2)}_{-6}} "10";"11"};
{\ar@{->}|{N^{(2)}_{-8}} "11";"12"};%
{\ar@{->}|{A^{-1}_{2,-8}} "2";"13"};{\ar@{->}|{A^{-1}_{2,-10}} "13";"14"};{\ar@{->}|{N^{(3)}_{-12}} "14";"15"};
{\ar@{->}|{A^{-1}_{2,-12}} "15";"16"};{\ar@{->}|{N^{(3)}_{-16}A^{-1}_{2,-14}} "16";"17"};%
{\ar@{->}|{A^{-1}_{2,-10}} "4";"18"};{\ar@{->}|{N^{(3)}_{-12}} "18";"19"};
{\ar@{->}|{A^{-1}_{2,-12}} "19";"20"};{\ar@{->}|{N^{(3)}_{-16}A^{-1}_{2,-14}} "20";"21"};%
{\ar@{->}|{N^{(3)}_{-12}} "22";"23"};{\ar@{->}|{A^{-1}_{2,-12}} "23";"24"};
{\ar@{->}|{N^{(3)}_{-16}A^{-1}_{2,-14}} "24";"25"};%
{\ar@{->}|(.7){N^{(2)}_{-6}} "18";"22"};{\ar@{->}|{N^{(2)}_{-6}} "19";"23"};{\ar@{->}|{N^{(2)}_{-6}} "20";"24"};
{\ar@{->}|{N^{(2)}_{-6}} "21";"25"};{\ar@{-->}|(.5){N^{(2)}_{-6}} "27";"28"};{\ar@{-->}|(.5){N_{-8}} "29";"31"};
{\ar@{-->}|(.5){N^{(2)}_{-10}}"31";"34"};{\ar@{-->}|(.5){N^{(2)}_{-12}}"34";"37"};%
{\ar@{->}|(.4){N^{(2)}_{-10}} "32";"35"};{\ar@{->}|(.4){N^{(2)}_{-12}} "35";"38"};%
{\ar@{-->}|{N^{(2)}_{-6}} "28";"29"};{\ar@{-->}|{N^{(2)}_{-6}} "30";"31"};{\ar@{-->}|{N^{(2)}_{-8}} "31";"32"};
{\ar@{-->}|{N^{(2)}_{-6}} "33";"34"};{\ar@{-->}|{N^{(2)}_{-8}} "34";"35"};{\ar@{->}|{N^{(2)}_{-6}} "36";"37"};
{\ar@{->}|{N^{(2)}_{-8}} "37";"38"};%
{\ar@{->}|{A^{(-1)}_{4,-18}} "1";"27"};{\ar@{-->}|{A^{(-1)}_{4,-18}} "2";"28"};{\ar@{->}|{A^{(-1)}_{4,-18}} "3";"29"};%
{\ar@{-->}|(.4){A^{(-1)}_{4,-18}} "4";"30"};{\ar@{-->}|{A^{(-1)}_{4,-18}} "5";"31"};{\ar@{->}|{A^{(-1)}_{4,-18}} "6";"32"};%
{\ar@{-->}|{A^{(-1)}_{4,-18}} "7";"33"};{\ar@{-->}|{A^{(-1)}_{4,-18}} "8";"34"};{\ar@{->}|{A^{(-1)}_{4,-18}} "9";"35"};%
{\ar@{->}|{A^{(-1)}_{4,-18}} "10";"36"};{\ar@{->}|{A^{(-1)}_{4,-18}} "11";"37"};{\ar@{->}|{A^{(-1)}_{4,-18}}"12";"38"};%
{\ar@{-->}|(.4){A^{-1}_{2,-8}}"28";"39"};{\ar@{-->}|{A^{-1}_{2,-10}}"39";"40"};{\ar@{->}|(.7){N^{(2)}_{-8}A_{2,-8}}"14";"18"};
{\ar@{->}|{N^{(2)}_{-8}A_{2,-8}} "15";"19"};{\ar@{->}|{N^{(2)}_{-8}A_{2,-8}} "16";"20"};{\ar@{->}|{N^{(2)}_{-8}A_{2,-8}} "17";"21"};{\ar@{-->}|(.45){N^{(2)}_{-8}A^{-1}_{2,-8}} "40";"41"};{\ar@{-->}|(.6){A^{-1}_{2,-10}} "30";"41"};%
{\ar@{->}|(.4){N^{(2)}_{-10}A_{2,-10}} "18";"7"};{\ar@{-->}|(.6){N^{(2)}_{-10}A_{2,-10}} "41";"33"};
{\ar@{->}|{N^{(2)}_{-8}A_{2,-8}} "13";"4"};{\ar@{-->}|(.45){N^{(2)}_{-8}A_{2,-8}} "39";"30"};%
{\ar@{-->}|(.55){N^{(3)}_{-12}A^{-1}_{3,-16}A^{-1}_{2,-14}A^{-1}_{2,-12}} "40";"17"};%
{\ar@{-->}|(.55){N^{(3)}_{-12}A^{-1}_{3,-16}A^{-1}_{2,-14}A^{-1}_{2,-12}} "41";"21"};%
{\ar@{->}|{N^{(3)}_{-12}A^{-1}_{3,-16}A^{-1}_{2,-14}A^{-1}_{2,-12}} "42";"25"};
\end{xy}}
\caption{The dominant monomial graph for $L(m_{1}) \bigotimes L(m_{2})$ ($M=m_{1}m_{2}$ where $m_{1}=1_{-4}1_{-2}1_{0}2_{-9}2_{-7}4_{-16}, m_{2}=1_{-4}1_{-2}1_{0}2_{-13}2_{-11}2_{-9}2_{-7}4_{-20})$.}\label{The Graph of dominant 11}
\end{figure}
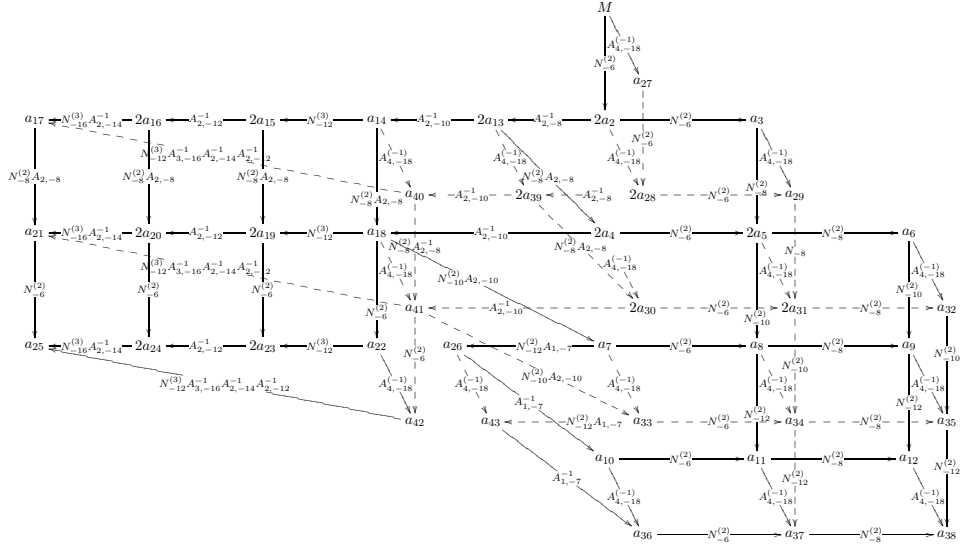

\section*{Acknowledgement}
The authors are supported by the National Natural Science Foundation of China (no. 11371177, 11401275), the Fundamental Research Funds for the Central Universities of China.


\begin{thebibliography}{99999999}

\bibitem[C95]{C95} V. Chari, \textit{Minimal affinizations of representations of quantum groups: the rank $2$ case}, Publ. Res. Inst. Math. Sci. \textbf{31} (1995), no. 5, 873--911.

\bibitem[CG11]{CG11} V. Chari, J. Greenstein, \textit{Minimal affinizations as projective objects}, J. Geom. Phys. \textbf{61} (2011), no. 3, 594--609.

\bibitem[CMY13]{CMY13} V. Chari, A. Moura, C. A. S. Young, \textit{Prime representations from a homological perspective}, Math. Z. \textbf{274} (2013), 1-2, 613--645

\bibitem[CP91]{CP91} V. Chari, A. Pressley, \textit{Quantum affine algebras}, Comm. Math. Phys. \textbf{142} (1991), no. 2, 261--283.

\bibitem[CP94]{CP94} V. Chari, A. Pressley, \textit{A guide to quantum groups}, Cambridge University Press, Cambridge (1994), xvi+651 pp.

\bibitem[CP95a]{CP95a} V. Chari, A. Pressley, \textit{Quantum affine algebras and their representations}, Representations of groups (Banff, AB, 1994), 59--78, CMS Conf. Proc., \textbf{16}, Amer. Math. Soc., Providence, RI, 1995.


\bibitem[CP96a]{CP96a} V. Chari, A. Pressley, \textit{Minimal affinizations of representations of quantum groups: the irregular case}, Lett. Math. Phys. \textbf{36} (1996), no. 3, 247--266.




\bibitem[Dri87]{Dri87} V. G. Drinfeld, \textit{Quantum groups}, Proceedings of the International Congress of Mathematicians, Vol. \textbf{1}, \textbf{2} (Berkeley, Calif., 1986), Amer. Math. Soc., Providence, RI (1987), 798--820.

\bibitem[Dri88]{Dri88} V. G. Drinfeld, \textit{A new realization of Yangians and of quantum affine algebras}, (Russian) Dokl. Akad. Nauk SSSR \textbf{296} (1987), no. 1, 13--17;
translation in Soviet Math. Dokl. \textbf{36} (1988), no. 2, 212--216.

\bibitem[FM01]{FM01} E. Frenkel, E. Mukhin, \textit{Combinatorics of $q$-characters of finite-dimensional representations of quantum affine algebras},
Comm. Math. Phys. \textbf{216} (2001), no. 1, 23--57.

\bibitem[FR98]{FR98} E. Frenkel, N. Yu. Reshetikin, \textit{The $q$-characters of representations of quantum affine algebras and deformations of W-algebras}, Recent developments in quantum affine algebras and related topics (Raleigh, NC, 1998), 163--205, Contemp. Math., \textbf{248}, Amer. Math. Soc., Providence, RI, 1999.


\bibitem[FZ02]{FZ02} S. Fomin, A. Zelevinsky, \textit{Cluster algebras I: Foundations}, J. Amer. math. Soc. \textbf{15} (2002), 497--529.



\bibitem[GG14]{GG14} J. Grabowski, S. Gratz, \textit{Cluster algebras of infinite rank, with an appendix by Michael Groechenig}, J. Lond. Math. Soc. \textbf{89} (2014), no. 2, 337--363.


\bibitem[Her06]{Her06} D. Hernandez, \textit{The Kirillov-Reshetikhin conjecture and solutions of T-systems},
J. Reine Angew. Math. \textbf{596} (2006), 63--87.

\bibitem[Her07]{Her07} D. Hernandez, \textit{On minimal affinizations of representations of quantum groups},
Comm. Math. Phys. \textbf{276} (2007), no. 1, 221--259.

\bibitem[HL10]{HL10} D. Hernandez, B. Leclerc, \textit{Cluster algebras and quantum affine algebras}, Duke Math. J. \textbf{154} (2010), no. 2, 265--341.

\bibitem[HL13]{HL13} D. Hernandez, B. Leclerc, \textit{A cluster algebra approach to $q$-characters of Kirillov-Reshetikhin modules}, arXiv:1303.0744, 1--45.




\bibitem[Jim85]{Jim85} M. Jimbo, \textit{A $q$-difference analogue of $U(\mathfrak{g})$ and the Yang-Baxter equation},
Lett. Math. Phys. \textbf{10} (1985), no. 1, 63--69.





\bibitem[Li15]{Li15}  J. R. Li, \textit{On the extended T-system of type $C_{3}$}, Journal of Algebraic Combinatorics, \textbf{41} (2015), 577-617.



\bibitem[LM13]{LM13} J. R. Li, E. Mukhin, \textit{Extended T -system of type $G_{2}$}, SIGMA Symmetry Integrability Geom. Methods Appl. \textbf{9} (2013), 054, 28pp.

\bibitem[M10]{M10} A. Moura, \textit{Restricted limits of minimal affinizations}, Pacific J. Math. \textbf{244} (2010), no. 2, 359--397.

\bibitem[MP07]{MP07} M. G. Moakes, A. N. Pressley, \textit{q-characters and minimal affinizations}, Int. Electron. J. Algebra \textbf{1} (2007), 55--97.

\bibitem[MP11]{MP11} A. Moura, F. Pereira, \textit{Graded limits of minimal affinizations and beyond: the multiplicity free case for type $E_6$}, Algebra Discrete Math. \textbf{12} (2011), no. 1,  69--115.


\bibitem[MY12a]{MY12a} E. Mukhin, C. A. S. Young, \textit{Extended T-systems}, Selecta Math.(N.S.), \textbf{18} (2012), no. 3, 591--631.

\bibitem[MY12b]{MY12b} E. Mukhin, C. A. S. Young, \textit{Path description of type B q-characters}, Adv. Math. \textbf{231} (2012), no. 2, 1119--1150.

\bibitem[MY14]{MY14} E. Mukhin, C. A. S. Young, \textit{Affinization of category $\mathcal{O}$ for quantum groups}, Trans. Amer. Math. Soc. \textbf{366} (2014), 4815--4847.

\bibitem[Nao13]{Nao13}K. Naoi, \textit{Demazure modules and graded limits of minimal affinizations}, Represent. Theory, \textbf{17} (2013), 524--556.

\bibitem[Nao14]{Nao14} K. Naoi, \textit{Graded limits of minimal affinizations in type D}, SIGMA Symmetry Integrability Geom. Methods Appl. \textbf{10} (2014), Paper 047, 20 pp.

\bibitem[QL14]{QL14} L. Qiao, J. R. Li, \textit{Cluster algebras and minimal affinizations of representations of the quantum group of type $G_2$}, arXiv:1412.3884, 2014, 1--17.

\bibitem[SS14]{SS14} S. V. Sam, \textit{Jacobi-Trudi determinants and characters of minimal affinizations}, Pacific J. Math. \textbf{272} (2014), no. 1, 237--244.

\bibitem[ZDLL15]{ZDLL15} Q. Q. Zhang, B. Duan, J. R. Li, Y. F. Luo,  \textit{M-systems and cluster algebras}, arXiv:1501.00146, 2015, 1--44.

\end{thebibliography}
\end{document}